\documentclass[10pt,UTF-8,reqno]{amsart}
\usepackage{enumerate, bbm}
\usepackage{amssymb,url,color, booktabs}
\usepackage{mathrsfs}

\usepackage{color}
\usepackage[colorlinks=true]{hyperref}
\hypersetup{
    linkcolor=blue,          % color of internal links
    citecolor=red,        % color of links to bibliography
    filecolor=blue,      % color of file links
    urlcolor=cyan
}

\usepackage{color}
\definecolor{MyDarkBlue}{cmyk}{0.8,0.3,0.8,0.4}
\definecolor{yellow}{rgb}{0.99,0.99,0.70}
\definecolor{white}{rgb}{1.0,1.0,1.0}
\definecolor{black}{rgb}{0.00,0.00,0.00}

\numberwithin{equation}{section}

\newcommand{\be}{\begin{eqnarray}}
\newcommand{\ee}{\end{eqnarray}}
\newcommand{\ce}{\begin{eqnarray*}}
\newcommand{\de}{\end{eqnarray*}}
\newtheorem{theorem}{Theorem}[section]
\newtheorem{lemma}{Lemma}[section]
\newtheorem{remark}{Remark}[section]
\newtheorem{definition}{Definition}[section]

\newtheorem{example}{Example}[section]
\newtheorem{corollary}{Corollary}[section]
\newtheorem{hypothesis}{Hypothesis}[section]

\def\[{{\Big[}}
\def\]{{\Big]}}
\def\<{{\langle}}
\def\>{{\rangle}}
\def\({{\big(}}
\def\){{\big)}}

\def\bb2{{\boldsymbol{2}}}

\def\={&\!\!=\!\!&}

\def\b1{{\mathbbm 1}}

\def\geq{\geqslant}
\def\leq{\leqslant}

\def\le{\leqslant}

\def\[{{\Big[}}
\def\]{{\Big]}}
\def\<{{\langle}}
\def\>{{\rangle}}

\def\={&\!\!=\!\!&}
\def\bt{\begin{theorem}}
\def\et{\end{theorem}}
\def\bl{\begin{lemma}}
\def\el{\end{lemma}}
\def\br{\begin{remark}}
\def\er{\end{remark}}
\def\bd{\begin{definition}}
\def\ed{\end{definition}}
\def\bc{\begin{corollary}}
\def\ec{\end{corollary}}

\def\geq{\geqslant}
\def\leq{\leqslant}

\def\le{\leqslant}

\def\<{\langle} \def\>{\rangle}

\allowdisplaybreaks

\begin{document}

\title[LDP for Multi-Scale SPDEs]
{Large Deviation Principle for Multi-Scale Fully Local Monotone Stochastic Dynamical Systems with Multiplicative Noise$^\dagger$}

\thanks{$\dagger$
 This work is supported by NSFC (No.~12171208, 12090011, 12090010) %, Postgraduate Research and Practice Innovation Program of Jiangsu Province (No.~KYCX22-2833).
 }

\maketitle
\centerline{\scshape Wei Hong,  Wei Liu\footnote{Corresponding author: weiliu@jsnu.edu.cn},  Luhan Yang }
\medskip

\centerline{  School of Mathematics and Statistics, Jiangsu Normal University, Xuzhou 221116, China}

\begin{abstract}
This paper is devoted to proving the small noise asymptotic behaviour, particularly large deviation principle, for multi-scale stochastic dynamical systems with fully local monotone coefficients driven by multiplicative noise. The main techniques are based on a combination of the weak convergence approach, the time discretization technique and the theory of pseudo-monotone operator. The main results derived in this paper have broad applicability to various multi-scale models, where the slow component could be such as stochastic porous medium equations, stochastic Cahn-Hilliard equations and stochastic 2D Liquid crystal equations.

\bigskip
\noindent
\textbf{Keywords}:
SPDE, multi-scale, large deviation principle, pseudo-monotone operator, weak convergence approach\\

\end{abstract}
\maketitle \rm

%\pagecolor{MyDarkBlue}\color{yellow}
\tableofcontents

\section{Introduction}
Multi-scale models involving slow and fast components are widely applied in various fields, including material sciences, chemistry and climate dynamics. For instance,  consider  the following simple multi-scale stochastic  differential equations (SDEs)
\begin{equation*}
d X_{t}^{\delta}=b(X_{t}^{\delta}, Y_{t / \delta}) d t+d W_{t}, X_{0}^\delta=x,
\end{equation*}
where the parameter $0<\delta\ll 1$ and $Y_{t}$ is an ergodic Markov process. For this reason, $X_{t}^{\delta}$ is referred as  the slow component, which models a phenomenon occurring at the natural time scale. On the other hand, $Y_{t / \delta}$ is referred as the fast component (with ergodic measure).
  In many cases, chemical reaction networks exhibit dynamics that occur on significantly different time scales, ranging from nanoseconds ($10^{-9}$s) to several days. For further background, we refer the reader to \cite{BR,EE,K,MHR}
and references therein. In fact, numerous physical systems possess a hierarchical structure, whereby components evolve at different rates. This implies that while some components change rapidly, others change very slowly.

Notably, multi-scale models frequently appear  in many real-world dynamical systems, in addition, considering multi-scale models with small random noise perturbations has practical significance. However,  it should be noted that studying such models has theoretical challenges due to the interaction between different time scales.
For this purpose, our aim is to investigate the small noise asymptotic behavior and establish the
large deviation principle (LDP) for a general class of multi-scale stochastic partial differential equations (SPDEs).
Specifically, we consider the following multi-scale SPDEs
\begin{eqnarray}\label{aimequation}
\left\{ \begin{aligned}
  &dX_{t}^{\varepsilon,\delta}=[\mathfrak{A}_{1}(X_{t}^{\varepsilon,\delta})+\mathfrak{f}(X_{t}^{\varepsilon,\delta}, Y_{t}^{\varepsilon,\delta})] d t+\sqrt{\varepsilon}\mathfrak{B}_{1}(X_{t}^{\varepsilon,\delta}) d W_{t}^{1}, \\
  &d Y_{t}^{\varepsilon,\delta}=\frac{1}{\delta} \mathfrak{A}_{2}(X_{t}^{\varepsilon,\delta}, Y_{t}^{\varepsilon,\delta}) d t+\frac{1}{\sqrt{\delta}} \mathfrak{B}_{2}(X_{t}^{\varepsilon,\delta}, Y_{t}^{\varepsilon,\delta}) d W_{t}^{2}, \\
  &X_{0}^{\varepsilon,\delta}=x\in \mathbb{H}_1,~~Y_{0}^{\varepsilon,\delta}=y\in \mathbb{H}_2,
\end{aligned} \right.
\end{eqnarray}
where $W_{t}^{i}$, $i=1,2$, denote  independent cylindrical Wiener processes defined on a complete filtered probability space $(\Omega, \mathscr{F},\{\mathscr{F}_{t}\}_{t\geq0}, \mathbb{P}).$ The coefficients $\mathfrak{A}_1,\mathfrak{f},\mathfrak{B}_1,\mathfrak{A}_2,\mathfrak{B}_2$ satisfy some appropriate conditions, $\varepsilon,\delta$ are small positive parameters where $\delta$ describes the ratio of the time scale between the
slow component $X_{t}^{\varepsilon,\delta}$ and the
fast component $Y_{t}^{\varepsilon,\delta}$, and $\varepsilon$ describes the intensity of the noise in slow component.

The LDP  is a powerful tool for investigating the asymptotic behavior  of probability distributions in the tails. It  has widespread applications in different fields, including information theory, thermodynamics, statistics and engineering. For a comprehensive understanding of LDP theory and its important applications, we refer readers to the classical monographs \cite{dembo2011large,varadhan1984large}. In the context of small noise LDP, also known as the Freidlin-Wentzell  LDP in the literature, for classical SDEs is established by Freidlin and Wentzell in the seminal  work \cite{FW}. Subsequently, Freidlin further studied the LDP for the small noise limit of stochastic reaction-diffusion equations in \cite{F}.
Over the past several decades, the small noise LDP has been extensively developed. For detailed investigations and applications, interested readers can refer to the  monographs \cite{A,dembo2011large,S1984,varadhan1984large}.

In the present paper, we are interested in the LDP for multi-scale models, which also has a long and rich history. For instance, Liptser \cite{L} formulated the LDP for diffusion pair with small diffusion parameters, using techniques such as the exponential tightness and the Puhalskii's theorem. In the work \cite{pukhalskii2016}, Pukhalskii obtained the LDP for the joint distribution of the slow process and of the empirical process of the fast variable. For the LDP in the context of infinite-dimensional multi-scale models, to the best of our knowledge, Wang et al.~\cite{WRD} firstly studied the LDP for stochastic reaction-diffusion equations involving slow-fast components with additive Gaussian perturbations by applying the contraction principle and certain approximations. This work was  further extended in \cite{HS} to a more general model  employing the functional occupation measure approach. Sun et al. \cite{SWXY} used the weak convergence method and the time discretization approach to establish the Freidlin-Wentzell LDP
for  stochastic Burgers equations with  fast oscillations, see also \cite{MO} for  the two-time-scale stochastic convective Brinkman-Forchheimer equations. More recently, in \cite{HLL}, the authors considered a general multi-scale system, which includes a wide class of quasilinear and semilinear SPDEs, within the locally monotone framework. However, in \cite{HLL}, only the case where the fast component is driven by additive noise was explored, owing to difficulties arising from locally monotone coefficients.

Therefore, our objective  is to establish a general and unified result for investigating the Freidlin-Wentzell LDP for a boarder class of infinite-dimensional dynamical systems with fast oscillations in which the coefficients in the slow component are assumed to satisfy a fully local monotonicity condition (see (\ref{conful}) below), while the fast component is driven by general multiplicative noise. We recall that Liu \cite{Liu2011} (see also \cite{liu2015stochastic}) studied SPDEs with additive noise and introduced a type of local monotonicity condition as follows,
\begin{equation}\label{conful}
_{\mathbb{V}^{*}}\left\langle \mathfrak{A}_{1}(u)-\mathfrak{A}_{1}(v), u-v\right\rangle_{\mathbb{V}}
\leq(C+\rho(u)+\eta(v))\|u-v\|_{\mathbb{H}}^{2},
\end{equation}
where $\rho,\eta$ are some measurable functions. Recently, R\"{o}ckner et al.~\cite{RSZ} extended the results from \cite{Liu2011} to SPDEs driven by multiplicative noise with fully local monotone coefficients (\ref{conful}). In their work, the theory of pseudo-monotone operator, initially introduced by Br\'{e}zis \cite{BH}, plays a crucial role. For further investigations  on the theory of pseudo-monotone operator, we refer to \cite{LJL,VZ1,VZ2,Z}
and references therein.

In this paper, the main techniques used to study the LDP of (\ref{aimequation}) are  a combination of the well-known weak convergence approach developed in \cite{budhiraja2000variational}, the time discretization scheme, and the theory of pseudo-monotone operator. More precisely, the LDP is derived by  demonstrating the weak convergence of solutions to the controlled slow model (see (\ref{controlequation}) below) towards its deterministic averaged limit. To achieve this, we firstly utilize the theory of pseudo-monotone operator to establish the well-posedness of the deterministic limit and prove the compactness of the rate function. Secondly, the main technical challenges in proving the weak convergence of the controlled slow model are addressed by introducing a time discretization scheme and proving that
the approximation converges at an appropriate rate. Moreover, it worth emphasizing that  in order to deal with the case where the fast component is driven by multiplicative noise, we work on improving a series of a priori estimates for solutions to the controlled system (\ref{controlequation}). These estimates, presented in Lemmas \ref{controlestimate}, \ref{condiff} and \ref{FastAuxiliaryDiffLemma}, are more general compared to previous works such as \cite{HS,HLL,MO,SWXY}.

On the other hand, we point out that  the parameters $\delta$ and $\varepsilon$ here are needed to satisfy the following condition
 $$\lim_{\varepsilon\to 0} \delta/\varepsilon = 0.$$
Notice that, depending on the order in which $\delta$ and $\varepsilon$ converge to zero, we have following three types of regimes
\begin{equation*}
\lim_{\varepsilon\rightarrow0}\delta/\varepsilon=\left\{
  \begin{array}{ll}
    0, & \hbox{Regime 1;} \\
    \gamma, & \hbox{Regime 2;} \\
    \infty, & \hbox{Regime 3.}
  \end{array}
\right.
\end{equation*}
For Regimes 2 and  3, the LDP becomes  considerably more complex, which remains the open problem in the literature, and our current approach is no longer directly applicable. In fact, in the case of infinite dimensions, the convergence of the underlying control problem is much challenging. In brief,  in Regimes 2 and 3, the invariant measure $\nu$ of the parameterized equation associated with the fast component (see (\ref{frozenequation}) in Appendix) will depend on the control function  introduced in the weak convergence approach. This dependence  makes the analysis very complicated and in particular we do not know whether $\nu$ is an invariant probability measure of a certain process or not.

Unlike previous studies, the drift term $\mathfrak{A}_1$ and the interaction term $\mathfrak{A}_2$ are not assumed to be globally Lipschitz continuous. Instead,  $\mathfrak{A}_1$ only satisfies a  fully local monotonicity condition, and  $\mathfrak{A}_2$ satisfies a monotonicity condition. Moreover, most previous investigations in the SPDE setting used specific properties of mild solutions and thus cannot cover the case of quasilinear SPDEs such as porous medium equations and $p$-Laplace equations. One novelty of this work is to show a much more general framework to handle a wide range of multi-scale examples,  including Cahn-Hilliard equations, 2D Liquid crystal equations and tamed 3D Navier-Stokes equations with fast oscillations. These examples cannot be covered by the results of  \cite{HLL} or other related works.

The rest of the paper is organized as follows. In section 2, we first sketch the
theory of the LDP and state the powerful weak convergence criterions, then we give the hypotheses on the coefficients and introduce the main results. In section 3, we derive  some necessary a priori estimates to the controlled equation. In section 4, we prove the weak convergence criterions (hence the LDP) hold for (\ref{aimequation}). In section 5, some concrete multi-scale models are given to illustrate the applicability of our main results. In section 6, some lemmas on the ergodicity of the fast process and the well-posedness of the skeleton equation are given.

Here are some conventions. Throughout this paper, $C$ denotes a generic positive
constant whose value may change from line to line. The dependence of constants on parameters if needed will be
indicated, e.g. $C_T.$
\section{Large deviation principle}
\setcounter{equation}{0}
 \setcounter{definition}{0}

In this section, we begin by introducing some basic definitions related to the LDP and present the powerful weak convergence approach used in this work. Next, we formulate the assumptions on the model considered in the current work and show the main results.

\subsection{Weak convergence analysis}
In this subsection, we explore the relationship between the LDP and the Laplace principle, and recall sufficient conditions for the Laplace principle.

Consider a family of random variables $\{X^\varepsilon\}$  defined on a probability space $(\Omega,\mathscr{F},\mathbb{P})$ and taking values in a Polish space $\mathscr{E}$. The LDP focuses on characterizing the asymptotic behavior of rare events, where the probability $\mathbb{P}(X^{\varepsilon}\in A)$ decays exponentially as $\varepsilon\to 0$. The rate of this exponential decay is determined by a function known as the ``rate function''.

\begin{definition} Let $\mathscr{E}$ be a Polish space. A function $I: \mathscr{E}\to [0,+\infty]$ is called
a rate function if $I$ is lower semi-continuous. Moreover, a rate function $I$
is called a {\it good rate function} if  the level set $\{x\in \mathscr{E}: I(x)\le
K\}$ is compact for any constant $K>0.$
\end{definition}

\begin{definition}  $\{X^\varepsilon\}$ is said to satisfy the LDP with rate function $I$ if for each Borel subset $B$ in $\mathscr{E}$,
\begin{equation*}
-\inf_{x\in B^o} I(x)\leq \liminf_{\varepsilon\to0}\varepsilon\log \mathbb P(X^\varepsilon\in B)\leq\limsup_{\varepsilon\to0}\varepsilon\log \mathbb P(X^\varepsilon\in B)\leq-\inf_{x\in\bar{B}} I(x).
\end{equation*}
where $B^o$ and $\bar{B}$ are the inner and closure of set $B$, respectively.
\end{definition}

As pointed in \cite{pukhalskii1994theory}, if $\mathscr{E}$ is a Polish space and $I$ is a good rate function, the LDP is equivalent to the following Laplace principle.

%This equivalence is essentially the result of of Varadhan's lemma \cite{varadhan1984large}  and Bryc's converse \cite{bryc1990large}, and interested readers can see \cite{dupuis2011weak,dembo2011large} for more details.

\begin{definition}\label{Laplaceprinciple}
If for each bounded continuous real-valued function $h$ defined on $\mathscr{E}$, we have
$$\lim_{\varepsilon\to 0}\varepsilon \log \mathbb{E}\left\lbrace
 \exp\left[-\frac{1}{\varepsilon} h(X^{\varepsilon})\right]\right\rbrace
= -\inf_{x\in \mathscr{E}}\left\{h(x)+I(x)\right\},$$
then $\{X^\varepsilon\}$ is said to satisfy the Laplace principle on $\mathscr{E}$ with a rate function $I$.
\end{definition}

Next, we introduce sufficient conditions for a sequence of Wiener functionals to satisfy the Laplace principle. Let $(\mathscr{U},\langle\cdot,\cdot\rangle_{\mathscr{U}})$ be a Hilbert space. Let $W$ be an $\mathscr{U}$-valued cylindrical Wiener process, where there is another Hilbert space $\mathscr{U}_0$ such that the embedding $\mathscr{U}\subset\mathscr{U}_0$ is Hilbert-Schmidt and the path of $W$ take values in $C([0,T];\mathscr{U}_0)$, defined on
complete  filtered probability space $(\Omega, \mathscr{F},\{\mathscr{F}_{t}\}_{t\geq0}, \mathbb{P})$. Consider a collection of control functions
$$
\mathcal A :=\Bigg\{\phi:\phi ~\text{is  $\mathscr{U}$-valued
 $(\mathscr{F}_t)$-predictable process and} \displaystyle\int_0^T\|\phi_s(\omega)\|_\mathscr{U}^2ds<\infty, \mathbb P\text{-a.s.}\Bigg\}.
$$
For any $M > 0$, we also introduce the following sets
\begin{eqnarray*}
\!\!\!\!\!\!\!\!&&S_M:=\Bigg\{\phi\in L^2([0,T];\mathscr{U}):\int_0^T\|\phi_s\|_\mathscr{U}^2ds\leq M\Bigg\},
\nonumber \\
\!\!\!\!\!\!\!\!&&\mathcal A_M:=\Big\{\phi\in \mathcal A:\phi_\cdot(\omega)\in S_M, \mathbb P\text{-a.s.}\Big\},
\end{eqnarray*}
where $S_M$ equipped with the weak topology in $L^2([0,T];\mathscr{U})$ turns out to be a Polish space.
%By a meausure-theoretical argument, we can show that any element in $\mathcal A_M$ is an $S_M$-valued random variable

In the following, a powerful criterion established by Budhiraj and Dupuis in \cite{budhiraja2000variational}  for the Laplace principle is given.
\begin{lemma}\label{WeakLemma}
Assume that
$\mathcal{G}^\varepsilon:C([0,T];\mathscr{U}_0)\to \mathscr{E}$ is a family of measurable mappings. If $X^{\varepsilon}=\mathcal{G}^\varepsilon(W_\cdot)$ and there exists a measurable map $\mathcal{G}^0:C([0,T];\mathscr{U}_0)\to \mathscr{E}$ such that the following two conditions hold:

\vspace{1mm}
(a) Let $\{\phi^\varepsilon:\varepsilon>0\}\subset\mathcal A_M$,~$M<\infty$. If $\phi^\varepsilon$ converges to $\phi$ in distribution as $S_M$-valued random elements, then
$$\mathcal{G}^\varepsilon\Big(W_\cdot+\displaystyle\frac 1{\sqrt{\varepsilon}}\int_0^\cdot\phi^\varepsilon_sds\Big)\to \mathcal{G}^0\Big(\int_0^\cdot\phi_sds\Big)$$
in distribution as $\varepsilon\to0.$

\vspace{1mm}
(b) For any $M<\infty,$ the set
$$K_M:=\Big\{\mathcal{G}^0\Big(\int_0^\cdot\phi_sds\Big):\phi\in S_M\Big\} $$
is a compact subset in $\mathscr{E}.$\\
Then the family $\{X^{\varepsilon}\}$ satisfies the Laplace principle (hence LDP) on $\mathscr{E}$ with the following good rate function
\begin{equation*}
I(f)=\inf _{\{\phi \in L^{2}([0, T] ; \mathscr{U_0}): f=\mathcal{G}^{0}(\int_{0}^{\cdot} \phi_{s} d s)\}}\left \{\frac{1}{2} \int_{0}^{T}\|\phi_{s}\|_{\mathscr{U_0}}^{2} d s\right \}.
\end{equation*}
\end{lemma}
\subsection{Hypotheses and main results}
Let $(\mathbb{U}_i,\langle \cdot,\cdot\rangle_{\mathbb{U}_i})$ and $(\mathbb{H}_i,\langle \cdot,\cdot\rangle_{\mathbb{H}_i}),~i=1,2,$ be separable Hilbert spaces with norms $\|\cdot\|_{\mathbb{U}_i}$ and $\|\cdot\|_{\mathbb{H}_i}$, respectively. Let $(\mathbb{V}_i,\|\cdot\|_{\mathbb{V}_i})$, $i=1,2$, be reflexive Banach spaces that are continuously and densely embedded into $\mathbb{H}_i$, $i=1,2$, respectively.  Identifying $\mathbb{H}_i$  via the Riesz isomorphism, we have the following Gelfand triples
$$\mathbb{V}_i\subset \mathbb{H}_i(\simeq \mathbb{H}_i^*)\subset \mathbb{V}^*_i,~i=1,2.$$
We denote the dualization between $\mathbb{V}_{i}^{*}$ and $\mathbb{V}_{i}$ by $_{\mathbb{V}_i^*}\langle\cdot, \cdot\rangle_{\mathbb{V}_{i}}$  (i.e.~$_{\mathbb{V}_i^*}\langle z, v\rangle_{\mathbb{V}_{i}}:=z(v)$ for $z \in \mathbb{V}_{i}^{*},~v \in \mathbb{V}_{i}$), it follows that
$$_{\mathbb{V}_i^*}\langle z, v\rangle_{\mathbb{V}_{i}}=\langle z, v\rangle_{\mathbb{H}_{i}}, \text{for all}~z \in \mathbb{H}_{i},~v \in \mathbb{V}_{i}.$$
Let $L_{2}(\mathbb{U}_{i}, \mathbb{H}_{i})$ denote the spaces of all Hilbert-Schmidt operators from $\mathbb{U}_{i}$ to $\mathbb{H}_{i}$ equipped with the Hilbert-Schmidt  norm $\|\cdot\|_{L_{2}(\mathbb{U}_{i}, \mathbb{H}_{i})}.$ For a Banach space $(\mathbb{B},\|\cdot\|_{\mathbb{B}})$, we denote by  $C([0,T];\mathbb{B})$ the space of all continuous functions on $[0,T]$ onto $\mathbb{B}$, which is a Polish space with respect to the metric
\begin{equation*}
d(h,g):=\sup_{t\in[0,T]}\|h_t-g_t\|_{\mathbb{B}}.
\end{equation*}

\vspace{1mm}
In the present paper, we assume that measurable maps
\begin{eqnarray*}
\!\!\!\!\!\!\!\!&&\mathfrak{A}_{1}: \mathbb{V}_{1} \to \mathbb{V}_{1}^{*};\\
\!\!\!\!\!\!\!\!&&\mathfrak{f}: \mathbb{H}_{1} \times \mathbb{H}_{2} \to \mathbb{H}_{1}; \\ \!\!\!\!\!\!\!\!&&\mathfrak{B}_{1}: \mathbb{V}_{1} \to L_{2}(\mathbb{U}_{1}, \mathbb{H}_{1});\\
\!\!\!\!\!\!\!\!&&\mathfrak{A}_{2}: \mathbb{H}_{1} \times \mathbb{V}_{2} \to \mathbb{V}_{2}^{*}; \\
\!\!\!\!\!\!\!\!&&\mathfrak{B}_{2}: \mathbb{H}_{1} \times \mathbb{V}_{2} \to L_{2}(\mathbb{U}_{2}, \mathbb{H}_{2})
\end{eqnarray*}
satisfy the following conditions:

\begin{hypothesis}\label{hypo1} (Slow component)
Assume that the embedding $\mathbb{V}_1\subset \mathbb{H}_1$ is compact, and there exist constants $\alpha_1 >1,~\beta_1 \geq 0,$
 and $\eta_1,C>0$ such that for all $u,v,w \in \mathbb{V}_1,~u_1,u_2\in \mathbb{H}_1,~v_1,v_2\in \mathbb{H}_2$, the following conditions hold.
\begin{enumerate}
\item [$({\mathbf{A}}{\mathbf{1}})$] (Hemicontinuity) The map $\lambda\mapsto\,_{\mathbb{V}_1^*}\langle \mathfrak{A}(u+\lambda v),w\rangle_{\mathbb{V}_1}$ is continuous on $\mathbb R$.
\item [$({\mathbf{A}}{\mathbf{2}})$] (Local monotonicity)
\begin{equation*}
_{\mathbb{V}_{1}^{*}}\left\langle \mathfrak{A}_{1}(u)-\mathfrak{A}_{1}(v), u-v\right\rangle_{\mathbb{V}_{1}}
\leq(C+\rho(u)+\eta(v))\|u-v\|_{\mathbb{H}_{1}}^{2},
\end{equation*}
where $\rho,\eta:\mathbb{V}_1\to \mathbb{R}_{+}$ are two measurable functions that satisfy
\begin{equation}\label{conloc}
\rho(u)+\eta(u) \leq C(1+\|u\|_{\mathbb{V}_{1}}^{\alpha_{1}})(1+\|u\|_{\mathbb{H}_{1}}^{\beta_1}).
\end{equation}
\item [$({\mathbf{A}}{\mathbf{3}})$] (Coercivity)
    \begin{equation*}
    2_{\mathbb{V}_{1}^{*}}\langle \mathfrak{A}_{1}(u), u\rangle_{\mathbb{V}_{1}}+\|\mathfrak{B}_{1}(u)\|_{L_{2}(\mathbb{U}_{1}, \mathbb{H}_{1})}^{2} \leq-\eta_{1}\|u\|_{\mathbb{V}_{1}}^{\alpha_{1}}+C(1+\|u\|_{\mathbb{H}_{1}}^{2}).
    \end{equation*}
\item [$({\mathbf{A}}{\mathbf{4}})$] (Growth)
\begin{equation*}
\|\mathfrak{A}_{1}(u)\|_{\mathbb{V}_{1}^{*}}^{\frac{\alpha_{1}}{\alpha_{1}-1}} \leq C(1+\|u\|_{\mathbb{V}_{1}}^{\alpha_{1}})(1+\|u\|_{\mathbb{H}_{1}}^{\beta_{1}}).
\end{equation*}
\item [$({\mathbf{A}}{\mathbf{5}})$](Lipschitz of $\mathfrak{f},\mathfrak{B}$)
\begin{equation*}
\|\mathfrak{f}(u_{1}, v_{1})-\mathfrak{f}(u_{2}, v_{2})\|_{\mathbb{H}_{1}} \leq C(\|u_{1}-u_{2}\|_{\mathbb{H}_{1}}+\|v_{1}-v_{2}\|_{\mathbb{H}_{2}}),\label{fLipschitz}
\end{equation*}
\begin{equation}\label{growthB1}
\|\mathfrak{B}_{1}(u)-\mathfrak{B}_{1}(v)\|_{L_{2}(\mathbb{U}_{1}, \mathbb{H}_{1})} \leq C\|u-v\|_{\mathbb{H}_{1}}.
\end{equation}

\end{enumerate}
\end{hypothesis}

\begin{hypothesis}\label{hypo2}(Fast component)
Assume that the embedding $\mathbb{V}_2\subset \mathbb{H}_2$ is compact, and there exist constants $\alpha_2 >1,\beta_2\geq0,\eta_2 ,\kappa>0$ and $C>0$ such that for all $v,v_1,v_2,w \in \mathbb{V}_2$, $u,u_1,u_2\in \mathbb{H}_1$ the following conditions hold.
\begin{enumerate}
\item [$({\mathbf{H}}{\mathbf{1}})$] (Hemicontinuity) The map $\lambda\mapsto\,_{\mathbb{V}_2^*}\langle \mathfrak{A}_2(u_1+\lambda u_2,v_1+\lambda v_2),w\rangle_{\mathbb{V}_2}$ is continuous on $\mathbb R$.
\item [$({\mathbf{H}}{\mathbf{2}})$] (Strict monotonicity)
\begin{eqnarray*}
\!\!\!\!\!\!\!\!&&2\,_{\mathbb{V}_2^*}\langle \mathfrak{A}_2(u_1,v_1)-\mathfrak{A}_2(u_1,v_2),v_1-v_2\rangle_{\mathbb{V}_2}+\|\mathfrak{B}_2(u_1,v_1)-\mathfrak{B}_2(u_1,v_2)\|_{L_2(\mathbb{U}_2,\mathbb{H}_2)}^2\\
\!\!\!\!\!\!\!\!&&\leq-\kappa\|v_1-v_2\|_{\mathbb{H}_2}^2+C\|u_1-u_2\|_{\mathbb{H}_1}^2
\end{eqnarray*}
\text{and}
\begin{equation*}
\|\mathfrak{B}_2(u_1,v_1)-\mathfrak{B}_2(u_2,v_2)\|_{L_2(\mathbb{U}_2,\mathbb{H}_2)}\leq C(\|u_1-u_2\|_{\mathbb{H}_1}+\|v_1-v_2\|_{\mathbb{H}_2}).
\end{equation*}

\item [$({\mathbf{H}}{\mathbf{3}})$] (Coercivity)
\begin{equation*}
2_{\mathbb{V}_2^*}\langle \mathfrak{A}_2(u,v),v\rangle_{\mathbb{V}_2}\leq C\|v\|_{\mathbb{H}_2}^2-\eta_2\|v\|_{\mathbb{V}_2}^{\alpha_2}+C(1+\|u\|_{\mathbb{H}_1}^2).
\end{equation*}

\item [$({\mathbf{H}}{\mathbf{4}})$] (Growth)
\begin{equation*}
\|\mathfrak{A}_{2}(u, v)\|_{\mathbb{V}_{2}^{*}}^{\frac{\alpha_{2}}{\alpha_{2}-1}} \leq C(1+\|v\|_{\mathbb{V}_{2}}^{\alpha_{2}})(1+\|v\|_{\mathbb{H}_{2}}^{\beta_{2}})+C\|u\|_{\mathbb{H}_{1}}^{2}
\end{equation*}
and
\begin{equation*}
\sup_{v\in \mathbb{V}_2}\|\mathfrak{B}_2(u,v)\|_{L_2(\mathbb{U}_2,\mathbb{H}_2)}\leq C(1+\|u\|_{\mathbb{H}_1}).
\end{equation*}
\end{enumerate}
\end{hypothesis}

The following well-posedness result of (\ref{aimequation}) can be easily proven under Hypothesises \ref{hypo1}-\ref{hypo2} by utilizing the arguments presented in \cite{liu2015stochastic} and \cite{RSZ}. Thus we omit the details.

\begin{lemma}\label{Lem2.1}
Suppose that Hypothesises \ref{hypo1}-\ref{hypo2} hold. For
each starting point $(x,y) \in \mathbb{H}_1\times \mathbb{H}_2$ and $\varepsilon,\delta >0$, (\ref{aimequation}) has a unique solution $(X^{\varepsilon,\delta},Y^{\varepsilon,\delta})$, which refers to a continuous $\mathbb{H}_1\times \mathbb{H}_2$-valued $(\mathscr F_t)$-adapted process  if for its $dt\otimes \mathbb P$-equivalent class $(\hat{X}^{\varepsilon,\delta},\hat{Y}^{\varepsilon,\delta})$ we have
\begin{equation*}
\hat{X}^{\varepsilon,\delta}\in L^{\alpha_1}([0,T]\times\Omega,dt\otimes\mathbb P;\mathbb{V}_1)\cap L^2([0,T]\times\Omega,dt\otimes\mathbb P;\mathbb{H}_1),
\end{equation*}
\begin{equation*}
\hat{Y}^{\varepsilon,\delta}\in L^{\alpha_2}([0,T]\times\Omega,dt\otimes\mathbb P;\mathbb{V}_2)\cap L^2([0,T]\times\Omega,dt\otimes\mathbb P;\mathbb{H}_2),
\end{equation*}
where $\alpha_1,\alpha_2$ are defined in $({\mathbf{A}}{\mathbf{3}})$ and $({\mathbf{H}}{\mathbf{3}})$, respectively, and $\mathbb P$-a.s.
\begin{eqnarray*}
\left\{ \begin{aligned}
  &\widetilde{X}_{t}^{\varepsilon,\delta}=x+\displaystyle\int_0^t[\mathfrak{A}_{1}(\widetilde{X}_{s}^{\varepsilon,\delta})+\mathfrak{f}(\widetilde{X}_{s}^{\varepsilon,\delta},\widetilde{Y}_{s}^{\varepsilon,\delta})] d s+\displaystyle\sqrt{\varepsilon}\int_{0}^{t} \mathfrak{B}_{1}(\widetilde{X}_{s}^{\varepsilon,\delta}) d W_{s}^{1}, \\
  &\widetilde{Y}_{t}^{\varepsilon,\delta}=y+\displaystyle\frac{1}{\delta} \int_{0}^{t} \mathfrak{A}_{2}(\widetilde{X}_{s}^{\varepsilon,\delta}, \widetilde{Y}_{s}^{\varepsilon,\delta}) d s+\displaystyle\frac{1}{\sqrt{\delta}} \int_{0}^{t} \mathfrak{B}_{2}(\widetilde{X}_{s}^{\varepsilon,\delta}, \widetilde{Y}_{s}^{\varepsilon,\delta}) d W_{s}^{2},
\end{aligned} \right.
\end{eqnarray*}
where $(\widetilde{X}^{\varepsilon,\delta},\widetilde{Y}^{\varepsilon,\delta})$ is any $\mathbb{V}_1\times \mathbb{V}_2$-valued progressively measurable $dt\otimes\mathbb P$ version of $(\hat{X}^{\varepsilon,\delta},\hat{Y}^{\varepsilon,\delta})$.
\end{lemma}

In this paper, we will employ the theory of pseudo-monotone operator. To this end, we first recall the definition of pseudo-monotone operator. For abbreviation, we use the notation  ``$\rightharpoonup$'' for weak convergence in a Banach space.

\begin{definition}\label{de1} An operator $\mathfrak{A}$ from $\mathbb{V}$ to $\mathbb{V}^*$ is called to be a pseudo-monotone operator if $u_n\rightharpoonup u$ in $\mathbb{V}$ and
$$\liminf _{n \to \infty}\,_{\mathbb{V}^*}\langle \mathfrak{A}(u_{n}), u_{n}-u\rangle_{\mathbb{V}} \geq 0,$$
then for any $v \in \mathbb{V}$,
$$\limsup _{n \to \infty}\,_{\mathbb{V}^*}\langle \mathfrak{A}(u_{n}), u_{n}-v\rangle_{\mathbb{V}} \leq \,_{\mathbb{V}^*}\langle \mathfrak{A}(u), u-v\rangle_{\mathbb{V}}.$$
\end{definition}

\begin{remark}
Note that Browder \cite{Browder1977} introduced a different definition  of pseudo-monotone operator: An operator $\mathfrak{A}$ from $\mathbb{V}$ to $\mathbb{V}^*$ is called  pseudo-monotone if $u_n\rightharpoonup u$ in $\mathbb{V}$ and
$$\liminf _{n \to \infty}\,_{\mathbb{V}^*}\langle \mathfrak{A}(u_{n}), u_{n}-u\rangle_{\mathbb{V}} \geq 0$$
implies $\mathfrak{A}(u_n)\rightharpoonup\mathfrak{A}(u)$ in $\mathbb{V}^*$ and
$$\lim _{n \to \infty}\,_{\mathbb{V}^*}\langle \mathfrak{A}(u_{n}), u_{n}\rangle_{\mathbb{V}}=\,_{\mathbb{V}^*}\langle \mathfrak{A}(u), u\rangle_{\mathbb{V}}.$$
This definition turns out to be equivalent to Definition \ref{de1} (cf.~\cite[Remark 5.2.12]{liu2015stochastic}).
\end{remark}
\begin{remark}\label{re2.2}
We  emphasize that the conditions $({\mathbf{A}}{\mathbf{1}})$ and $({\mathbf{A}}{\mathbf{2}})$ imply that  the map $\mathbb{V}_1 \ni u  \mapsto \mathfrak{A}_1(u) \in \mathbb{V}_1^{*}$  is  pseudo-monotone  (see Lemma 2.15 in \cite{RSZ}), which plays a crucial role in both the proof of the well-posedness of the skeleton equation (\ref{skletonequation}) and the verification of sufficient condition (b) in Lemma \ref{WeakLemma}.
\end{remark}

To present the main result of this work, we need to make some necessary notations. Let $\mathscr{U}=\mathbb{U}_1\times \mathbb{U}_2$ be Hilbert product spaces. We define an $\mathscr{U}$-valued cylindrical Wiener process $W$,  then  we can select projection operators $P_1:\mathscr{U} \to \mathbb{U}_1,P_2:\mathscr{U} \to \mathbb{U}_2$ such that cylindrical Wiener processes $W^i, i=1,2,$ are given by
$$W^1_t:=P_1W_t,~W^2_t:=P_2W_t.$$

Now, we define the following skeleton equation
\begin{eqnarray}\label{skletonequation}
\left\{\begin{aligned}
&\displaystyle\frac{d \bar{X}_{t}^{\phi}}{d t}=\mathfrak{A}_1(\bar{X}_{t}^{\phi})+\bar{\mathfrak{f}}(\bar{X}_{t}^{\phi})+\mathfrak{B}_{1}(\bar{X}_{t}^{\phi})P_1\phi,\\
&\bar{X}_{0}^{\phi}=x,
\end{aligned}\right.
\end{eqnarray}
where $\phi\in L^2([0,T];\mathscr{U}),~\bar{\mathfrak{f}}(x):=\displaystyle\int_{\mathbb{H}_2}\mathfrak{f}(x,z)\mu^{x}(dz),~x\in{\mathbb{H}_1},$ and $\mu^{x}$ is the
unique invariant measure of the Markov semigroup to the parameterized equation  associated with the fast component in (\ref{aimequation}) (see (\ref{frozenequation}) below). The well-posedness of solutions to (\ref{skletonequation}) and some useful lemmas are postponed in Appendix. Based on this, we define a measurable map $\mathcal{G}^0:C([0,T];\mathscr{U}_0)\to C([0,T];\mathbb{H}_1)$ by
\begin{eqnarray*}
\mathcal{G}^0\Big(\int_0^\cdot\phi_s ds\Big):=\left\{ \begin{aligned}&\bar{X}^\phi,~~\phi\in L^2([0,T]; \mathscr{U});\\
&0,~~~~~~\text{otherwise}.
\end{aligned} \right.
\end{eqnarray*}

Now, we present the main result of this paper.

\begin{theorem}\label{t1}
Assume that  Hypothesises \ref{hypo1}-\ref{hypo2} hold. If
\begin{equation*}
\lim _{\varepsilon \to 0} \frac{\delta}{\varepsilon}=0,
\end{equation*}
then $\{X^{\varepsilon,\delta}:\varepsilon>0\}$ to (\ref{aimequation}) satisfies the LDP on $C([0,T]; \mathbb{H}_1)$ with good rate function
\begin{equation}\label{ratef}
I(f)=\inf _{\{\phi \in L^{2}([0, T] ; \mathscr{U_0}): f=\mathcal{G}^{0}(\int_{0}^{\cdot} \phi_{s} d s)\}}\left \{\frac{1}{2} \int_{0}^{T}\|\phi_{s}\|_{\mathscr{U_0}}^{2} d s\right \}.
\end{equation}

\end{theorem}

\begin{remark}
Compared to previous works \cite{HS,HLL,SWXY}, the main novelties in Theorem \ref{t1} are twofold. Firstly, we incorporate  the theory of pseudo-monotone operator to extend the slow component to encompass fully local monotone operators. This extension allows us to explore a wider range of interesting applications of   SPDEs, such as  stochastic Cahn-Hilliard equations and stochastic 2D Liquid crystal equations, which cannot be covered by  existing papers.
Secondly, we present more general estimates and offers greater flexibility that encompass more general multiplicative noise on the fast component.

\end{remark}

\section{Stochastic controlled equations}
\setcounter{equation}{0}
 \setcounter{definition}{0}
In this section, we introduce the stochastic controlled equation associated with (\ref{aimequation}) in the weak convergence analysis, and derive its necessary a priori estimates.

\subsection{Some energy estimates}
In view of Lemma \ref{Lem2.1} and the Yamada-Watanabe theorem in infinite dimensions, there exists a Borel-measurable function
$$\mathcal{G}^{\varepsilon}:C([0,T];\mathscr{U}_0)\to C([0,T];\mathbb{H}_1)$$
such that we have the representation
$X^{\varepsilon,\delta}=\mathcal{G}^{\varepsilon}(W_\cdot),~\mathbb P\text{-a.s.}.$
Then for any $\phi^{\varepsilon} \in \mathcal{A}_{M}$, by means of Girsanov theorem, it is clear that
$$X^{\phi^{\varepsilon}}:=\mathcal{G}^{\varepsilon}\left(W .+\frac{1}{\sqrt{\varepsilon}} \int_{0}^{\cdot} \phi_{s}^{\varepsilon} d s\right)$$
is the first component of solution $(X^{\phi^{\varepsilon}},Y^{\phi^{\varepsilon}})$ to the following stochastic controlled equation
\begin{eqnarray}\label{controlequation}
\left\{ \begin{aligned}
  &dX_{t}^{\phi^{\varepsilon}}=[\mathfrak{A}_{1}(X_{t}^{\phi^{\varepsilon}})+\mathfrak{f}(X_{t}^{\phi^{\varepsilon}}, Y_{t}^{\phi^{\varepsilon}})] d t
  +\mathfrak{B}_{1}(X_{t}^{\phi^{\varepsilon}})P_1\phi^\varepsilon_tdt+\sqrt{\varepsilon}\mathfrak{B}_{1}(X_{t}^{\phi^{\varepsilon}}) d W_{t}^{1}, \\
  &d Y_{t}^{\phi^{\varepsilon}}=\frac{1}{\delta} \mathfrak{A}_{2}(X_{t}^{\phi^{\varepsilon}}, Y_{t}^{\phi^{\varepsilon}}) d t+\frac{1}{\sqrt{\delta\varepsilon}}\mathfrak{B}_{2}(X_{t}^{\phi^{\varepsilon}}, Y_{t}^{\phi^{\varepsilon}})P_2\phi^\varepsilon_tdt
  +\frac{1}{\sqrt{\delta}} \mathfrak{B}_{2}(X_{t}^{\phi^{\varepsilon}}, Y_{t}^{\phi^{\varepsilon}}) d W_{t}^{2}, \\
  &X_{0}^{\phi^{\varepsilon}}=x\in \mathbb{H}_1,~~Y_{0}^{\phi^{\varepsilon}}=y\in \mathbb{H}_2.
\end{aligned}\right.
\end{eqnarray}

We first derive some improved energy estimates compared with \cite{HLL}, which is of great importance in the verification of the criterion (a).
\begin{lemma}\label{controlestimate}
For any $T>0$, $p\geq 1$ and $\phi^\varepsilon\in\mathcal A_M,M<\infty$,  there exists a  constant $C_{p,T,M}>0$ such that for any $x\in \mathbb{H}_1,~y\in \mathbb{H}_2$ and $\varepsilon,\delta>0$ small enough,
\begin{equation}\label{conslowestimate}
\mathbb E\Bigg\{\sup_{t\in[0,T]}\|X^{\phi^\varepsilon}_t\|_{\mathbb{H}_1}^{2p}+\Big( \int_0^T\|X^{\phi^\varepsilon}_t\|_{\mathbb{V}_1}^{\alpha_1}dt\Big)^p\Bigg\} \leq C_{p,T,M}(1+\|x\|_{\mathbb{H}_1}^{2p}+\|y\|_{\mathbb{H}_2}^{2p}),
\end{equation}
\begin{equation*}
\mathbb E\Big(\int_0^T\|Y^{\phi^\varepsilon}_t\|_{\mathbb{H}_2}^2dt\Big)^p\leq C_{p,T,M}(1+\|x\|_{\mathbb{H}_1}^{2p}+\|y\|_{\mathbb{H}_2}^{2p}).
\end{equation*}
\end{lemma}
\begin{proof}
Applying It\^{o}'s formula, we can get that
\begin{eqnarray}\label{con1}
\!\!\!\!\!\!\!\!&&\|Y^{\phi^\varepsilon}_t\|_{\mathbb{H}_2}^2
\nonumber\\
=\!\!\!\!\!\!\!\!&&\|y\|_{\mathbb{H}_2}^2+\frac {1}{\delta}\int_0^t \big(2\,_{\mathbb{V}_2^*}\langle \mathfrak{A}_2(X^{\phi^\varepsilon}_s,Y^{\phi^\varepsilon}_s),Y^{\phi^\varepsilon}_s\rangle_{\mathbb{V}_2}
+\|\mathfrak{B}_2(X^{\phi^\varepsilon}_s,Y^{\phi^\varepsilon}_s)\|_{L_2(\mathbb{U}_2,\mathbb{H}_2)}^2\big)ds
\nonumber\\
\!\!\!\!\!\!\!\!&&+\frac 2{\sqrt{\varepsilon\delta}}\int_0^t\langle \mathfrak{B}_2(X^{\phi^\varepsilon}_s,Y^{\phi^\varepsilon}_s)P_2\phi^\varepsilon_s,Y^{\phi^\varepsilon}_s\rangle_{\mathbb{H}_2}ds
+\frac 2{\sqrt{\delta}}\int_0^t\langle \mathfrak{B}_2(X^{\phi^\varepsilon}_s,Y^{,\phi^\varepsilon}_s)dW^2_s,Y^{\phi^\varepsilon}_s\rangle_{\mathbb{H}_2}.
\end{eqnarray}
By similar arguments as in the proof of \cite[Theorem 4.3.8]{liu2015stochastic}, the following claim can be obtained: For any $u\in \mathbb{H}_1,~v\in \mathbb{V}_2,$ there is a constant $\lambda \in (0,\kappa)$ such that
\begin{equation}\label{e1}
2\,_{\mathbb{V}_2^*}\langle \mathfrak{A}_2(u,v),v\rangle_{\mathbb{V}_2}\leq -\lambda\|v\|_{\mathbb{H}_2}^2+C(1+\|u\|_{\mathbb{H}_1}^2).
\end{equation}
For the third term of the right-hand side of (\ref{con1}), using Young's inequality we have
\begin{eqnarray}\label{con3}
\!\!\!\!\!\!\!\!&&\frac 2{\sqrt{\varepsilon\delta}}\int_0^t\langle \mathfrak{B}_2(X^{\phi^\varepsilon}_s,Y^{\phi^\varepsilon}_s)P_2\phi^\varepsilon_s,Y^{\phi^\varepsilon}_s\rangle_{\mathbb{H}_2}ds\nonumber\\
\leq\!\!\!\!\!\!\!\!&& \frac {\gamma}{\delta}\int_0^t\|Y^{\phi^\varepsilon}_s\|_{\mathbb{H}_2}^2ds+\frac C\varepsilon\int_0^t\|\mathfrak{B}_2(X^{\phi^\varepsilon}_s,Y^{\phi^\varepsilon}_s)\|_{L_2(\mathbb{U}_2,\mathbb{H}_2)}^2\|P_2\|^2\|\phi^\varepsilon_s\|^2_\mathscr{U}ds\nonumber\\
\leq\!\!\!\!\!\!\!\!&& \frac {\gamma}{\delta}\int_0^t\|Y^{\phi^\varepsilon}_s\|_{\mathbb{H}_2}^2ds+\frac C\varepsilon\int_0^t[(1+\|X^{\phi^\varepsilon}_s\|^2_{\mathbb{H}_1})\|\phi^\varepsilon_s\|^2_\mathscr{U}]ds,
\end{eqnarray}
where we choose $\gamma\in(0,\lambda)$ and used the condition $({\mathbf{H}}{\mathbf{4}})$ in the last step.
Substituting (\ref{e1})-(\ref{con3}) into (\ref{con1}) and applying  $({\mathbf{H}}{\mathbf{4}})$ again yields that
\begin{eqnarray*}
\|Y^{\phi^\varepsilon}_t\|_{\mathbb{H}_2}^2\leq\!\!\!\!\!\!\!\!&& \|y\|_{\mathbb{H}_2}^2-\frac {\lambda-\gamma}{\delta}\int_0^t\|Y^{\phi^\varepsilon}_s\|_{\mathbb{H}_2}^2ds+\frac C \delta\int_0^t(1+\|X^{\phi^\varepsilon}_s\|_{\mathbb{H}_1}^2)ds\\
\!\!\!\!\!\!\!\!&&+\frac C\varepsilon\int_0^t[(1+\|X^{\phi^\varepsilon}_s\|^2_{\mathbb{H}_1})\|\phi^\varepsilon_s\|^2_\mathscr{U}]ds+\frac 2{\sqrt{\delta}}\Big|\int_0^t\langle \mathfrak{B}_2(X^{\phi^\varepsilon}_s,Y^{,\phi^\varepsilon}_s)dW^2_s,Y^{\phi^\varepsilon}_s\rangle_{\mathbb{H}_2}\Big|.
\end{eqnarray*}
Thus, it follows that
\begin{eqnarray*}
\int_0^t\|Y^{\phi^\varepsilon}_s\|_{\mathbb{H}_2}^2ds\leq\!\!\!\!\!\!\!\!&&C\|y\|_{\mathbb{H}_2}^2+C\int_0^t(1+\|X^{\phi^\varepsilon}_s\|_{\mathbb{H}_1}^2)ds\\
\!\!\!\!\!\!\!\!&&+C\Big(\frac \delta\varepsilon\Big)\int_0^t[(1+\|X^{\phi^\varepsilon}_s\|^2_{\mathbb{H}_1})\|\phi^\varepsilon_s\|^2_\mathscr{U}]ds\\
\!\!\!\!\!\!\!\!&&+C\sqrt{\delta}\Big|\int_0^t\langle \mathfrak{B}_2(X^{\phi^\varepsilon}_s,Y^{,\phi^\varepsilon}_s)dW^2_s,Y^{\phi^\varepsilon}_s\rangle_{\mathbb{H}_2}\Big|.
\end{eqnarray*}
Using Burkholder-Davis-Gundy inequality and  $({\mathbf{H}}{\mathbf{4}})$, we derive
\begin{eqnarray*}
\!\!\!\!\!\!\!\!&&\mathbb{E}\Big(\int_0^T \|Y^{\phi^\varepsilon}_t\|_{\mathbb{H}_2}^2dt\Big)^p\\
\leq\!\!\!\!\!\!\!\!&& C_{p}\|y\|_{\mathbb{H}_2}^{2p}+C_{p,T,M}\Big(\frac \delta\varepsilon\Big)^p\mathbb{E}\Big[\sup_{t\in[0,T]}(1+\|X^{\phi^\varepsilon}_t\|^{2p}_{\mathbb{H}_1})\Big]
+C_p\delta^{\frac{p}{2}}\mathbb{E}\Big(\int_0^T \|X^{\phi^\varepsilon}_t\|_{\mathbb{H}_1}^2dt\Big)^p\\
\!\!\!\!\!\!\!\!&&+C_p\mathbb{E}\Big(\int_0^T \|\mathfrak{B}_2(X^{\phi^\varepsilon}_s,Y^{,\phi^\varepsilon}_s)\|_{L_2(\mathbb{U}_2,\mathbb{H}_2)}^2\|Y^{\phi^\varepsilon}_t\|_{\mathbb{H}_2}^2dt\Big)^{\frac{p}{2}}\\
\leq\!\!\!\!\!\!\!\!&& C_{p}\|y\|_{\mathbb{H}_2}^{2p}+C_{p,T,M}\Big(\frac \delta\varepsilon\Big)^p\mathbb{E}\Big[\sup_{t\in[0,T]}(1+\|X^{\phi^\varepsilon}_t\|^{2p}_{\mathbb{H}_1})\Big]
+C_p\mathbb{E}\Big(\int_0^T \|X^{\phi^\varepsilon}_t\|_{\mathbb{H}_1}^2dt\Big)^p\\
\!\!\!\!\!\!\!\!&&+C_p\delta^{\frac{p}{2}}\mathbb{E}\Big(\int_0^T (1+\|X^{\phi^\varepsilon}_t\|^2_{\mathbb{H}_1})\|Y^{\phi^\varepsilon}_t\|_{\mathbb{H}_2}^2dt\Big)^{\frac{p}{2}}\\
\leq\!\!\!\!\!\!\!\!&& C_{p}\|y\|_{\mathbb{H}_2}^{2p}+C_{p,T,M}\Big[\Big(\frac \delta\varepsilon\Big)^p+\delta^p\Big]\mathbb{E}\Big[\sup_{t\in[0,T]}(1+\|X^{\phi^\varepsilon}_t\|^{2p}_{\mathbb{H}_1})\Big]\\
\!\!\!\!\!\!\!\!&&+C_p\mathbb{E}\Big(\int_0^T \|X^{\phi^\varepsilon}_t\|_{\mathbb{H}_1}^2dt\Big)^p+\frac{1}{2} \mathbb{E}\Big(\int_0^T \|Y^{\phi^\varepsilon}_t\|_{\mathbb{H}_2}^2dt\Big)^p,
\end{eqnarray*}
where we used Young's inequality in the last step.
Hence, we can get
\begin{eqnarray}\label{conY}
\mathbb{E}\Big(\int_0^T \|Y^{\phi^\varepsilon}_t\|_{\mathbb{H}_2}^2dt\Big)^p\leq \!\!\!\!\!\!\!\!&& C_{p,T,M}(1+\|y\|_{\mathbb{H}_2}^{2p})+C_p\mathbb{E}\Big(\int_0^T \|X^{\phi^\varepsilon}_t\|_{\mathbb{H}_1}^2dt\Big)^p\nonumber\\
\!\!\!\!\!\!\!\!&&+C_{p,T,M}\Big[\Big(\frac \delta\varepsilon\Big)^p+\delta^p\Big]\mathbb{E}\Big[\sup_{t\in[0,T]}\|X^{\phi^\varepsilon}_t\|^{2p}_{\mathbb{H}_1}\Big].
\end{eqnarray}

Now, we need to estimate $\|X^{\phi^\varepsilon}_t\|^{2p}_{\mathbb{H}_1}$. Applying It\^{o}'s formula for $\|X^{\phi^\varepsilon}_t\|^{2}_{\mathbb{H}_1}$, $t\in[0,T]$, we have
\begin{eqnarray}\label{con5}
\!\!\!\!\!\!\!\!&& \|X_{t}^{\phi^{\varepsilon}}\|_{\mathbb{H}_{1}}^{2} \nonumber\\
= \!\!\!\!\!\!\!\!&& \|x\|_{\mathbb{H}_{1}}^{2}+2 \int_{0}^{t}\,_{\mathbb{V}_1^*}\langle \mathfrak{A}_1(X_{s}^{\phi^{\varepsilon}}), X_{s}^{\phi^{\varepsilon}}\rangle_{\mathbb{V}_{1}} d s+2 \int_{0}^{t}\langle \mathfrak{f}(X_{s}^{\phi^{\varepsilon}}, Y_{s}^{\phi^{\varepsilon}}), X_{s}^{\phi^{\varepsilon}}\rangle_{\mathbb{H}_{1}}ds\nonumber\\
\!\!\!\!\!\!\!\!&&+2 \int_{0}^{t}\langle \mathfrak{B}_{1}(X_{s}^{\phi^{\varepsilon}})P_1 \phi_{s}^{\varepsilon}, X_{s}^{\phi^{\varepsilon}}\rangle_{\mathbb{H}_{1}} d s+\varepsilon \int_{0}^{t}\|\mathfrak{B}_{1}(X_{s}^{\phi^{\varepsilon}})\|_{L_{2}(\mathbb{U}_1, \mathbb{H}_{1})}^{2} d s \nonumber\\
\!\!\!\!\!\!\!\!&& +2 \sqrt{\varepsilon} \int_{0}^{t}\langle \mathfrak{B}_{1}(X_{s}^{\phi^{\varepsilon}}) d W^1_{s}, X_{s}^{\phi^{\varepsilon}}\rangle_{\mathbb{H}_{1}}
 \nonumber\\
=:\!\!\!\!\!\!\!\!&&\|x\|_{\mathbb{H}_{1}}^{2}+I(t)+II(t)+III(t)+IV(t)+2 \sqrt{\varepsilon} \int_{0}^{t}\langle \mathfrak{B}_{1}(X_{s}^{\phi^{\varepsilon}}) d W^1_{s}, X_{s}^{\phi^{\varepsilon}}\rangle_{\mathbb{H}_{1}}.
\end{eqnarray}
By condition $({\mathbf{A}}{\mathbf{3}})$ and $\varepsilon$ small enough,  we obtain
\begin{equation}\label{con6}
I(t)+IV(t)\leq -\eta_1\int_0^t \|X^{\phi^\varepsilon}_s\|_{\mathbb{V}_1}^{\alpha_1} ds+C\int_0^t \big(1+\|X^{\phi^\varepsilon}_s\|_{\mathbb{H}_1}^2\big)ds.
\end{equation}
According to (\ref{fLipschitz})  and using Young's inequality, we have
\begin{equation}\label{con7}
II(t)\leq C\int_{0}^{t}\big(1+\|X^{\phi^\varepsilon}_s\|_{\mathbb{H}_1}^2+\|Y^{\phi^\varepsilon}_s\|_{\mathbb{H}_1}^2\big)ds.
\end{equation}
Using H\"older's inequality and Young's inequality, the fourth term of the right-hand side of (\ref{con5}) are controlled by
\begin{eqnarray}\label{con8}
III(t)
\leq\!\!\!\!\!\!\!\!&&\frac{1}{2} \Big [\sup_{s\in[0,t]}\|X^{\phi^\varepsilon}_s\|_{\mathbb{H}_1}^2 \Big ]+C\Big (\int_{0}^{t} \| \mathfrak{B}_{1}(X_{s}^{\phi^{\varepsilon}}) \|_{L_{2}(\mathbb{U}_1, \mathbb{H}_{1})}\|P_1\|\|\phi^\varepsilon_s\|_\mathscr{U}ds\Big )^2\nonumber\\
\leq\!\!\!\!\!\!\!\!&&\frac{1}{2} \Big [\sup_{s\in[0,t]}\|X^{\phi^\varepsilon}_s\|_{\mathbb{H}_1}^2 \Big ]+C\Big (\int_{0}^{t} \| \mathfrak{B}_{1}(X_{s}^{\phi^{\varepsilon}}) \|_{L_{2}(\mathbb{U}_1, \mathbb{H}_{1})}^2ds\Big)\Big (\int_{0}^{t}\|\phi^\varepsilon_s\|^2_\mathscr{U}ds\Big)\nonumber\\
\leq\!\!\!\!\!\!\!\!&&\frac{1}{2} \Big [\sup_{s\in[0,t]}\|X^{\phi^\varepsilon}_s\|_{\mathbb{H}_1}^2 \Big ]+C\Big (\int_{0}^{t} (1+\|X^{\phi^\varepsilon}_s\|_{\mathbb{H}_1}^2)ds\Big)\Big (\int_{0}^{t}\|\phi^\varepsilon_s\|^2_\mathscr{U}ds\Big),
\end{eqnarray}
where the last step used condition $({\mathbf{A}}{\mathbf{5}})$. Then substituting (\ref{con6})-(\ref{con8}) into (\ref{con5}) and by the fact $\phi^{\varepsilon} \in \mathcal{A}_{M},$  we deduce
\begin{eqnarray*}
\!\!\!\!\!\!\!\!&&\Big (\sup_{t\in[0,T]}\|X^{\phi^\varepsilon}_t\|_{\mathbb{H}_1}^2+\eta_1\int_{0}^{T}\|X^{\phi^\varepsilon}_t\|_{\mathbb{V}_1}^{\alpha_1} dt\Big )^p\\
\leq\!\!\!\!\!\!\!\!&& C_{p,T,M}(1+\|x\|_{\mathbb{H}_1}^{2p})+C_{p,M}\Big (\int_{0}^{T}\|X^{\phi^\varepsilon}_t\|_{\mathbb{H}_1}^2dt\Big )^p+C_p\Big (\int_{0}^{T}\|Y^{\phi^\varepsilon}_t\|_{\mathbb{H}_2}^2dt\Big )^p\\
\!\!\!\!\!\!\!\!&&+C_p\varepsilon^{\frac{p}{2}}\Big[\sup_{t\in[0,T]}\Big| \int_{0}^{t}\langle \mathfrak{B}_{1}(X_{s}^{\phi^{\varepsilon}}) d W^1_{s}, X_{s}^{\phi^{\varepsilon}}\rangle_{\mathbb{H}_{1}}\Big|^p\Big].
\end{eqnarray*}
Taking expectation and following Burkholder-Davis-Gundy inequality, Young's inequality and condition $({\mathbf{A}}{\mathbf{5}})$, we have
\begin{eqnarray*}
\!\!\!\!\!\!\!\!&&\mathbb{E}\Big (\sup_{t\in[0,T]}\|X^{\phi^\varepsilon}_t\|_{\mathbb{H}_1}^2+\eta_1\int_{0}^{T}\|X^{\phi^\varepsilon}_t\|_{\mathbb{V}_1}^{\alpha_1} dt\Big )^p\\
\leq\!\!\!\!\!\!\!\!&& C_{p,T,M}(1+\|x\|_{\mathbb{H}_1}^{2p})+C_{p,M}\mathbb{E}\Big (\int_{0}^{T}\|X^{\phi^\varepsilon}_t\|_{\mathbb{H}_1}^2dt\Big )^p+C_p\mathbb{E}\Big (\int_{0}^{T}\|Y^{\phi^\varepsilon}_t\|_{\mathbb{H}_2}^2dt\Big )^p\\
\!\!\!\!\!\!\!\!&&+C_p\varepsilon^{\frac{p}{2}}\mathbb{E}\Big[ \int_{0}^{T}\|\mathfrak{B}_{1}(X_{t}^{\phi^{\varepsilon}})\|_{L_{2}(\mathbb{U}_1, \mathbb{H}_{1})}^{2}\|X_{t}^{\phi^{\varepsilon}}\|_{\mathbb{H}_1}^2 dt\Big]^{\frac{p}{2} }\\
\leq\!\!\!\!\!\!\!\!&& C_{p,T,M}(1+\|x\|_{\mathbb{H}_1}^{2p})+C_{p,M}\mathbb{E}\Big (\int_{0}^{T}\|X^{\phi^\varepsilon}_t\|_{\mathbb{H}_1}^2dt\Big )^p+C_p\mathbb{E}\Big (\int_{0}^{T}\|Y^{\phi^\varepsilon}_t\|_{\mathbb{H}_2}^2dt\Big )^p\\
\!\!\!\!\!\!\!\!&&+\frac{1}{2} \mathbb{E}\Big [\sup_{t\in[0,T]}\|X^{\phi^\varepsilon}_t\|_{\mathbb{H}_1}^{2p}\Big].
\end{eqnarray*}
Thus by (\ref{conY}), we take $\frac{\delta}{\varepsilon}$ and $\delta$ small enough and derive
\begin{eqnarray}\label{e2}
\!\!\!\!\!\!\!\!&&\mathbb{E}\Big (\sup_{t\in[0,T]}\|X^{\phi^\varepsilon}_t\|_{\mathbb{H}_1}^2+\eta_1\int_{0}^{T}\|X^{\phi^\varepsilon}_t\|_{\mathbb{V}_1}^{\alpha_1} dt\Big )^p
\nonumber\\
\leq\!\!\!\!\!\!\!\!&& C_{p,T,M}(1+\|x\|_{\mathbb{H}_1}^{2p}+\|y\|_{\mathbb{H}_2}^{2p})+C_{p,M}\mathbb{E}\Big (\int_{0}^{T}\|X^{\phi^\varepsilon}_t\|_{\mathbb{H}_1}^2dt\Big )^p.
\end{eqnarray}
By Gronwall's lemma we have
\begin{equation}\label{e3}
\mathbb{E}\Big [\sup_{t\in[0,T]}\|X^{\phi^\varepsilon}_t\|_{\mathbb{H}_1}^{2p}\Big]\leq C_{p,T,M}(1+\|x\|_{\mathbb{H}_1}^{2p}+\|y\|_{\mathbb{H}_2}^{2p}).
\end{equation}
Then substituting (\ref{e3}) into (\ref{e2}), it follows that
$$
\mathbb{E}\Big (\int_{0}^{T}\|X^{\phi^\varepsilon}_t\|_{\mathbb{V}_1}^{\alpha_1} dt\Big )^p
\leq C_{p,T,M}(1+\|x\|_{\mathbb{H}_1}^{2p}+\|y\|_{\mathbb{H}_2}^{2p}).
$$
Moreover, substituting (\ref{e3}) into (\ref{conY}), it follows that
$$\mathbb{E}\Big(\int_0^T \|Y^{\phi^\varepsilon}_t\|_{\mathbb{H}_2}^2dt\Big)^p\leq C_{p,T,M}(1+\|x\|_{\mathbb{H}_1}^{2p}+\|y\|_{\mathbb{H}_2}^{2p}).$$
We complete the proof.
\end{proof}

\subsection{A time increment estimate}
In order to employ the time discretization method, we give the following time increment estimates of the slow component to the controlled equation (\ref{controlequation}).

\begin{lemma}\label{condiff}
For any $T>0$, $\phi^\varepsilon\in\mathcal A_M,M<\infty$ and $x\in \mathbb{H}_1,\,y\in \mathbb{H}_2$, there exists a constant $C_{x,y,T,M}>0$ such that for any $\varepsilon,\delta,\zeta>0$ small enough, we have
$$\mathbb E\int_0^{T}\|X^{\phi^\varepsilon}_t-X^{\phi^\varepsilon}_{t(\zeta)}\|_{\mathbb{H}_1}^2dt\leq C_{x,y,T,M}\zeta,$$
where $t(\zeta):=[\frac t\zeta]\zeta$ and $[t]$ is the largest integer smaller than $t.$
\end{lemma}
\begin{proof}
By (\ref{conslowestimate}), we have
\begin{eqnarray}\label{controllDifferent}
\mathbb E\int_0^T\|X^{\phi^\varepsilon}_t-X^{\phi^\varepsilon}_{t(\zeta)}\|_{\mathbb{H}_1}^2dt
=\!\!\!\!\!\!\!\!&&\mathbb E\int_0^\zeta\|X^{\phi^\varepsilon}_t-x\|_{\mathbb{H}_1}^2dt+
\mathbb E\int_\zeta^T\|X^{\phi^\varepsilon}_t-X^{\phi^\varepsilon}_{t(\zeta)}\|_{\mathbb{H}_1}^2dt\nonumber\\
\leq \!\!\!\!\!\!\!\!&&C(1+\|x\|_{\mathbb{H}_1}^2+\|y\|_{\mathbb{H}_2}^2)\zeta+2\mathbb E\int_\zeta^T\|X^{\phi^\varepsilon}_t-X^{\phi^\varepsilon}_{t-\zeta}\|_{\mathbb{H}_1}^2 dt\nonumber\\
\!\!\!\!\!\!\!\!&&+2\mathbb E\int_\zeta^T\|X^{\phi^\varepsilon}_{t(\zeta)}-X^{\phi^\varepsilon}_{t-\zeta}\|_{\mathbb{H}_1}^2 dt
\nonumber\\
=: \!\!\!\!\!\!\!\!&&C(1+\|x\|_{\mathbb{H}_1}^2+\|y\|_{\mathbb{H}_2}^2)\zeta+I+II.
\end{eqnarray}
We now estimate the term $I$. Applying It\^{o}'s formula for $\|X_{t}^{\phi^\varepsilon}-X_{t-\zeta}^{\phi^\varepsilon}\|_{\mathbb{H}_{1}}^{2}$ yields that
\begin{eqnarray}\label{condiffx}
\!\!\!\!\!\!\!\!&& \|X_{t}^{\phi^\varepsilon}-X_{t-\zeta}^{\phi^\varepsilon}\|_{\mathbb{H}_{1}}^{2} \nonumber\\
= \!\!\!\!\!\!\!\!&& 2 \int_{t-\zeta}^{t} \,_{\mathbb{V}_{1}^{*}}\langle \mathfrak{A}_1(X_{s}^{\phi^\varepsilon}), X_{s}^{\phi^\varepsilon}-X_{t-\zeta}^{\phi^\varepsilon}\rangle_{\mathbb{V}_{1}} d s +2 \int_{t-\zeta}^{t}\langle \mathfrak{f}(X_{s}^{\phi^\varepsilon}, Y_{s}^{\phi^\varepsilon}), X_{s}^{\phi^{\epsilon}}-X_{t-\zeta}^{\phi^\varepsilon}\rangle_{\mathbb{H}_{1}} d s \nonumber\\
\!\!\!\!\!\!\!\!&& +2 \int_{t-\zeta}^{t}\langle \mathfrak{B}_{1}(X_{s}^{\phi^\varepsilon})P_1\phi_{s}^{\varepsilon}, X_{s}^{\phi^\varepsilon}-X_{t-\zeta}^{\phi^\varepsilon}\rangle_{\mathbb{H}_{1}} d s+\varepsilon \int_{t-\zeta}^{t}\|\mathfrak{B}_{1}(X_{s}^{\phi^\varepsilon})\|_{L_{2}(\mathbb{U}_1, \mathbb{H}_{1})}^{2} d s \nonumber\\
\!\!\!\!\!\!\!\!&& +2 \sqrt{\varepsilon} \int_{t-\zeta}^{t}\langle \mathfrak{B}_{1}(X_{s}^{\phi^\varepsilon}) d W^1_{s}, X_{s}^{\phi^\varepsilon}-X_{t-\zeta}^{\phi^\varepsilon}\rangle_{\mathbb{H}_{1}} \nonumber\\
=: \!\!\!\!\!\!\!\!&& \sum_{i=1}^{5} I_{i}(t) .
\end{eqnarray}
We use condition $({\mathbf{A}}{\mathbf{4}})$ and H\"older's inequality to get that
\begin{eqnarray}\label{condiff1}
\!\!\!\!\!\!\!\!&& \mathbb{E}\Big[\int_{\zeta}^{T}I_{1}(t)d t\Big] \nonumber\\
\leq \!\!\!\!\!\!\!\!&& C \mathbb{E}\Big[\int_{\zeta}^{T} \int_{t-\zeta}^{t}\|\mathfrak{A}_1(X_{s}^{\phi^\varepsilon})\|_{\mathbb{V}_{1}^{*}}\|X_{s}^{\phi^\varepsilon}-X_{t-\zeta}^{\phi^\varepsilon} \|_{\mathbb{V}_{1}}  d s d t\Big] \nonumber\\
\leq \!\!\!\!\!\!\!\!&& C\Big[\mathbb{E}\Big(\int_{\zeta}^{T} \int_{t-\zeta}^{t}\|\mathfrak{A}_1(X_{s}^{\phi^\varepsilon})\|_{\mathbb{V}_{1}^{*}}^{\frac{\alpha_{1}}{\alpha_{1}-1}}  d s d t\Big)\Big]^{\frac{\alpha_{1}-1}{\alpha_{1}}} \Big[\mathbb{E}\Big(\int_{\zeta}^{T} \int_{t-\zeta}^{t}\|X_{s}^{\phi^\varepsilon}-X_{t-\zeta}^{\phi^\varepsilon}\|_{\mathbb{V}_{1}}^{\alpha_{1}}  d s d t\Big)\Big]^{\frac{1}{\alpha_{1}}} \nonumber\\
\leq \!\!\!\!\!\!\!\!&& C\Big[\zeta \mathbb{E}\Big(\int_{0}^{T}(1+\|X_{t}^{\phi^\varepsilon}\|_{\mathbb{V}_{1}}^{\alpha_{1}})(1+\|X_{t}^{\phi^\varepsilon}\|_{\mathbb{H}_{1}}^{\beta_{1}})  d t\Big)\Big]^{\frac{\alpha_{1}-1}{\alpha_{1}}}\Big[\zeta \mathbb{E} \int_{0}^{T}\|X_{t}^{\phi^\varepsilon}\|_{\mathbb{V}_{1}}^{\alpha_{1}} d t\Big]^{\frac{1}{\alpha_{1}}} \nonumber\\
\leq \!\!\!\!\!\!\!\!&& C_{T} \zeta\Big(1+\|x\|_{\mathbb{H}_{1}}^{\frac{(2\beta_1+4)(\alpha_1-1)+2}{\alpha_1}}+\|y\|_{\mathbb{H}_{2}}^{\frac{(2\beta_1+4)(\alpha_1-1)+2}{\alpha_1}}\Big),
\end{eqnarray}
where the last step is due to
\begin{eqnarray*}
\!\!\!\!\!\!\!\!&&\mathbb{E}\Big(\int_{0}^{T}(1+\|X_{t}^{\phi^\varepsilon}\|_{\mathbb{V}_{1}}^{\alpha_{1}})(1+\|X_{t}^{\phi^\varepsilon}\|_{\mathbb{H}_{1}}^{\beta_{1}})  d t\Big)\\
\leq\!\!\!\!\!\!\!\!&&C_T+\mathbb{E}\int_{0}^{T}\|X_{t}^{\phi^\varepsilon}\|_{\mathbb{V}_{1}}^{\alpha_{1}}dt+C_T\mathbb{E}\Big[\sup_{t \in [0,T]} \|X_{t}^{\phi^\varepsilon}\|_{\mathbb{H}_{1}}^{\beta_{1}}\Big]\\
\!\!\!\!\!\!\!\!&&+\mathbb{E}\Big[\Big(\sup_{t \in [0,T]} \|X_{t}^{\phi^\varepsilon}\|_{\mathbb{H}_{1}}^{\beta_{1}}\Big)\Big(\int_{0}^{T}\|X_{t}^{\phi^\varepsilon}\|_{\mathbb{V}_{1}}^{\alpha_{1}}dt\Big)\Big]\\
\leq\!\!\!\!\!\!\!\!&&C_T+\mathbb{E}\int_{0}^{T}\|X_{t}^{\phi^\varepsilon}\|_{\mathbb{V}_{1}}^{\alpha_{1}}dt+C_T\mathbb{E}\Big[\sup_{t \in [0,T]} \|X_{t}^{\phi^\varepsilon}\|_{\mathbb{H}_{1}}^{\beta_{1}}\Big]\\
\!\!\!\!\!\!\!\!&&+C\mathbb{E}\Big[\sup_{t \in [0,T]} \|X_{t}^{\phi^\varepsilon}\|_{\mathbb{H}_{1}}^{2\beta_{1}}\Big]+C\mathbb{E}\Big(\int_{0}^{T}\|X_{t}^{\phi^\varepsilon}\|_{\mathbb{V}_{1}}^{\alpha_{1}}dt\Big)^2\\
\leq \!\!\!\!\!\!\!\!&& C_{T}(1+\|x\|_{\mathbb{H}_{1}}^{2\beta_1+4}+\|y\|_{\mathbb{H}_{2}}^{2\beta_1+4}).
\end{eqnarray*}
By H\"older's inequality and (\ref{fLipschitz}), we have
\begin{eqnarray}\label{condiff2}
\!\!\!\!\!\!\!\!&& \mathbb{E}\Big[\int_{\zeta}^{T}I_{2}(t) d t\Big] \nonumber\\
\leq \!\!\!\!\!\!\!\!&&  C\Big[\mathbb{E}\Big(\int_{\zeta}^{T} \int_{t-\zeta}^{t}\|\mathfrak{f}(X_{s}^{\phi^\varepsilon}, Y_{s}^{\phi^\varepsilon})\|_{\mathbb{H}_{1}}^{2}d s d t\Big)\Big]^{\frac{1}{2}}
\Big[\mathbb{E}\Big(\int_{\zeta}^{T} \int_{t-\zeta}^{t}\|X_{s}^{\phi^\varepsilon}-X_{t-\zeta}^{\phi^\varepsilon}\|_{\mathbb{H}_{1}}^{2}  d s d t\Big)\Big]^{\frac{1}{2}} \nonumber\\
\leq \!\!\!\!\!\!\!\!&&  C\Big[\zeta \mathbb{E}\Big(\int_{0}^{T}(1+\|X_{s}^{\phi^\varepsilon}\|_{\mathbb{H}_{1}}^{2}+\|Y_{s}^{\phi^\varepsilon}\|_{\mathbb{H}_{2}}^{2}) d s\Big)\Big]^{\frac{1}{2}}\Big[\zeta \mathbb{E}\Big(\int_{0}^{T}\|X_{s}^{\phi^\varepsilon}\|_{\mathbb{H}_{1}}^{2} d s\Big)\Big]^{\frac{1}{2}} \nonumber\\
\leq \!\!\!\!\!\!\!\!&&  C_{T} \zeta(1+\|x\|_{\mathbb{H}_{1}}^{2}+\|y\|_{\mathbb{H}_{2}}^{2}) .
\end{eqnarray}
By (\ref{growthB1}) and H\"older's inequality, it follows that
\begin{eqnarray}\label{condiff3}
\!\!\!\!\!\!\!\!&&  \mathbb{E}\Big[\int_{\zeta}^{T}\big(I_{3}(t)+I_{4}(t)\big)  d t\Big] \nonumber\\
\leq \!\!\!\!\!\!\!\!&&  C\Big[\mathbb{E}\Big(\int_{\zeta}^{T} \int_{t-\zeta}^{t}\|\mathfrak{B}_{1}(X_{s}^{\phi^\varepsilon})\|_{L_{2}(\mathbb{U}_1, \mathbb{H}_{1})}^{2}\|\phi_{s}^{\varepsilon}\|_{\mathscr{U}}^{2}  d s d t\Big)\Big]^{\frac{1}{2}} \nonumber\\
\!\!\!\!\!\!\!\!&& \cdot\Big[\mathbb{E}\Big(\int_{\zeta}^{T} \int_{t-\zeta}^{t}\|X_{s}^{\phi^\varepsilon}-X_{t-\zeta}^{\phi^\varepsilon}\|_{\mathbb{H}_{1}}^{2}  d s d t\Big)\Big]^{\frac{1}{2}}
\nonumber\\
 \!\!\!\!\!\!\!\!&&
+\varepsilon\Big[\mathbb{E}\Big(\int_{\zeta}^{T} \int_{t-\zeta}^{t}\|\mathfrak{B}_{1}(X_{s}^{\phi^\varepsilon})\|_{L_{2}(\mathbb{U}_1, \mathbb{H}_{1})}^{2}dsdt\Big)\Big] \nonumber\\
\leq \!\!\!\!\!\!\!\!&&  C\Big[\zeta \mathbb{E}\Big(\int_{0}^{T}(1+\|X_{t}^{\phi^\varepsilon}\|_{\mathbb{H}_{1}}^{2})\|\phi_{t}^{\varepsilon}\|_{\mathscr{U}}^{2}  d t\Big)\Big]^{\frac{1}{2}}\Big[\zeta \mathbb{E}\Big(\int_{0}^{T}\|X_{s}^{\phi^\varepsilon}\|_{\mathbb{H}_{1}}^{2} d s\Big)\Big]^{\frac{1}{2}} \nonumber\\
\!\!\!\!\!\!\!\!&& +C_{T} \zeta \varepsilon \mathbb{E}\Big[\sup _{t \in[0, T]}\big(1+\|X_{t}^{\phi^\varepsilon}\|_{\mathbb{H}_{1}}^{2}\big)\Big]\nonumber\\
\leq \!\!\!\!\!\!\!\!&&  C_{M, T} \zeta(1+\|x\|_{\mathbb{H}_{1}}^{2}+\|y\|_{\mathbb{H}_{2}}^{2}) .
\end{eqnarray}
For the last term $I_{5}(t)$, the martingale property and Fubini's theorem imply that
\begin{equation}\label{condiff4}
\mathbb{E}\Big[\int_{\zeta}^{T}I_{5}(t)  d t\Big]
=2\sqrt{\varepsilon} \int_{\zeta}^{T}\mathbb{E}\Big[ \int_{t-\zeta}^{t}\langle \mathfrak{B}_{1}(X_{s}^{\phi^\varepsilon}) d W^1_{s}, X_{s}^{\phi^\varepsilon}-X_{t-\zeta}^{\phi^\varepsilon}\rangle_{\mathbb{H}_{1}}\Big]d t=0.
\end{equation}
Substituting (\ref{condiff1})-(\ref{condiff4}) into (\ref{condiffx}), we conclude that
\begin{equation*}
I \leq C_{M, T} \zeta\Big(1+\|x\|_{\mathbb{H}_{1}}^{\frac{(2\beta_1+4)(\alpha_1-1)+2}{\alpha_1}}+\|y\|_{\mathbb{H}_{2}}^{\frac{(2\beta_1+4)(\alpha_1-1)+2}{\alpha_1}}\Big).
\end{equation*}
Following similar arguments, the term $II$ in (\ref{controllDifferent}) is estimated by
\begin{equation*}
II \leq C_{M, T} \zeta\Big(1+\|x\|_{\mathbb{H}_{1}}^{\frac{(2\beta_1+4)(\alpha_1-1)+2}{\alpha_1}}+\|y\|_{\mathbb{H}_{2}}^{\frac{(2\beta_1+4)(\alpha_1-1)+2}{\alpha_1}}\Big).
\end{equation*}
Recalling (\ref{controllDifferent}), the proof is completed.
\end{proof}

\section{Proof of main results}
\setcounter{equation}{0}
 \setcounter{definition}{0}
In this section, we employ the time discretization scheme developed in \cite{Khasminskii1968}, the theory of pseudo-monotone operators and the technique of stopping times to establish  the criterion (a) and (b) in Lemma \ref{WeakLemma}. By doing so, we can get that the LDP holds.

\subsection{Verification of (b)}\label{Sub4.4}
This subsection is devoted to the verification of (b) in Lemma \ref{WeakLemma}, which implies that the rate function $I$ defined in (\ref{ratef}) is a good rate function.

First, we take any sequence $\{\bar{X}^{\phi^{n}}\}_{n\geq1}$ in $K_M,$ which solves (\ref{skletonequation}) and satisfies (\ref{seletonestimate}) in Appendix with $\phi^n \in S_M$ replacing $\phi$.
We can derive the following pre-compact lemma.
\begin{lemma}\label{pre}
The sequence $\{\bar{X}^{\phi^{n}}\}_{n\geq1}$ is pre-compact in $\mathbb{S}_T:=C([0, T], \mathbb{V}_1^{*}) \cap L^{\alpha_1}([0, T], \mathbb{H}_1).$
\end{lemma}
\begin{proof} The proof is divided into following two steps.

\textbf{Step 1}: In this step, we show that $\{\bar{X}^{\phi^{n}}\}_{n\geq1}$ is pre-compact in $C([0, T], \mathbb{V}_1^{*})$. By (\ref{se1}) in Appendix, we can choose a constant $L_0>0$ such that
\begin{equation*}
\sup_{n\geq1}\Big\{\sup\limits_{t \in[0, T]}\|\bar{X}_{t}^{\phi^{n}}\|_{\mathbb{H}_{1}}^{2}+ \displaystyle \int_{0}^{T}\|\bar{X}_{t}^{\phi^{n}}\|_{\mathbb{V}_{1}}^{\alpha_{1}} d t \Big\}=L_0.
\end{equation*}
By Arzela-Ascoli Theorem, it suffices to show that for any $e$ (an orthonormal basis of $\mathbb{V}_1$), $\{\,_{\mathbb{V}_1^*}\langle \bar{X}^{\phi^{n}},e\rangle_{\mathbb{V}_1}\}_{n\geq1}$ is equi-continuous as a family of real-valued functions. For $t,s \in [0,T],$ condition $({\mathbf{A}}{\mathbf{4}})$, $({\mathbf{A}}{\mathbf{5}})$ implies that
\begin{eqnarray*}
\!\!\!\!\!\!\!\!&&\big|\,_{\mathbb{V}_1^*}\langle \bar{X}_t^{\phi^{n}}-\bar{X}_s^{\phi^{n}}, e\rangle_{\mathbb{V}_1}\big|  \\
\leq \!\!\!\!\!\!\!\!&&\int_{s}^{t}\Big(\|\mathfrak{A}_1(\bar{X}_{r}^{\phi^{n}})\|_{\mathbb{V}_1^{*}}\|e\|_{\mathbb{V}_1} +\|\bar{\mathfrak{f}}(\bar{X}_{r}^{\phi^{n}})\|_{\mathbb{H}_1}\|e\|_{\mathbb{H}_1} +\|\mathfrak{B}_1(\bar{X}_{r}^{\phi^{n}})\|_{L_{2}(\mathbb{U}_1,\mathbb{H}_1)}\|\phi^{n}_{r}\|_\mathscr{U} \|e\|_{\mathbb{H}_1}\Big) d r \\
\leq \!\!\!\!\!\!\!\!&& C\Big(\int_{s}^{t}\|\mathfrak{A}_1(\bar{X}_{r}^{\phi^{n}})\|_{\mathbb{V}^{*}}^{\frac{\alpha_1}{\alpha_1-1}} d r\Big)^{\frac{\alpha_1-1}{\alpha_1}}|t-s|^{\frac{1}{\alpha_1}}+C\int_{s}^{t}(1+\|\bar{X}_{r}^{\phi^{n}}\|_{\mathbb{H}_1})(1+\|\phi^{n}_{r}\|_\mathscr{U})d r\\
\leq\!\!\!\!\!\!\!\!&& C_{L_0,M}\big(|t-s|^{\frac{1}{\alpha_1}}+|t-s|^{\frac{1}{2}}\big),
\end{eqnarray*}
which completes the assertion.

\textbf{Step 2}: In this step, we show that $\{\bar{X}^{\phi^{n}}\}_{n\geq1}$ is pre-compact in $L^{\alpha_1}([0, T], \mathbb{H}_1).$  In view of \cite[Theorem 5]{S}, it suffices to prove
\begin{equation}\label{preall}
\lim _{h \to 0} \sup _{n\geq1} \int_{0}^{T-h}\|\bar{X}_{t+h}^{\phi^n}-\bar{X}_{t}^{\phi^n}\|_{\mathbb{V}_1^{*}}^{\alpha_1} d t=0.
\end{equation}
To this end, we consider two cases
according to the value of $\alpha_1.$

We first consider the case $1< \alpha_1\leq 2.$
By the energy equality, we know
\begin{equation}\label{pre1}
\|\bar{X}_{t+h}^{\phi^n}-\bar{X}_{t}^{\phi^n}\|_{\mathbb{H}_{1}}^{2}=\sum_{i=1}^{3} \mathscr{H} _{i}(t),
\end{equation}
where
\begin{eqnarray*}
\mathscr{H} _{1}(t):= \!\!\!\!\!\!\!\!&& 2 \int_{t}^{t+h} \,_{\mathbb{V}_{1}^{*}}\langle \mathfrak{A}_1(\bar{X}_{r}^{\phi^n}), \bar{X}_{r}^{\phi^n}-\bar{X}_{t}^{\phi^n}\rangle_{\mathbb{V}_{1}} d r;  \nonumber\\
\mathscr{H} _{2}(t):=\!\!\!\!\!\!\!\!&& 2 \int_{t}^{t+h}\langle \bar{\mathfrak{f}}(\bar{X}_{r}^{\phi^n}), \bar{X}_{r}^{\phi^n}-\bar{X}_{t}^{\phi^n}\rangle_{\mathbb{H}_{1}} d r;\nonumber\\
 \mathscr{H} _{3}(t):=\!\!\!\!\!\!\!\!&& 2 \int_{t}^{t+h}\langle \mathfrak{B}_1(\bar{X}_{r}^{\phi^n})P_1\phi^n_{r}, \bar{X}_{r}^{\phi^n}-\bar{X}_{t}^{\phi^n}\rangle_{\mathbb{H}_{1}} d r.
\end{eqnarray*}
Using condition $({\mathbf{A}}{\mathbf{4}})$ and H\"older's inequality to get that
\begin{eqnarray}\label{pre2}
\!\!\!\!\!\!\!\!&&\int_{0}^{T-h}\mathscr{H} _{1}(t)dt\nonumber\\
\leq \!\!\!\!\!\!\!\!&& C \int_{0}^{T-h} \int_{t}^{t+h}\|\mathfrak{A}_1(\bar{X}_{r}^{\phi^n})\|_{\mathbb{V}_{1}^{*}}\|\bar{X}_{r}^{\phi^n}-\bar{X}_{t}^{\phi^n} \|_{\mathbb{V}_{1}}  d r d t \nonumber\\
\leq \!\!\!\!\!\!\!\!&& C\Big(\int_{0}^{T-h} \int_{t}^{t+h}\|\mathfrak{A}_1(\bar{X}_{r}^{\phi^n})\|_{\mathbb{V}_{1}^{*}}^{\frac{\alpha_{1}}{\alpha_{1}-1}}  d r d t\Big)^{\frac{\alpha_{1}-1}{\alpha_{1}}} \Big(\int_{0}^{T-h} \int_{t}^{t+h}\|\bar{X}_{r}^{\phi^n}-\bar{X}_{t}^{\phi^n}\|_{\mathbb{V}_{1}}^{\alpha_{1}}  d r d t\Big)^{\frac{1}{\alpha_{1}}} \nonumber\\
\leq \!\!\!\!\!\!\!\!&& C\Big(h\int_{0}^{T}(1+\|\bar{X}_{t}^{\phi^n}\|_{\mathbb{V}_{1}}^{\alpha_{1}})(1+\|\bar{X}_{t}^{\phi^n}\|_{\mathbb{H}_{1}}^{\beta_{1}})  d t\Big)^{\frac{\alpha_{1}-1}{\alpha_{1}}} \Big(h\int_{0}^{T}\|\bar{X}_{t}^{\phi^n}\|_{\mathbb{V}_{1}}^{\alpha_{1}} d t\Big)^{\frac{1}{\alpha_{1}}} \nonumber\\
\leq \!\!\!\!\!\!\!\!&& C_{L_0} h,
\end{eqnarray}
where the third inequality is due to
\begin{eqnarray*}
\!\!\!\!\!\!\!\!&&\int_{0}^{T-h} \int_{t}^{t+h}\|\mathfrak{A}_1(\bar{X}_{r}^{\phi^n})\|_{\mathbb{V}_{1}^{*}}^{\frac{\alpha_{1}}{\alpha_{1}-1}}  d r d t\\
\leq\!\!\!\!\!\!\!\!&& \int_{T-h}^{T} \int_{r-h}^{T-h}\|\mathfrak{A}_1(\bar{X}_{r}^{\phi^n})\|_{\mathbb{V}_{1}^{*}}^{\frac{\alpha_{1}}{\alpha_{1}-1}}dtdr+\int_{h}^{T-h} \int_{r-h}^{r}\|\mathfrak{A}_1(\bar{X}_{r}^{\phi^n})\|_{\mathbb{V}_{1}^{*}}^{\frac{\alpha_{1}}{\alpha_{1}-1}}dtdr\\
\!\!\!\!\!\!\!\!&&+ \int_{0}^{h} \int_{0}^{r}\|\mathfrak{A}_1(\bar{X}_{r}^{\phi^n})\|_{\mathbb{V}_{1}^{*}}^{\frac{\alpha_{1}}{\alpha_{1}-1}}dtdr\\
\leq\!\!\!\!\!\!\!\!&&h\int_{0}^{T} \|\mathfrak{A}_1(\bar{X}_{t}^{\phi^n})\|_{\mathbb{V}_{1}^{*}}^{\frac{\alpha_{1}}{\alpha_{1}-1}}dt.
\end{eqnarray*}
Similarly, by H\"older's inequality and the Lipschitz continuity of $\bar{\mathfrak{f}}$,
\begin{eqnarray}\label{pre3}
\!\!\!\!\!\!\!\!&& \int_{0}^{T-h}\mathscr{H} _{2}(t)dt \nonumber\\
\leq \!\!\!\!\!\!\!\!&& C\Big(\int_{0}^{T-h} \int_{t}^{t+h}\|\bar{\mathfrak{f}}(\bar{X}_{r}^{\phi^n})\|_{\mathbb{H}_{1}}^{2}d r d t\Big)^{\frac{1}{2}}
\Big(\int_{0}^{T-h} \int_{t}^{t+h}\|\bar{X}_{r}^{\phi^n}-\bar{X}_{t}^{\phi^n}\|_{\mathbb{H}_{1}}^{2}  d r d t\Big)^{\frac{1}{2}} \nonumber\\
\leq \!\!\!\!\!\!\!\!&& C\Big(h\int_{0}^{T}(1+\|\bar{X}_{t}^{\phi^n}\|_{\mathbb{H}_{1}}^{2}) d t\Big)^{\frac{1}{2}}\Big(h\int_{0}^{T}\|\bar{X}_{t}^{\phi^n}\|_{\mathbb{H}_{1}}^{2} d t\Big)^{\frac{1}{2}} \nonumber\\
\leq \!\!\!\!\!\!\!\!&& C_{L_0} h,
\end{eqnarray}
and by condition $({\mathbf{A}}{\mathbf{5}})$,
\begin{eqnarray}\label{pre4}
\!\!\!\!\!\!\!\!&& \int_{0}^{T-h}\mathscr{H} _{3}(t)dt \nonumber\\
\leq \!\!\!\!\!\!\!\!&& C\Big(\int_{0}^{T-h} \int_{t}^{t+h}\|\mathfrak{B}_1(\bar{X}_{r}^{\phi^n})\|_{L_{2}(\mathbb{U}_1, \mathbb{H}_{1})}^{2}\|\phi_{r}^{n}\|_{\mathscr{U}}^{2}  d r d t\Big)^{\frac{1}{2}} \nonumber\\
\!\!\!\!\!\!\!\!&& \cdot\Big(\int_{0}^{T-h} \int_{t}^{t+h}\|\bar{X}_{r}^{\phi^n}-\bar{X}_{t}^{\phi^n}\|_{\mathbb{H}_{1}}^{2}  d r d t\Big)^{\frac{1}{2}} \nonumber\\
\leq \!\!\!\!\!\!\!\!&& C\Big(h\int_{0}^{T}(1+\|\bar{X}_{t}^{\phi^n}\|_{\mathbb{H}_{1}}^{2})\|\phi_{t}^{n}\|_{\mathscr{U}}^{2}  dt\Big)^{\frac{1}{2}}\Big(h\int_{0}^{T}\|\bar{X}_{t}^{\phi^n}\|_{\mathbb{H}_{1}}^{2} d s\Big)^{\frac{1}{2}} \nonumber\\
\leq \!\!\!\!\!\!\!\!&& C_{L_0} h .
\end{eqnarray}
Recalling (\ref{pre1}), (\ref{pre2})-(\ref{pre4}) yields that
\begin{equation*}
\int_{0}^{T-h}\|\bar{X}_{t+h}^{\phi^n}-\bar{X}_{t}^{\phi^n}\|_{\mathbb{V}_1^{*}}^{\alpha_1}\leq C\int_{0}^{T-h}\|\bar{X}_{t+h}^{\phi^n}-\bar{X}_{t}^{\phi^n}\|_{\mathbb{H}_1^{*}}^{2} d t\leq C_{L_0} h,
\end{equation*}
which implies (\ref{preall}) hold for the case $1< \alpha_1\leq 2.$

Next, we consider the remaining case $\alpha_1>2.$ It is easy to get
\begin{equation}\label{pre5}
\|\bar{X}_{t+h}^{\phi^n}-\bar{X}_{t}^{\phi^n}\|_{\mathbb{H}_{1}}^{\alpha_1}=\sum_{i=1}^{3} \mathscr{K} _{i}(t),
\end{equation}
where
\begin{eqnarray*}
\mathscr{K} _{1}(t):= \!\!\!\!\!\!\!\!&& \alpha_1 \int_{t}^{t+h} \|\bar{X}_{r}^{\phi^n}-\bar{X}_{t}^{\phi^n}\|_{\mathbb{H}_{1}}^{\alpha_1-2} \,_{\mathbb{V}_{1}^{*}}\langle \mathfrak{A}_1(\bar{X}_{r}^{\phi^n}), \bar{X}_{r}^{\phi^n}-\bar{X}_{t}^{\phi^n}\rangle_{\mathbb{V}_{1}} d r;  \nonumber\\
\mathscr{K} _{2}(t):=\!\!\!\!\!\!\!\!&& \alpha_1 \int_{t}^{t+h} \|\bar{X}_{r}^{\phi^n}-\bar{X}_{t}^{\phi^n}\|_{\mathbb{H}_{1}}^{\alpha_1-2}\langle \bar{\mathfrak{f}}(\bar{X}_{r}^{\phi^n}), \bar{X}_{r}^{\phi^n}-\bar{X}_{t}^{\phi^n}\rangle_{\mathbb{H}_{1}} d r;\nonumber\\
\mathscr{K} _{3}(t):=\!\!\!\!\!\!\!\!&& \alpha_1 \int_{t}^{t+h} \|\bar{X}_{r}^{\phi^n}-\bar{X}_{t}^{\phi^n}\|_{\mathbb{H}_{1}}^{\alpha_1-2}\langle \mathfrak{B}_1(\bar{X}_{r}^{\phi^n})P_1\phi^n_{r}, \bar{X}_{r}^{\phi^n}-\bar{X}_{t}^{\phi^n}\rangle_{\mathbb{H}_{1}} d r.
\end{eqnarray*}
Following similar arguments of (\ref{pre2})-(\ref{pre4}), we have
\begin{eqnarray}
\int_{0}^{T-h}\mathscr{K} _{1}(t)dt
\leq \!\!\!\!\!\!\!\!&& C_{L_0} \int_{0}^{T-h} \int_{t}^{t+h}\,_{\mathbb{V}_{1}^{*}}\langle \mathfrak{A}_1(\bar{X}_{r}^{\phi^n}), \bar{X}_{r}^{\phi^n}-\bar{X}_{t}^{\phi^n}\rangle_{\mathbb{V}_{1}}  d r d t \nonumber\\
\leq \!\!\!\!\!\!\!\!&&C_{L_0} h,\label{pre6}\\ \int_{0}^{T-h}\mathscr{K} _{2}(t)dt 
\leq \!\!\!\!\!\!\!\!&& C_{L_0}\int_{0}^{T-h} \int_{t}^{t+h}\langle \bar{\mathfrak{f}}(\bar{X}_{r}^{\phi^n}), \bar{X}_{r}^{\phi^n}-\bar{X}_{t}^{\phi^n}\rangle_{\mathbb{H}_{1}}  d r d t \nonumber\\
\leq \!\!\!\!\!\!\!\!&&C_{L_0} h,\label{pre7}\\ \int_{0}^{T-h}\mathscr{K} _{3}(t)dt 
\leq \!\!\!\!\!\!\!\!&& C_{L_0}\int_{0}^{T-h} \int_{t}^{t+h}\langle \mathfrak{B}_1(\bar{X}_{r}^{\phi^n})P_1\phi^n_{r}, \bar{X}_{r}^{\phi^n}-\bar{X}_{t}^{\phi^n}\rangle_{\mathbb{H}_{1}}  d r d t \nonumber\\
\leq \!\!\!\!\!\!\!\!&&C_{L_0}. h,\label{pre8}
\end{eqnarray}
Substituting (\ref{pre6})-(\ref{pre8}) into (\ref{pre5}), we conclude that
\begin{equation*}
\int_{0}^{T-h}\|\bar{X}_{t+h}^{\phi^n}-\bar{X}_{t}^{\phi^n}\|_{\mathbb{V}_1^{*}}^{\alpha_1}\leq C_{L_0} h,
\end{equation*}
which implies (\ref{preall}) hold for the case $\alpha_1>2.$ Thus, we complete the proof of this lemma.
\end{proof}

We now proceed to prove that for any fixed $M>0$, the set $K_{M}:=\big\{\bar{X}^{\phi}: \phi \in S_{M}\big\}$ is compact. The proof process is divided into two steps:

\vspace{1mm}
\textbf{Step 1}: It is worth noting that $S_M$ is a bounded closed subset in $L^2([0,T];\mathscr{U})$, which implies it is weakly
compact. Thus there exists $\phi\in S_M$ such that
\begin{equation*}
\phi^{n}\rightharpoonup\phi~~\text{in}~~ L^{2}([0, T] ; \mathscr{U}).
\end{equation*}
Here taking a
subsequence if necessary. Moreover, by conditions $({\mathbf{A}}{\mathbf{4}})$, $({\mathbf{A}}{\mathbf{5}})$, Lemma \ref{pre}, the Lipschitz continuity of $\bar{\mathfrak{f}}$ and (\ref{seletonestimate}), there exist a subsequence $(n_k)_{k\geq 1}$ (denoted again by $(n)_{n\geq 1}$ for simplicity), $X\in\mathbb{S}_T$, $\mathscr{A}(\cdot) \in L^{\frac{\alpha_1}{\alpha_1-1}}([0, T], \mathbb{V}_1^{*})$  such that
\begin{eqnarray}
\!\!\!\!\!\!\!\!&&\bar{X}^{\phi^{n}}_{\cdot} \to X_{\cdot}~~\text {in }~~ \mathbb{S}_T;\label{conver1}\\
\!\!\!\!\!\!\!\!&&\bar{X}^{\phi^{n}}_{\cdot} \rightharpoonup X_{\cdot}~~\text {in }~~ L^{\alpha_{1}}([0, T] ; \mathbb{V}_{1});\label{conver2}\\
\!\!\!\!\!\!\!\!&&\mathfrak{A}_1(\bar{X}^{\phi^{n}}_{\cdot}) \rightharpoonup \mathscr{A}(\cdot)~~\text { in }~~  L^{\frac{\alpha_1}{\alpha_1-1}}([0, T], \mathbb{V}_1^{*}),\nonumber
\end{eqnarray}
as $n \to \infty.$
Denote
\begin{equation}\label{e5}
\widehat{X}_{t}^{\phi} := x+\int_{0}^{t} \mathscr{A}(s) d s+\int_{0}^{t}\bar{\mathfrak{f}}(X_s)ds+\int_{0}^{t} \mathfrak{B}(X_s)\phi_s d s.
\end{equation}
Note that
\begin{equation}\label{e4}
\bar{X}_{t}^{\phi^n}=\int_0^t\mathfrak{A}_1(\bar{X}_{s}^{\phi^n})ds+\int_0^t\bar{\mathfrak{f}}(\bar{X}_{s}^{\phi^n})ds+\int_0^t\mathfrak{B}_{1}(\bar{X}_{s}^{\phi^n})P_1\phi^n_s ds.
\end{equation}
By (\ref{conver1}), we know that
\begin{equation}\label{conver3}
\bar{X}^{\phi^{n}}_t \to X_t~~\text {in }~~ \mathbb{H}_1,~dt\text{-a.e.}.
\end{equation}
Taking the weak limit in (\ref{e4}) (see also the proof of (\ref{esB})), in light of condition (\ref{growthB1}), (\ref{conver2}) and (\ref{conver3}), it is easy to see
$$\widehat{X}_{t}^{\phi}=X_t,~dt\text{-a.e.}.$$

It remains to prove
\begin{equation}
\mathscr{A}(\cdot)=\mathfrak{A}_1(X_{\cdot}).\label{A11}
\end{equation}
Then, by the uniqueness of solutions to (\ref{skletonequation}), which follows from condition $({\mathbf{A}}{\mathbf{2}})$,  we can deduce that
$$X_t=\bar{X}_t^\phi,~dt\text{-a.e.},$$
which
is the solution of (\ref{skletonequation}).

 According to Remark \ref{re2.2}, the map $\mathbb{V}_1 \ni u \mapsto \mathfrak{A}_1(u) \in \mathbb{V}_1^*$ is pseudo-monotone.  Following the same argument as in the proof of \cite[Lemma 2.16]{RSZ}, we can get the following Claim:

\textbf{Claim}: If
\begin{eqnarray}
\!\!\!\!\!\!\!\!&&\bar{X}_\cdot^{{\phi}^n} \rightharpoonup X ~~\text {in }~~ L^{\alpha_1}([0, T], \mathbb{V}_1); \nonumber\\
\!\!\!\!\!\!\!\!&&\mathfrak{A}_1(\bar{X}_\cdot^{{\phi}^n}) \rightharpoonup \mathscr{A}(\cdot) ~~\text {in }~~ L^\frac{\alpha_1} {\alpha_1-1}([0, T], \mathbb{V}_1^{*});\nonumber\\
\!\!\!\!\!\!\!\!&&\liminf _{n \to \infty} \int_{0}^{T}\,_{\mathbb{V}_1^*}\langle \mathfrak{A}(\bar{X}_t^{{\phi}^n}), \bar{X}_t^{{\phi}^n}\rangle_{\mathbb{V}_1} d t \geq \int_{0}^{T}\,_{\mathbb{V}_1^*}\langle\mathscr{A}(t), X_t\rangle_{\mathbb{V}_1} d t,\label{e6}
\end{eqnarray}
then
\begin{equation*}
\mathscr{A}(t)=\mathfrak{A}_1(\bar{X}_t^{\phi}) ~~\text{for~a.e.}~~ t \in [0, T].
\end{equation*}

Note that by (\ref{conver1}) and  the lower semicontinuity of the norm $\|\cdot\|_{\mathbb{H}_1}$ in $\mathbb{V}_1^*$, it follows that
\begin{equation}\label{e7}
\sup_{t\in[0,T]}\|X_t\|_{\mathbb{H}_1} \leq \liminf _{n \to \infty}\sup_{t\in[0,T]}\|\bar{X}^{\phi^n}_t\|_{\mathbb{H}_1}<\infty.
\end{equation}
Similarly, we also have
\begin{equation}\label{e8}
\int_{0}^{T}\|X_{t}\|_{\mathbb{V}_{1}}^{\alpha_{1}} d t<\infty.
\end{equation}

In light of the energy equalities for (\ref{e5}) and (\ref{e4}), in order to obtain (\ref{e6}), it is sufficient to show
\begin{equation}\label{e9}
\int_{0}^{T}\langle \mathfrak{B}_1(\bar{X}^{\phi^n}_t)P_1\phi_t^n,\bar{X}^{\phi^n}_t\rangle_{\mathbb{H}_1}dt\to\int_{0}^{T}\langle \mathfrak{B}_1(X_t)P_1\phi_t,X_t\rangle_{\mathbb{H}_1}dt.
\end{equation}
Collecting estimates  (\ref{e7}), (\ref{e8}) and (\ref{seletonestimate}), we can choose a constant $N>0$ such that
\begin{equation}\label{e10}
\sup _{n\geq1}\Bigg\{\sup _{t \in[0, T]}(\|\bar{X}_{t}^{\phi^{n}}\|_{\mathbb{H}_{1}}^2+\|X_{t}\|_{\mathbb{H}_{1}}^2)+ \int_{0}^{T}(\|\bar{X}_{t}^{\phi^{n}}\|_{\mathbb{V}_{1}}^{\alpha_{1}}+\|X_{t}\|_{\mathbb{V}_{1}}^{\alpha_{1}}) d t\Bigg\}=N.
\end{equation}
Notice that
\begin{eqnarray}\label{esB}
\!\!\!\!\!\!\!\!&&\int_{0}^{T}\big[\langle \mathfrak{B}_1(\bar{X}^{\phi^n}_t)P_1\phi_t^n,\bar{X}^{\phi^n}_t\rangle_{\mathbb{H}_1}-\langle \mathfrak{B}_1(X_t)P_1\phi_t,X_t\rangle_{\mathbb{H}_1}\big]dt\nonumber\\
=\!\!\!\!\!\!\!\!&&\int_{0}^{T}\langle \big[\mathfrak{B}_1(\bar{X}^{\phi^n}_t)-\mathfrak{B}_1(X_t)\big]P_1\phi_t^n,\bar{X}^{\phi^n}_t\rangle_{\mathbb{H}_1}dt\nonumber\\
\!\!\!\!\!\!\!\!&&+\int_{0}^{T}\langle \mathfrak{B}_1(X_t)P_1\phi_t^n,\bar{X}^{\phi^n}_t-X_t\rangle_{\mathbb{H}_1}dt\nonumber\\
\!\!\!\!\!\!\!\!&&+\int_{0}^{T}\langle \mathfrak{B}_1(X_t)P_1(\phi_t^n-\phi_t),X_t\rangle_{\mathbb{H}_1}dt\nonumber\\
=:\!\!\!\!\!\!\!\!&& \sum_{i=1}^{3}L_i(n).
\end{eqnarray}
By condition $({\mathbf{A}}{\mathbf{5}})$ and H\"older's inequality, we have
\begin{eqnarray*}
L_1(n)\leq\!\!\!\!\!\!\!\!&& C\Big[\int_{0}^{T}\|\bar{X}^{\phi^n}_t\|^2_{\mathbb{H}_1}\|\phi^n_t\|^2_\mathscr{U}dt\Big]^\frac{1}{2} \Big[\int_{0}^{T}\|\bar{X}^{\phi^n}_t-X_t\|^2_{\mathbb{H}_1}dt\Big]^\frac{1}{2}\\
\leq\!\!\!\!\!\!\!\!&&C_{T,M,N}\Big[\int_{0}^{T}\|\bar{X}^{\phi^n}_t-X_t\|^2_{\mathbb{H}_1}dt\Big]^\frac{1}{2}
\end{eqnarray*}
and
\begin{eqnarray*}
L_2(n)\leq\!\!\!\!\!\!\!\!&& C\Big[\int_{0}^{T}(1+\|\bar{X}^{\phi^n}_t\|^2_{\mathbb{H}_1})\|\phi^n_t\|^2_\mathscr{U}dt\Big]^\frac{1}{2} \Big[\int_{0}^{T}\|\bar{X}^{\phi^n}_t-X_t\|^2_{\mathbb{H}_1}dt\Big]^\frac{1}{2}\\
\leq\!\!\!\!\!\!\!\!&&C_{T,M,N}\Big[\int_{0}^{T}\|\bar{X}^{\phi^n}_t-X_t\|^2_{\mathbb{H}_1}dt\Big]^\frac{1}{2},
\end{eqnarray*}
where (\ref{conver3}), (\ref{e9}) and the dominated convergence theorem imply that $L_{1}(n),L_{2}(n)\to0$. In addition, by the weak convergence of $\phi^{n}$ in $L^{2}([0, T]; \mathscr{U}),$ we can easily obtain $L_{3}(n)\to0.$
Thus, (\ref{e9}) is proved. As a result, $X$ is the solution of (\ref{skletonequation}), i.e.,~$X_t=\bar{X}_t^\phi$, $dt$-a.e..

\vspace{1mm}
\textbf{Step 2}: Now, we are left to prove that $\bar{X}^{\phi^{n}}\to\bar{X}^{\phi}$ in $C([0, T] ; \mathbb{H}_{1})$, as $n \to \infty .$

Denote $Z_t^n:=\bar{X}^{\phi^n}_t-\bar{X}^{\phi}_t.$ We recall the energy equality of $Z_t^n$,
\begin{eqnarray*}
\|Z^n_t\|_{\mathbb{H}_1}^2=\!\!\!\!\!\!\!\!&&2\int_{0}^{t}\,_{\mathbb{V}_1^*}\big\langle \mathfrak{A}_1(\bar{X}^{\phi^n}_s)-\mathfrak{A}_1(\bar{X}^{\phi}_s),Z^n_s\big\rangle_{\mathbb{V}_1}ds+2\int_{0}^{t}\big\langle \bar{\mathfrak{f}}(\bar{X}^{\phi^n}_s)-\bar{\mathfrak{f}}(\bar{X}^{\phi}_s),Z^n_s\big\rangle_{\mathbb{H}_1}ds\\
\!\!\!\!\!\!\!\!&&+2\int_{0}^{t}\big\langle \mathfrak{B}_1(\bar{X}^{\phi^n}_s)P_1\phi^n_s-\mathfrak{B}_1(\bar{X}^{\phi}_s)P_1\phi_s,Z^n_s\big\rangle_{\mathbb{H}_1}ds.
\end{eqnarray*}
According to condition $({\mathbf{A}}{\mathbf{2}})$, the Lipschitz continuity of $\bar{\mathfrak{f}}$ and Young's inequality, it is easy to get
\begin{eqnarray*}
\|Z^n_t\|_{\mathbb{H}_1}^2\leq\!\!\!\!\!\!\!\!&& C \int_{0}^{t}(1+\rho(X_{s}^{\phi^{n}})+\eta(\bar{X}_{s}^{\phi})+\|\phi_{s}^n\|_{\mathscr{U}}^2)\|Z^n_s\|^2_{\mathbb{H}_1}ds\\
\!\!\!\!\!\!\!\!&&+\int_{0}^{t}\langle \mathfrak{B}_1(\bar{X}^{\phi}_s)P_1(\phi_s^n-\phi_s),Z^n_s\rangle_{\mathbb{H}_1}ds.
\end{eqnarray*}
Using Gronwall's lemma leads to
\begin{eqnarray*}
\sup _{t \in[0, T]}\|Z_{t}^{n}\|_{\mathbb{H}_{1}}^{2}\leq\!\!\!\!\!\!\!\!&& C_{T,M,N} \sup _{t \in[0, T]}\Big|\int_{0}^{t}\langle \mathfrak{B}_1(\bar{X}^{\phi}_s)P_1(\phi_s^n-\phi_s),Z^n_s\rangle_{\mathbb{H}_1}dt\Big|
\\
\leq\!\!\!\!\!\!\!\!&&C_{T,M,N} \sum_{i=1}^{4} \widetilde{L}_{i}(n),
\end{eqnarray*}
where
\begin{eqnarray*}
\!\!\!\!\!\!\!\!&&\widetilde{L}_{1}(n):=\int_{0}^{T}\big|\langle \mathfrak{B}_1(\bar{X}_{t}^{\phi})P_1(\phi_{t}^{n}-\phi_{t}), Z_{t}^{n}-Z_{t(\zeta)}^{n}\rangle_{\mathbb{H}_{1}}\big| d t, \\
\!\!\!\!\!\!\!\!&&\widetilde{L}_{2}(n):=\int_{0}^{T}\big|\langle(\mathfrak{B}_1(\bar{X}_{t}^{\phi})-\mathfrak{B}_1(\bar{X}_{t(\zeta)}^{\phi}))P_1(\phi_{t}^{n}-\phi_{t}), Z_{t(\zeta)}^{n}\rangle_{\mathbb{H}_{1}}\big| d t, \\
\!\!\!\!\!\!\!\!&&\widetilde{L}_{3}(n):=\sup _{t \in[0, T]}\big|\int_{t(\zeta)}^{t}\langle \mathfrak{B}_1(\bar{X}_{s(\zeta)}^{\phi})P_1(\phi_{s}^{n}-\phi_{s}), Z_{s(\zeta)}^{n}\rangle_{\mathbb{H}_{1}} d s\big|, \\
\!\!\!\!\!\!\!\!&&\widetilde{L}_{4}(n):=\sup _{t \in[0, T]} \sum_{k=0}^{[t / \zeta]-1}\big|\langle \mathfrak{B}_1(\bar{X}_{k \zeta}^{\phi}) \int_{k \zeta}^{(k+1) \zeta}P_1(\phi_{s}^{n}-\phi_{s}) d s, Z_{k \zeta}^{n}\rangle_{\mathbb{H}_{1}}\big| .
\end{eqnarray*}
By Lemma \ref{MonoSkePrioLemm} and H\"older's inequality, we can deduce that
\begin{eqnarray}
\widetilde{L}_{1}(n)\leq \!\!\!\!\!\!\!\!&& \Bigg\{\int_{0}^{T}\|\mathfrak{B}_1(\bar{X}_{t}^{\phi})\|_{L_{2}(\mathbb{U}_1, \mathbb{H}_{1})}^{2}\|P_1\|^2\|\phi_{t}^{n}-\phi_{t}\|_{\mathscr{U}}^{2} d t\Bigg\}^{\frac{1}{2}}\cdot\Bigg\{\int_{0}^{T }\|Z_{t}^{n}-Z_{t(\zeta)}^{n}\|_{\mathbb{H}_{1}}^{2} d t\Bigg\}^{\frac{1}{2}} \nonumber\\
\leq \!\!\!\!\!\!\!\!&& C_{T,M,N} \zeta^{\frac{1}{2}},\label{L21}
\\
\widetilde{L}_{2}(n)\leq\!\!\!\!\!\!\!\!&& C_N\Bigg\{\int_{0}^{T }\|\bar{X}_{t}^{\phi}-\bar{X}_{t(\zeta)}^{\phi}\|_{\mathbb{H}_{1}}^{2}\|Z_{t(\zeta)}^{n}\|_{\mathbb{H}_{1}}^{2} d t\Bigg\}^{\frac{1}{2}}
\cdot\Bigg\{\int_{0}^{T}\|\phi_{t}^{n}-\phi_{t}\|_{\mathscr{U}}^{2} d t\Bigg\}^{\frac{1}{2}} \nonumber\\
\leq \!\!\!\!\!\!\!\!&&C_{T,M,N} \zeta^{\frac{1}{2}},\label{L22}
\\
\widetilde{L}_{3}(n)\leq \!\!\!\!\!\!\!\!&&\zeta^{\frac{1}{2}}\Bigg\{\int_{0}^{T }(1+\|\bar{X}_{s(\zeta)}^{\phi}\|_{\mathbb{H}_{1}}^{2})\|\phi_{s}^{n}-\phi_{s}\|_{\mathscr{U}}^{2} d s\Bigg\}^{\frac{1}{2}}
\cdot\Bigg\{\sup _{s \in[0, T ]}\|X_{s}^{\phi^{n}}-\bar{X}_{s}^{\phi}\|_{\mathbb{H}_{1}}^{2}\Bigg\}^{\frac{1}{2}} \nonumber\\
\leq \!\!\!\!\!\!\!\!&& C_{T,M,N} \zeta^{\frac{1}{2}}.\label{L23}
\end{eqnarray}
In view of the compactness of $\mathfrak{B}_1(\bar{X}_{k \zeta}^{\phi})$ and the weak convergence of $\phi^n_t$, it follows that
\begin{equation}\label{L24}
\mathfrak{B}_1(\bar{X}_{k \zeta}^{\phi}) \int_{k \zeta}^{(k+1) \zeta}P_1(\phi_{s}^{n}-\phi_{s}) d s\to0~~\text {in }~~\mathbb{H}_1,
\end{equation}
as $n\to\infty.$ Then, $\widetilde{L}_{4}(n)\to0$ follows from (\ref{e10}) and (\ref{seletonestimate}). Combining (\ref{L21})-(\ref{L24}) and letting $\zeta\to0,$ one can conclude that $K_M$ is pre-compact in $C([0, T] ; \mathbb{H}_{1}).$ Moreover, the above argument also implies  that $K_M$ is closed. As a result, the verification of (b) is completed.
\hspace{\fill}$\Box$

\subsection{Verification of (a)}\label{Sub4.3}
Since we intend to utilize the time discretization scheme in this part, we first need to introduce the following auxiliary process. Dividing the time interval $[0,T]$ into some subintervals of size $\zeta$ that is a positive number depending on $\delta$, we consider
\begin{eqnarray*}
\left\{\begin{aligned}
&d\hat{Y}^{\varepsilon,\delta}_t=\frac 1\delta \mathfrak{A}_2(X^{\phi^\varepsilon}_{t(\zeta)},\hat{Y}_t^{\varepsilon,\delta})dt+\frac 1{\sqrt{\delta}}\mathfrak{B}_2(X^{\phi^\varepsilon}_{t(\zeta)},\hat{Y}^{\varepsilon,\delta}_t)dW^2_t,\\
&\hat{Y}^{\varepsilon,\delta}_0=y\in \mathbb{H}_2.
\end{aligned}\right.
\end{eqnarray*}
Clearly, for $k\in\mathbb N,\,t\in[k\zeta,(k+1)\zeta\wedge T]$ we have
$$\hat{Y}^{\varepsilon,\delta}_t=\hat{Y}^{\varepsilon,\delta}_{k\zeta}+\frac 1\delta\int_{k\zeta}^t \mathfrak{A}_2(X^{\phi^\varepsilon}_{k\zeta},\hat{Y}_s^{\varepsilon,\delta})ds+\frac 1{\sqrt{\delta}}\int_{k\zeta}^t\mathfrak{B}_2(X^{\phi^\varepsilon}_{k\zeta},\hat{Y}_s^{\varepsilon,\delta})dW^2_s.$$

 The following is an energy estimate of $\hat{Y}^{\varepsilon,\delta}$.
\begin{lemma}\label{auxiliaryestimate}
For any $T>0$, there is a constant $C_T>0$ such that for any $x \in \mathbb{H}_{1},~y \in \mathbb{H}_{2}$ and $\varepsilon, \delta >0$ small enough,
\begin{equation}\label{auxiliaryY}
\sup _{t \in[0, T]} \mathbb{E}\|\widehat{Y}_{t}^{\varepsilon, \delta}\|_{\mathbb{H}_{2}}^{2} \leq C_T(1+\|x\|_{\mathbb{H}_{1}}^{2}+\|y\|_{\mathbb{H}_{2}}^{2}).
\end{equation}
\end{lemma}
\begin{proof}
Applying It\^{o}'s formula yields that
\begin{eqnarray*}
\|\hat{Y}^{\varepsilon,\delta}_t\|_{\mathbb{H}_2}^2
=\!\!\!\!\!\!\!\!&&\|y\|_{\mathbb{H}_2}^2+\frac 2\delta\int_0^t\,_{\mathbb{V}_2^*}\langle \mathfrak{A}_2(X^{\phi^\varepsilon}_{s(\zeta)},\hat{Y}^{\varepsilon,\delta}_s),\hat{Y}^{\varepsilon,\delta}_s\rangle_{\mathbb{V}_2}ds+\frac 1{\delta}\int_0^t\|
\mathfrak{B}_2(X^{\phi^\varepsilon}_{s(\zeta)},\hat{Y}^{\varepsilon,\delta}_s)\|_{L_2(\mathbb{U}_2,\mathbb{H}_2)}^2ds\\
\!\!\!\!\!\!\!\!&&+\frac 2{\sqrt{\delta}}\int_0^t\langle \mathfrak{B}_2(X^{\phi^\varepsilon}_{s(\zeta)},\hat{Y}^{\varepsilon,\delta}_s)dW^2_s,\hat{Y}^{\varepsilon,\delta}_s\rangle_{\mathbb{H}_2}.
\end{eqnarray*}
Then taking expectation and by (\ref{e1}), we have
\begin{eqnarray*}\label{aY0}
\frac{d}{dt} \mathbb{E}\|\hat{Y}^{\varepsilon,\delta}_t\|_{\mathbb{H}_2}^2=\!\!\!\!\!\!\!\!&&\frac 1\delta\mathbb{E}\Big[2\,_{\mathbb{V}_2^*}\langle \mathfrak{A}_2(X^{\phi^\varepsilon}_{t(\zeta)},\hat{Y}^{\varepsilon,\delta}_t),\hat{Y}^{\varepsilon,\delta}_t\rangle_{\mathbb{V}_2}
+\|\mathfrak{B}_2(X^{\phi^\varepsilon}_{t(\zeta)},\hat{Y}^{\varepsilon,\delta}_t)\|_{L_2(\mathbb{U}_2,\mathbb{H}_2)}^2\Big]
\\
\leq\!\!\!\!\!\!\!\!&&-\frac \lambda\delta\mathbb{E}\|\hat{Y}^{\varepsilon,\delta}_t\|_{\mathbb{H}_2}^2
+\frac C \delta\big(1+\mathbb{E}\|X^{\phi^\varepsilon}_{t(\zeta)}\|_{\mathbb{H}_1}^2\big).
\end{eqnarray*}
By Gronwall's lemma, we have
\begin{equation*}
\mathbb{E}\|\hat{Y}^{\varepsilon,\delta}_t\|_{\mathbb{H}_2}^2 \leq e^{-\frac \lambda\delta t}\|y\|_{\mathbb{H}_2}^{2}+\frac{C}{\delta} \int_{0}^{t} e^{-\frac\lambda\delta(t-s)}\Big(1+\mathbb{E}\|X^{\phi^\varepsilon}_{s(\zeta)}\|_{\mathbb{H}_1}^2\Big)ds.
\end{equation*}
Using Lemma \ref{controlestimate}, we can get
\begin{equation*}
\sup _{t \in[0, T]} \mathbb{E}\|\widehat{Y}_{t}^{\varepsilon, \delta}\|_{\mathbb{H}_{2}}^{2} \leq C_{T}(1+\|x\|_{\mathbb{H}_{1}}^{2}+\|y\|_{\mathbb{H}_{2}}^{2}).
\end{equation*}
We complete the proof.
\end{proof}

We provide the difference between the fast component $Y^{\phi^\varepsilon}$ of controlled equation (\ref{controlequation}) and the auxiliary process $\hat{Y}^{\varepsilon,\delta}$.
\begin{lemma}\label{FastAuxiliaryDiffLemma}
For $T>0$, $\phi^\varepsilon \in\mathcal A_M,M<\infty$ and $x\in \mathbb{H}_1,\,y\in \mathbb{H}_2$, there is a constant $C_{x,y,T,M}>0$ such that for any $\varepsilon,\delta>0$ small enough,
\begin{equation*}
\mathbb E\int_0^{T}\|Y^{\phi^\varepsilon}_t-\hat{Y}^{\varepsilon,\delta}_t\|_{\mathbb{H}_2}^2dt\leq C_{x,y,T,M}\Big(\zeta+\frac \delta\varepsilon\Big),
\end{equation*}
\end{lemma}
\begin{proof}
Note  that
\begin{eqnarray*}
Y^{\phi^\varepsilon}_t-\hat{Y}^{\varepsilon,\delta}_t=\!\!\!\!\!\!\!\!&&\frac 1\delta\int_0^t\big[\mathfrak{A}_2(X^{\phi^\varepsilon}_s,Y^{\phi^\varepsilon}_s)-
\mathfrak{A}_2(X^{\phi^\varepsilon}_{s(\zeta)},\hat{Y}^{\varepsilon,\delta}_s)\big]ds\\
\!\!\!\!\!\!\!\!&&+\frac 1{\sqrt{\varepsilon\delta}}\int_0^t\mathfrak{B}_2(X^{\phi^\varepsilon}_s,Y^{\phi^\varepsilon}_s)P_2\phi^{\varepsilon}_sds \\
\!\!\!\!\!\!\!\!&&+\frac 1{\sqrt{\delta}}\int_0^t\big[\mathfrak{B}_2(X^{\phi^\varepsilon}_s,Y^{\phi^\varepsilon}_s)-
\mathfrak{B}_2(X^{\phi^\varepsilon}_{s(\zeta)},\hat{Y}^{\varepsilon,\delta}_s)\big]dW^2_s.
\end{eqnarray*}
Applying It\^{o}'s formula yields that
\begin{eqnarray}\label{Adiff}
\!\!\!\!\!\!\!\!&&\|Y^{\phi^\varepsilon}_t-\hat{Y}^{\varepsilon,\delta}_t\|_{\mathbb{H}_2}^2\nonumber\\
=\!\!\!\!\!\!\!\!&&\frac 2\delta\int_0^t\,_{\mathbb{V}_2^*}\langle \mathfrak{A}_2(X^{\phi^\varepsilon}_s,Y^{\phi^\varepsilon}_s)-
\mathfrak{A}_2(X^{\phi^\varepsilon}_{s(\zeta)},\hat{Y}^{\varepsilon,\delta}_s),Y^{\phi^\varepsilon}_s-\hat{Y}^{\varepsilon,\delta}_s\rangle_{\mathbb{V}_2}ds\nonumber\\
\!\!\!\!\!\!\!\!&&+\frac 2{\sqrt{\varepsilon\delta}}\int_0^t\langle \mathfrak{B}_2(X^{\phi^\varepsilon}_s,Y^{\phi^\varepsilon}_s)P_2\phi^{\varepsilon}_s,
Y^{\phi^\varepsilon}_s-\hat{Y}^{\varepsilon,\delta}_s\rangle_{\mathbb{H}_2}ds \nonumber\\
\!\!\!\!\!\!\!\!&&+\frac 1{\delta}\int_0^t\|\mathfrak{B}_2(X^{\phi^\varepsilon}_s,Y^{\phi^\varepsilon}_s)-
\mathfrak{B}_2(X^{\phi^\varepsilon}_{s(\zeta)},\hat{Y}^{\varepsilon,\delta}_s)\|_{L_2(\mathbb{U}_2,\mathbb{H}_2)}^2ds\nonumber\\
\!\!\!\!\!\!\!\!&&+\frac 2{\sqrt{\delta}}\int_0^t\langle [\mathfrak{B}_2(X^{\phi^\varepsilon}_s,Y^{\phi^\varepsilon}_s)-
\mathfrak{B}_2(X^{\phi^\varepsilon}_{s(\zeta)},\hat{Y}^{\varepsilon,\delta}_s)]dW^2_s,Y^{\phi^\varepsilon}_s-\hat{Y}^{\varepsilon,\delta}_s\rangle_{\mathbb{H}_2}.
\nonumber\\
=:\!\!\!\!\!\!\!\!&&I(t)+II(t)+III(t)+IV(t).
\end{eqnarray}
By conditions $({\mathbf{H}}{\mathbf{2}})$, $({\mathbf{H}}{\mathbf{4}})$ and Young's inequality, it follows that
\begin{equation}\label{Adiff1}
I(t)+III(t)\leq\frac 1{\delta}\int_0^t\Big(-\kappa\|Y^{\phi^\varepsilon}_s-\hat{Y}^{\varepsilon,\delta}_s\|^2_{\mathbb{H}_2}+C\|X^{\phi^\varepsilon}_s-X^{\phi^\varepsilon}_{s(\zeta)}\|^2_{\mathbb{H}_1}\Big)ds
\end{equation}
and
\begin{eqnarray}\label{Adiff2}
II(t)\leq \!\!\!\!\!\!\!\!&& \frac 2{\sqrt{\varepsilon\delta}} \int_0^t\|\mathfrak{B}_2(X^{\phi^\varepsilon}_s,Y^{\phi^\varepsilon}_s)\|_{L_2(\mathbb{U}_2,\mathbb{H}_2)}\|P_2\|\|\phi_{s}^{\varepsilon}\|_{\mathscr{U}}\|Y^{\phi^\varepsilon}_s-\hat{Y}^{\varepsilon,\delta}_s\|_{\mathbb{H}_2}ds\nonumber\\
\leq \!\!\!\!\!\!\!\!&& \frac{\gamma}{\delta} \int_0^t\|Y^{\phi^\varepsilon}_s-\hat{Y}^{\varepsilon,\delta}_s\|_{\mathbb{H}_2}^2ds+\frac{C}{\varepsilon} \int_0^t(1+\|X_{s}^{\phi^\varepsilon}\|_{\mathbb{H}_{1}}^{2})\|\phi_{s}^{\varepsilon}\|_{\mathscr{U}}^{2}ds,
\end{eqnarray}
where we choose $\gamma\in(0,\kappa).$
Substituting (\ref{Adiff1}) and (\ref{Adiff2}) into (\ref{Adiff}) leads to
\begin{eqnarray*}
\!\!\!\!\!\!\!\!&& \|Y^{\phi^\varepsilon}_t-\hat{Y}^{\varepsilon,\delta}_t\|_{\mathbb{H}_2}^2\\
\leq \!\!\!\!\!\!\!\!&& -\frac{\kappa-\gamma}{\delta} \int_0^t\|Y^{\phi^\varepsilon}_s-\hat{Y}^{\varepsilon,\delta}_s\|_{\mathbb{H}_2}^2ds+\frac{C}{\delta}\int_0^t\|X^{\phi^\varepsilon}_s-X^{\phi^\varepsilon}_{s(\zeta)}\|^2_{\mathbb{H}_1}ds
\\
\!\!\!\!\!\!\!\!&& +\frac 2{\sqrt{\delta}}\int_0^t\langle [\mathfrak{B}_2(X^{\phi^\varepsilon}_s,Y^{\phi^\varepsilon}_s)-
\mathfrak{B}_2(X^{\phi^\varepsilon}_{s(\zeta)},\hat{Y}^{\varepsilon,\delta}_s)]dW^2_s,Y^{\phi^\varepsilon}_s-\hat{Y}^{\varepsilon,\delta}_s\rangle_{\mathbb{H}_2}
\\
\!\!\!\!\!\!\!\!&&
+\frac{C}{\varepsilon} \int_0^t(1+\|X_{s}^{\phi^\varepsilon}\|_{\mathbb{H}_{1}}^{2})\|\phi_{s}^{\varepsilon}\|_{\mathscr{U}}^{2}ds.
\end{eqnarray*}
Taking expectation yields that
\begin{eqnarray*}
\!\!\!\!\!\!\!\!&& \frac{\kappa-\gamma}{\delta} \mathbb E\Big[\int_0^t\|Y^{\phi^\varepsilon}_s-\hat{Y}^{\varepsilon,\delta}_s\|_{\mathbb{H}_2}^2ds\Big]\\
\leq \!\!\!\!\!\!\!\!&& \frac{C}{\delta}\mathbb E\Big[\int_0^t\|X^{\phi^\varepsilon}_s-X^{\phi^\varepsilon}_{s(\zeta)}\|^2_{\mathbb{H}_1}ds\Big]+\frac{C}{\varepsilon} \mathbb E\Big[\int_0^t(1+\|X_{s}^{\phi^\varepsilon}\|_{\mathbb{H}_{1}}^{2})\|\phi_{s}^{\varepsilon}\|_{\mathscr{U}}^{2}ds\Big].
\end{eqnarray*}
Multiplying $\frac{\delta}{\kappa-\gamma}$ on both sides of the above inequality,
\begin{eqnarray*}
\!\!\!\!\!\!\!\!&&\mathbb E\Big[\int_0^T\|Y^{\phi^\varepsilon}_t-\hat{Y}^{\varepsilon,\delta}_t\|_{\mathbb{H}_2}^2dt\Big]\\
\leq \!\!\!\!\!\!\!\!&&C\mathbb E\Big[\int_0^T\|X^{\phi^\varepsilon}_t-X^{\phi^\varepsilon}_{t(\zeta)}\|^2_{\mathbb{H}_1}dt\Big]+C\Big(\frac{\delta}{\varepsilon}\Big) \mathbb E\Big[\int_0^T(1+\|X_{t}^{\phi^\varepsilon}\|_{\mathbb{H}_{1}}^{2})\|\phi_{t}^{\varepsilon}\|_{\mathscr{U}}^{2}dt\Big] \\
\leq \!\!\!\!\!\!\!\!&&C_{T,M}\Big(1+\|x\|_{\mathbb{H}_1}^{\frac{(2\beta_1+4)(\alpha_1-1)+2}{\alpha_1}}+\|y\|_{\mathbb{H}_2}^{\frac{(2\beta_1+4)(\alpha_1-1)+2}{\alpha_1}}\Big)\Big(\zeta+\frac \delta\varepsilon\Big),
\end{eqnarray*}
where the last inequality is due to Lemmas \ref{controlestimate}-\ref{condiff} and $\phi^\varepsilon \in\mathcal A_M.$
\end{proof}

Set a stopping time
$$\tau_{R}^{\varepsilon}:=\inf \Bigg\{t \in[0, T]:\|X_{t}^{\phi^{\varepsilon}}\|_{\mathbb{H}_{1}}+\|\bar{X}_{t}^{\phi}\|_{\mathbb{H}_{1}}+\int_{0}^{t}\|X_{s}^{\phi^{\varepsilon}}\|_{\mathbb{V}_{1}}^{\alpha_{1}}ds+\int_{0}^{t}\|\bar{X}_{s}^{\phi}\|_{\mathbb{V}_{1}}^{\alpha_{1}} d s>R\Bigg\}\wedge T,$$
with the convention $\inf\emptyset=\infty$.

Now, we proceed to prove the criterion (a), namely the convergence of  solutions of controlled equations (\ref{controlequation}) to the solution of skeleton equation (\ref{skletonequation}) in distribution, as $\varepsilon \to 0$. The proof  is divided into three steps:

\vspace{1mm}
\textbf{Step 1}: We first assume that $\phi^\varepsilon$ converges to $\phi$, $\mathbb{P}$-a.s., as  $S_M$-valued random element. Set  ${Z}_{t}^{\varepsilon}:=X_{t}^{\phi^{\varepsilon}}-\bar{X}_{t}^{\phi}$ that solves the following SPDE:
\begin{eqnarray*}
\left\{\begin{aligned}
d {Z}_{t}^{\varepsilon}= & {\ \big[\mathfrak{A}_1(X_{t}^{\phi^{\varepsilon}})-\mathfrak{A}_1(\bar{X}_{t}^{\phi})+\mathfrak{f}(X_{t}^{\phi^{\varepsilon}}, Y_{t}^{\phi^{\varepsilon}})-\bar{ \mathfrak{f}}(\bar{X}_{t}^{\phi})\big] d t } \\
& +\big[\mathfrak{B}_1(X_{t}^{\phi^{\varepsilon}}) P_1\phi_{t}^{\varepsilon}-\mathfrak{B}_1(\bar{X}_{t}^{\phi}) P_1\phi_{t}\big] d t+\sqrt{\varepsilon} \mathfrak{B}_1(X_{t}^{\phi^{\varepsilon}}) d W^1_{t}, \\
{Z}_{0}^{\varepsilon}=&~~ 0 .
\end{aligned}\right.
\end{eqnarray*}
Applying It\^{o}'s formula yields that
\begin{eqnarray*}
\!\!\!\!\!\!\!\!&&\|{Z}_{t}^{\varepsilon}\|_{\mathbb{H}_{1}}^{2}\\
=\!\!\!\!\!\!\!\!&&2 \int_{0}^{t}\,_{\mathbb{V}_1^*}\langle \mathfrak{A}_1(X_{s}^{\phi^{\varepsilon}})-\mathfrak{A}_1(\bar{X}_{s}^{\phi}), Z_{s}^{\varepsilon}\rangle_{\mathbb{V}_{1}} d s
+2 \int_{0}^{t}\langle \mathfrak{f}(X_{s}^{\phi^{\varepsilon}}, Y_{s}^{\phi^{\varepsilon}})-\bar{\mathfrak{f}}(\bar{X}_{s}^{\phi}), Z_{s}^{\varepsilon}\rangle_{\mathbb{H}_{1}} d s \\
\!\!\!\!\!\!\!\!&&+2 \int_{0}^{t}\langle \mathfrak{B}_1(X_{s}^{\phi^{\varepsilon}})P_1 \phi_{s}^{\varepsilon}-\mathfrak{B}_1(\bar{X}_{s}^{\phi})P_1 \phi_{s},Z_{s}^{\varepsilon}\rangle_{\mathbb{H}_{1}} d s+\varepsilon \int_{0}^{t}\|\mathfrak{B}_1(X_{s}^{\phi^{\varepsilon}})\|_{L_{2}(\mathbb{U}_1, \mathbb{H}_{1})}^{2} d s \\
\!\!\!\!\!\!\!\!&&+2 \sqrt{\varepsilon} \int_{0}^{t}\langle Z_{s}^{\varepsilon}, \mathfrak{B}_1(X_{s}^{\phi^{\varepsilon}}) d W^1_{s}\rangle_{\mathbb{H}_{1}} .
\end{eqnarray*}
Then we can get that
\begin{eqnarray}\label{aall}
\!\!\!\!\!\!\!\!&&\|Z_{t}^{\varepsilon}\|_{\mathbb{H}_{1}}^{2} \nonumber\\
=\!\!\!\!\!\!\!\!&&2 \int_{0}^{t} \,_{\mathbb{V}_{1}^{*}}\langle \mathfrak{A}_1(X_{s}^{\phi^{\varepsilon}})-\mathfrak{A}_1(\bar{X}_{s}^{\phi}), Z_{s}^{\varepsilon}\rangle_{\mathbb{V}_{1}} d s \nonumber\\
\!\!\!\!\!\!\!\!&&+2 \int_{0}^{t}\langle\bar{\mathfrak{f}}(X_{s}^{\phi^{\varepsilon}})-\bar{\mathfrak{f}}(\bar{X}_{s}^{\phi}), Z_{s}^{\varepsilon}\rangle_{\mathbb{H}_{1}} d s \nonumber\\
\!\!\!\!\!\!\!\!&&+2 \int_{0}^{t}\langle \mathfrak{f}(X_{s}^{\phi^{\varepsilon}}, Y_{s}^{\phi^{\varepsilon}})-\bar{\mathfrak{f}}(X_{s}^{\phi^{\varepsilon}})-\mathfrak{f}(X_{s(\zeta)}^{\phi^{\varepsilon}}, \widehat{Y}_{s}^{\varepsilon, \delta})+\bar{\mathfrak{f}}(X_{s(\zeta)}^{\phi^{\varepsilon}}), Z_{s}^{\varepsilon}\rangle_{\mathbb{H}_{1}} d s \nonumber\\
\!\!\!\!\!\!\!\!&&+2 \int_{0}^{t}\langle \mathfrak{f}(X_{s(\zeta)}^{\phi^{\varepsilon}}, \widehat{Y}_{s}^{\varepsilon, \delta})-\bar{\mathfrak{f}}(X_{s(\zeta)}^{\phi^{\varepsilon}}), Z_{s}^{\varepsilon}-Z_{s(\zeta)}^{\varepsilon}\rangle_{\mathbb{H}_{1}} d s \nonumber\\
\!\!\!\!\!\!\!\!&&+2 \int_{0}^{t}\langle[\mathfrak{B}_1(X_{s}^{\phi^{\varepsilon}})-\mathfrak{B}_1(\bar{X}_{s}^{\phi})]P_1\phi_{s}^{\varepsilon}, Z_{s}^{\varepsilon}\rangle_{\mathbb{H}_{1}} d s \nonumber\\
\!\!\!\!\!\!\!\!&&+2 \int_{0}^{t}\langle \mathfrak{f}(X_{s(\zeta)}^{\phi^{\varepsilon}}, \widehat{Y}_{s}^{\varepsilon, \delta})-\bar{\mathfrak{f}}(X_{s(\zeta)}^{\phi^{\varepsilon}}), Z_{s(\zeta)}^{\varepsilon}\rangle_{\mathbb{H}_{1}} d s \nonumber\\
\!\!\!\!\!\!\!\!&&+2 \int_{0}^{t}\langle \mathfrak{B}_1(\bar{X}_{s}^{\phi})P_1(\phi_{s}^{\varepsilon}-\phi_{s}), Z_{s}^{\varepsilon}\rangle_{\mathbb{H}_{1}} d s \nonumber\\
\!\!\!\!\!\!\!\!&&+\varepsilon \int_{0}^{t}\|\mathfrak{B}_1(X_{s}^{\phi^{\varepsilon}})\|_{L_{2}(\mathbb{U}_1, \mathbb{H}_{1})}^{2} d s+2 \sqrt{\varepsilon} \int_{0}^{t}\langle Z_{s}^{\varepsilon}, \mathfrak{B}_1(X_{s}^{\phi^{\varepsilon}}) d W^1_{s}\rangle_{\mathbb{H}_{1}} \nonumber\\
=:\!\!\!\!\!\!\!\!&& \sum_{i=1}^{9} K_{i}(t) .
\end{eqnarray}
First, we can estimate $K_{i}(t),1\leq i\leq 5$. By condition $({\mathbf{A}}{\mathbf{2}})$ and Young's inequality, it follows that
\begin{eqnarray}\label{a1}
\!\!\!\!\!\!\!\!&&K_{1}(t)+K_{5}(t)+K_{8}(t) \nonumber\\
\leq\!\!\!\!\!\!\!\!&& \int_{0}^{t} \Big(2_{\mathbb{V}_{1}^{*}}\langle \mathfrak{A}_1(X_{s}^{\phi^{\varepsilon}})-\mathfrak{A}_1(\bar{X}_{s}^{\phi}), Z_{s}^{\varepsilon}\rangle_{\mathbb{V}_{1}}+\|\mathfrak{B}_1(X_{s}^{\phi^{\varepsilon}})-\mathfrak{B}_1(\bar{X}_{s}^{\phi})\|_{L_{2}(\mathbb{U}_1, \mathbb{H}_{1})}^{2}\Big) d s \nonumber\\ \!\!\!\!\!\!\!\!&&+\int_{0}^{t}\|P_1\|^2\|\phi_{s}^{\varepsilon}\|_{\mathscr{U}}^{2}\|Z_{s}^{\varepsilon}\|_{\mathbb{H}_{1}}^{2} d s+\varepsilon\int_{0}^{t}(1+\|X_{s}^{\phi^{\varepsilon}}\|_{\mathbb{H}_{1}}^2)ds \nonumber\\
\leq\!\!\!\!\!\!\!\!&&\int_{0}^{t}(C+\rho(X_{s}^{\phi^{\varepsilon}})+\eta(\bar{X}_{s}^{\phi}))\|Z_{s}^{\varepsilon}\|_{\mathbb{H}_{1}}^{2} d s+\int_{0}^{t}\|\phi_{s}^{\varepsilon}\|_{\mathscr{U}}^{2}\|Z_{s}^{\varepsilon}\|_{\mathbb{H}_{1}}^{2} d s+C_{x,y,T,M}\varepsilon .
\end{eqnarray}
By the Lipschitz continuity of $\mathfrak{f}$ and $\bar{\mathfrak{f}}$ (see Step 1 in the proof of Lemma \ref{skletonproof} in Appendix) and Young's inequality, we have
\begin{eqnarray}\label{a2}
\!\!\!\!\!\!\!\!&&K_{2}(t)+K_{3}(t)\nonumber\\
\leq\!\!\!\!\!\!\!\!&&C \int_{0}^{t}\|\bar{\mathfrak{f}}(X_{s}^{\phi^{\varepsilon}})-\bar{\mathfrak{f}}(\bar{X}_{s}^{\phi})\|_{\mathbb{H}_{1}}\|Z_{s}^{\varepsilon}\|_{\mathbb{H}_{1}} d s\nonumber\\
\!\!\!\!\!\!\!\!&&+C \int_{0}^{t}\|\mathfrak{f}(X_{s}^{\phi^{\varepsilon}}, Y_{s}^{\phi^{\varepsilon}})-\bar{\mathfrak{f}}(X_{s}^{\phi^{\varepsilon}})-\mathfrak{f}(X_{s(\zeta)}^{\phi^{\varepsilon}}, \widehat{Y}_{s}^{\varepsilon, \delta})+\bar{\mathfrak{f}}(X_{s(\zeta)}^{\phi^{\varepsilon}})\|_{\mathbb{H}_{1}}\|Z_{s}^{\varepsilon}\|_{\mathbb{H}_{1}} d s \nonumber\\
\leq\!\!\!\!\!\!\!\!&& C \int_{0}^{t}\|Z_{s}^{\varepsilon}\|_{\mathbb{H}_{1}}^{2} d s+C \int_{0}^{t}(\|X_{s}^{\phi^{\varepsilon}}-X_{s(\zeta)}^{\phi^{\varepsilon}}\|_{\mathbb{H}_{1}}^{2}+\|Y_{s}^{\phi^{\varepsilon}}-\widehat{Y}_{s}^{\varepsilon, \delta}\|_{\mathbb{H}_{2}}^{2}) d s
\nonumber\\
\leq\!\!\!\!\!\!\!\!&& C \int_{0}^{t}\|Z_{s}^{\varepsilon}\|_{\mathbb{H}_{1}}^{2} d s+
C_{x,y,T,M}\Big(\zeta+\frac \delta\varepsilon\Big).
\end{eqnarray}
where we used Lemmas \ref{condiff} and \ref{FastAuxiliaryDiffLemma} in the last step. Using H\"older's inequality, we can get
\begin{eqnarray}\label{a3}
K_{4}(t)\leq\!\!\!\!\!\!\!\!&& C \int_{0}^{t} \|\mathfrak{f}(X_{s(\zeta)}^{\phi^{\varepsilon}}, \widehat{Y}_{s}^{\varepsilon, \delta})-\bar{\mathfrak{f}}(X_{s(\zeta)}^{\phi^{\varepsilon}})\|_{\mathbb{H}_{1}}\|Z_{s}^{\varepsilon}-Z_{s(\zeta)}^{\varepsilon}\|_{\mathbb{H}_{1}} d s \nonumber\\
\leq\!\!\!\!\!\!\!\!&& C\Big[\int_{0}^{t}(\|\mathfrak{f}(X_{s(\zeta)}^{\phi^{\varepsilon}}, \widehat{Y}_{s}^{\varepsilon, \delta})\|_{\mathbb{H}_{1}}^{2}+\|\bar{\mathfrak{f}}(X_{s(\zeta)}^{\phi^{\varepsilon}})\|_{\mathbb{H}_{1}}^{2}) d s\Big]^{\frac{1}{2}} \nonumber\\
\!\!\!\!\!\!\!\!&&\cdot {\Big[\int_{0}^{t}(\|X_{s}^{\phi^{\varepsilon}}-X_{s(\zeta)}^{\phi^{\varepsilon}}\|_{\mathbb{H}_{1}}^{2}+\|\bar{X}_{s}^{\phi}-\bar{X}_{s(\zeta)}^{\phi}\|_{\mathbb{H}_{1}}^{2}) d s\Big]^{\frac{1}{2}} } \nonumber\\
\leq\!\!\!\!\!\!\!\!&& C \Big[\int_{0}^{t}(1+\|X_{s(\zeta)}^{\phi^{\varepsilon}}\|_{\mathbb{H}_{1}}^{2}+\|\widehat{Y}_{s}^{\varepsilon, \delta}\|_{\mathbb{H}_{1}}^{2}) d s\Big]^{\frac{1}{2}} \nonumber\\
\!\!\!\!\!\!\!\!&&\cdot \Big[\int_{0}^{t}(\|X_{s}^{\phi^{\varepsilon}}-X_{s(\zeta)}^{\phi^{\varepsilon}}\|_{\mathbb{H}_{1}}^{2}+\|\bar{X}_{s}^{\phi}-\bar{X}_{s(\zeta)}^{\phi}\|_{\mathbb{H}_{1}}^{2}) d s\Big]^{\frac{1}{2}}
\nonumber\\
\leq\!\!\!\!\!\!\!\!&&
C_{x,y,T,M}\Big(\zeta+\frac \delta\varepsilon\Big)^{\frac{1}{2}}.
\end{eqnarray}
Substituting (\ref{a1})-(\ref{a3}) into (\ref{aall}) yields that
\begin{eqnarray*}
\|Z_{t}^{\varepsilon}\|_{\mathbb{H}_{1}}^{2}
\leq\!\!\!\!\!\!\!\!&& C_{x,y,T,M}\Big[\Big(\zeta+\frac \delta\varepsilon\Big)^{\frac{1}{2}}+\varepsilon\Big]+C \int_{0}^{t}(1+\rho(X_{s}^{\phi^{\varepsilon}})+\eta(\bar{X}_{s}^{\phi})+\|\phi_{s}^{\varepsilon}\|_{\mathscr{U}}^{2})\|Z_{s}^{\varepsilon}\|_{\mathbb{H}_{1}}^{2} d s
\\
\!\!\!\!\!\!\!\!&&+2 \int_{0}^{t}\langle \mathfrak{f}(X_{s(\zeta)}^{\phi^{\varepsilon}}, \widehat{Y}_{s}^{\varepsilon, \delta})-\bar{\mathfrak{f}}(X_{s(\zeta)}^{\phi^{\varepsilon}}), Z_{s(\zeta)}^{\varepsilon}\rangle_{\mathbb{H}_{1}} d s \nonumber\\
\!\!\!\!\!\!\!\!&&+2 \int_{0}^{t}\langle \mathfrak{B}_1(\bar{X}_{s}^{\phi})P_1(\phi_{s}^{\varepsilon}-\phi_{s}), Z_{s}^{\varepsilon}\rangle_{\mathbb{H}_{1}} d s \nonumber\\
\!\!\!\!\!\!\!\!&&+2 \sqrt{\varepsilon} \Big|\int_{0}^{t}\langle Z_{s}^{\varepsilon}, \mathfrak{B}_1(X_{s}^{\phi^{\varepsilon}}) d W^1_{s}\rangle_{\mathbb{H}_{1}}\Big|.
\end{eqnarray*}
Then we can use Gronwall's lemma and the definition of $\tau_{R}^{\varepsilon}$ to get
\begin{eqnarray*}
\!\!\!\!\!\!\!\!&&\sup _{t \in[0, T \wedge \tau_{R}^{\varepsilon}]}\|Z_{t}^{\varepsilon}\|_{\mathbb{H}_{1}}^{2}\\
\leq \!\!\!\!\!\!\!\!&& C_{R,x,y,T,M}\Bigg\{\Big[\Big(\zeta+\frac \delta\varepsilon\Big)^{\frac{1}{2}}+\varepsilon\Big]+\sqrt{\varepsilon} \sup _{t \in[0, T \wedge \tau_{R}^{\varepsilon}]}\Big|\int_{0}^{t}\langle Z_{t}^{\varepsilon}, \mathfrak{B}_1(X_{t}^{\phi^{\varepsilon}}) d W^1_{t}\rangle_{\mathbb{H}_{1}}\Big| \\
\!\!\!\!\!\!\!\!&&+\sup _{t \in[0, T \wedge \tau_{R}^{\varepsilon}]}\Big|\int_{0}^{t}\langle \mathfrak{B}_1(\bar{X}_{s}^{\phi})P_1(\phi_{s}^{\varepsilon}-\phi_{s}), Z_{s}^{\varepsilon}\rangle_{\mathbb{H}_{1}} d s\Big| \\
\!\!\!\!\!\!\!\!&&+\sup _{t \in [0, T \wedge \tau_{R}^{\varepsilon}]}\Big|\int_{0}^{t}\langle \mathfrak{f}(X_{s(\zeta)}^{\phi^{\varepsilon}}, \widehat{Y}_{s}^{\varepsilon, \delta})-\bar{\mathfrak{f}}(X_{s(\zeta)}^{\phi^{\varepsilon}}), Z_{s(\zeta)}^{\varepsilon}\rangle_{\mathbb{H}_{1}} d s\Big|\Bigg\}.
\end{eqnarray*}
Taking expectation  we obtain
\begin{equation}\label{aaall}
\mathbb{E}\Big[\sup _{t \in[0, T \wedge \tau_{R}^{\varepsilon}]}\|Z_{t}^{\varepsilon}\|_{\mathbb{H}_{1}}^{2}\Big]\leq C_{R,x,y,T,M}\Bigg\{\Big[\Big(\zeta+\frac \delta\varepsilon\Big)^{\frac{1}{2}}+\varepsilon\Big]+ \sum_{i=1}^{3} \widetilde{K}_{i}\Bigg\},
\end{equation}
where
\begin{eqnarray*}
\!\!\!\!\!\!\!\!&&\widetilde{K}_{1}:=\sqrt{\varepsilon} \mathbb{E}\Big[\sup _{t \in[0, T \wedge \tau_{R}^{\varepsilon}]}\Big|\int_{0}^{t}\langle Z_{t}^{\varepsilon}, \mathfrak{B}_1(X_{t}^{\phi^{\varepsilon}}) d W^1_{t}\rangle_{\mathbb{H}_{1}}\Big| \Big];
\nonumber\\
\!\!\!\!\!\!\!\!&&\widetilde{K}_{2}:=\mathbb{E}\Big[\sup _{t \in[0, T \wedge \tau_{R}^{\varepsilon}]}\Big|\int_{0}^{t}\langle \mathfrak{B}_1(\bar{X}_{s}^{\phi})P_1(\phi_{s}^{\varepsilon}-\phi_{s}), Z_{s}^{\varepsilon}\rangle_{\mathbb{H}_{1}} d s\Big|\Big];
\nonumber\\
\!\!\!\!\!\!\!\!&&\widetilde{K}_{3}:=\mathbb{E}\Big[\sup _{t \in [0, T \wedge \tau_{R}^{\varepsilon}]}\Big|\int_{0}^{t}\langle \mathfrak{f}(X_{s(\zeta)}^{\phi^{\varepsilon}}, \widehat{Y}_{s}^{\varepsilon, \delta})-\bar{\mathfrak{f}}(X_{s(\zeta)}^{\phi^{\varepsilon}}), Z_{s(\zeta)}^{\varepsilon}\rangle_{\mathbb{H}_{1}} d s\Big|\Big]\Big\}.
\end{eqnarray*}
By Burkholder--Davis--Gundy inequality and Young's inequality,
\begin{eqnarray}\label{K23}
\widetilde{K}_{1}
\leq\!\!\!\!\!\!\!\!&& C \sqrt{\varepsilon} \mathbb{E}\Big(\int_{0}^{T \wedge \tau_{R}^{\varepsilon}}\|Z_{t}^{\varepsilon}\|_{\mathbb{H}_{1}}^{2}\|\mathfrak{B}_1(X_{t}^{\phi^{\varepsilon}})\|_{L_{2}(\mathbb{U}_1, \mathbb{H}_{1})}^{2} d t\Big)^{\frac{1}{2}}
 \nonumber\\
\leq\!\!\!\!\!\!\!\!&& C\sqrt{\varepsilon} \mathbb{E}\Big(\sup _{t \in[0, T \wedge \tau_{R}^{\varepsilon}]}\|\mathfrak{B}_1(X_{t}^{\phi^{\varepsilon}})\|_{L_{2}(\mathbb{U}_1, \mathbb{H}_{1})}^{2} \cdot \int_{0}^{T \wedge \tau_{R}^{\varepsilon}}\|Z_{t}^{\varepsilon}\|_{\mathbb{H}_{1}}^{2} d t\Big)^{\frac{1}{2}}
\nonumber\\
\leq\!\!\!\!\!\!\!\!&& C\varepsilon  \mathbb{E}\Big[\sup _{t \in[0, T \wedge \tau_{R}^{\varepsilon}]}(1+\|X_{t}^{\phi^{\varepsilon}}\|_{\mathbb{H}_{1}}^{2})\Big]+C \mathbb{E} \int_{0}^{T \wedge \tau_{R}^{\varepsilon}}\|Z_{t}^{\varepsilon}\|_{\mathbb{H}_{1}}^{2} d t \nonumber\\
\leq\!\!\!\!\!\!\!\!&&C_{x,y,T,M}\varepsilon+C  \int_{0}^{T}\mathbb{E}\Big[\sup_{s\in[0,t\wedge \tau_{R}^{\varepsilon}]}\|Z_{s}^{\varepsilon}\|_{\mathbb{H}_{1}}^{2}\Big] d t.
\end{eqnarray}
Combining (\ref{aaall})-(\ref{K23}), using Gronwall' lemma again, we derive
\begin{equation}\label{aaall2}
\mathbb{E}\Big[\sup _{t \in[0, T \wedge \tau_{R}^{\varepsilon}]}\|Z_{t}^{\varepsilon}\|_{\mathbb{H}_{1}}^{2}\Big]\leq C_{R,x,y,T,M}\Bigg\{\Big[\Big(\zeta+\frac \delta\varepsilon\Big)^{\frac{1}{2}}+\varepsilon\Big]+\sum_{i=2}^{3} \widetilde{K}_{i}\Bigg\}.
\end{equation}
Assume that the following convergence holds for fixed $R>0$, whose proof is given later,
\begin{equation}\label{aaall3}
\sum_{i=2}^{3} \widetilde{K}_{i}\to0,~\text{as}~\varepsilon\to 0.
\end{equation}
By Chebyshev's inequality and H\"older's inequality, it is clear that for any $\epsilon_0>0$,
\begin{eqnarray}\label{aaall4}
\!\!\!\!\!\!\!\!&& \mathbb{P}\Big\{\Big(\sup _{t \in[0, T]}\|Z_{t}^{\varepsilon}\|_{\mathbb{H}_{1}}^{2}\Big)^{\frac{1}{2}}>\epsilon_{0}\Big\} \nonumber\\
\leq \!\!\!\!\!\!\!\!&& \frac{1}{\epsilon_{0}} \mathbb{E}\Big[\Big(\sup _{t \in[0, T]}\|Z_{t}^{\varepsilon}\|_{\mathbb{H}_{1}}^{2}\Big)^{\frac{1}{2}} \mathbf{1}_{\{T \leq \tau_{R}^{\varepsilon}\}}\Big]+\frac{1}{\epsilon_{0}} \mathbb{E}\Big[\Big(\sup _{t \in[0, T]}\|Z_{t}^{\varepsilon}\|_{\mathbb{H}_{1}}^{2}\Big)^{\frac{1}{2}} \mathbf{1}_{\{T>\tau_{R}^{\varepsilon}\}}\Big] \nonumber\\
\leq \!\!\!\!\!\!\!\!&& \frac{1}{\epsilon_{0}} \mathbb{E}\Big[\Big(\sup _{t \in[0, T]}\|Z_{t}^{\varepsilon}\|_{\mathbb{H}_{1}}^{2}\Big)^{\frac{1}{2}} \mathbf{1}_{\{T \leq \tau_{R}^{\varepsilon}\}}\Big]+\frac{1}{\epsilon_{0}}\Big[\mathbb{E}\Big(\sup _{t \in[0, T]}\|Z_{t}^{\varepsilon}\|_{\mathbb{H}_{1}}^{2}\Big)\Big]^{\frac{1}{2}}\Big[\mathbb{P}\left(T>\tau_{R}^{\varepsilon}\right)\Big]^{\frac{1}{2}} \nonumber\\
\leq \!\!\!\!\!\!\!\!&& \frac{1}{\epsilon_{0}} \mathbb{E}\Big[\Big(\sup _{t \in[0, T]}\|Z_{t}^{\varepsilon}\|_{\mathbb{H}_{1}}^{2}\Big)^{\frac{1}{2}} \mathbf{1}_{\{T \leq \tau_{R}^{\varepsilon}\}}\Big]+\frac{C_{x,y,T,M}}{\epsilon_{0} \sqrt{R}}.
\end{eqnarray}
Combining (\ref{aaall2})-(\ref{aaall4}) implies that
\begin{equation}\label{aaall5}
\lim_{R\to \infty }\limsup_{\varepsilon \to 0} \mathbb{P}\Big\{\Big(\sup _{t \in[0, T]}\|Z_{t}^{\varepsilon}\|_{\mathbb{H}_{1}}^{2}\Big)^{\frac{1}{2}}>\varepsilon_{0}\Big\}
\leq\lim_{R\to \infty }\frac{C_{x,y,T,M}}{\sqrt{R}}=0,
\end{equation}
where we used the fact $\frac \delta\varepsilon\to 0$ and choose a suitable $\zeta$ such that $\zeta\to 0$, as $\varepsilon\to 0$.

Now, we consider the case that $\phi^\varepsilon$ converges in distribution to $\phi$.  Recall that $S_M$ is compact and the distribution of $W$ is tight, thus $(\phi^\varepsilon,W)$ is tight in $S_M\times C([0,T];\mathscr{U}_0)$.  Applying Prokhorov theorem and Skorohod's representation theorem, we deduce that there exist  a subsequence  denoted again by $\{\varepsilon\}$ of $\{\varepsilon\}$, a probability space $(\tilde{\Omega},\tilde{\mathscr{F}},\tilde{\mathbb{P}})$ and, on this space, random variables $(\tilde{\phi}^{\varepsilon},\tilde{W}^{\varepsilon})$ and $(\tilde{\phi},\tilde{W})$ such that
\begin{eqnarray*}
\!\!\!\!\!\!\!\!&&\text{(i)}~~(\tilde{\phi}^{\varepsilon},\tilde{W}^{\varepsilon})~\text{converges~to}~(\tilde{\phi},\tilde{W})~\text{in}~S_M\times C([0,T];\mathscr{U}_0),~\tilde{\mathbb{P}}\text{-a.s.,}~\text{as}~\varepsilon\to 0,
 \nonumber\\
\!\!\!\!\!\!\!\!&&\text{(ii)}~~(\tilde{\phi}^{\varepsilon},\tilde{W}^{\varepsilon})~\text{coincides in distribution with}~(\phi^{\varepsilon},W).
\end{eqnarray*}
In view of (\ref{aaall5}), we can replace $(\phi^{\varepsilon},W)$ by $(\tilde{\phi}^{\varepsilon},\tilde{W}^{\varepsilon})$ and obtain
$$\mathcal{G}^{\varepsilon}\Big(\tilde{W}^{\varepsilon}_{\cdot}+\frac 1{\sqrt{\varepsilon}}\int_0^\cdot\tilde{\phi}^{\varepsilon}_s ds\Big)\to \bar{X}_{s}^{\tilde{\phi}},~\text{in probability}.$$
Based on this, we finally get the desired convergence, i.e.,
$$\mathcal{G}^{\varepsilon}\Big(W_{\cdot}+\frac 1{\sqrt{\varepsilon}}\int_0^\cdot\phi^{\varepsilon}_s ds\Big)\to \bar{X}_{s}^{\phi},~\text{in distribution},$$
which completes the verification of (a).

\vspace{1mm}
\textbf{Step 2}:
In this step, we prove the convergence of the term $\widetilde{K}_{2}$ in (\ref{aaall3}). Notice that
\begin{equation}\label{a2all}
\widetilde{K}_{2} \leq \sum_{i=1}^{3} \widetilde{K}_{2i}+\mathbb{E}(\widehat{K}_{24}),
\end{equation}
where
\begin{eqnarray*}
\!\!\!\!\!\!\!\!&&\widetilde{K}_{21}:=\mathbb{E}\Big[\sup _{t \in[0, T \wedge \tau_{R}^{\varepsilon}]}\Big|\int_{0}^{t}\langle \mathfrak{B}_1(\bar{X}_{s}^{\phi})P_1(\phi_{s}^{\varepsilon}-\phi_{s}), Z_{s}^{\varepsilon}-Z_{s(\zeta)}^{\varepsilon}\rangle_{\mathbb{H}_{1}} d s\Big|\Big], \\
\!\!\!\!\!\!\!\!&&\widetilde{K}_{22}:=\mathbb{E}\Big[\sup _{t \in[0, T \wedge \tau_{R}^{\varepsilon}]}\Big|\int_{0}^{t}\langle(\mathfrak{B}_1(\bar{X}_{s}^{\phi})-\mathfrak{B}_1(\bar{X}_{s(\zeta)}^{\phi}))P_1(\phi_{s}^{\varepsilon}-\phi_{s}), Z_{s(\zeta)}^{\varepsilon}\rangle_{\mathbb{H}_{1}} d s\Big|\Big], \\
\!\!\!\!\!\!\!\!&&\widetilde{K}_{23}:=\mathbb{E}\Big[\sup _{t \in[0, T \wedge \tau_{R}^{\varepsilon}]}\Big|\int_{t(\zeta)}^{t}\langle \mathfrak{B}_1(\bar{X}_{s(\zeta)}^{\phi})P_1(\phi_{s}^{\varepsilon}-\phi_{s}), Z_{s(\zeta)}^{\varepsilon}\rangle_{\mathbb{H}_{1}} d s\Big|\Big], \\
\!\!\!\!\!\!\!\!&&\widehat{K}_{24}:=\sum_{k=0}^{[(T \wedge \tau_{R}^{\varepsilon}) / \zeta]-1}\Big|\langle \mathfrak{B}_1(\bar{X}_{k \zeta}^{\phi}) \int_{k \zeta}^{(k+1) \zeta}P_1(\phi_{s}^{\varepsilon}-\phi_{s}) d s, Z_{k \zeta}^{\varepsilon}\rangle_{\mathbb{H}_{1}}\Big| .
\end{eqnarray*}
For $\widetilde{K}_{21}$ and $\widetilde{K}_{23},$ using condition $({\mathbf{A}}{\mathbf{5}})$ and H\"older's inequality we can get
\begin{eqnarray}\label{a21}
\widetilde{K}_{21} \leq \!\!\!\!\!\!\!\!&& \mathbb{E}\Big[\int_{0}^{T \wedge \tau_{R}^{\varepsilon}}\|\mathfrak{B}_1(\bar{X}_{t}^{\phi})\|_{L_{2}(\mathbb{U}_1, \mathbb{H}_{1})}\|P_1\|\|\phi_{t}^{\varepsilon}-\phi_{t}\|_{\mathscr{U}}\|Z_{t}^{\varepsilon}-Z_{t(\zeta)}^{\varepsilon}\|_{\mathbb{H}_{1}} d t\Big] \nonumber\\
\leq \!\!\!\!\!\!\!\!&& \Big\{\mathbb{E}\Big[\int_{0}^{T \wedge \tau_{R}^{\varepsilon}}\|\mathfrak{B}_1(\bar{X}_{t}^{\phi})\|_{L_{2}(\mathbb{U}_1, \mathbb{H}_{1})}^{2}\|P_1\|^2\|\phi_{t}^{\varepsilon}-\phi_{t}\|_{\mathscr{U}}^{2} d t\Big]\Big\}^{\frac{1}{2}}
\nonumber\\
 \!\!\!\!\!\!\!\!&&\cdot
\Big\{\mathbb{E}\Big[\int_{0}^{T \wedge \tau_{R}^{\varepsilon}}\|Z_{t}^{\varepsilon}-Z_{t(\zeta)}^{\varepsilon}\|_{\mathbb{H}_{1}}^{2} d t\Big]\Big\}^{\frac{1}{2}} \nonumber\\
\leq \!\!\!\!\!\!\!\!&& C\Big\{\mathbb{E}\Big[\int_{0}^{T } (1+\|\bar{X}_{t}^{\phi}\|_{\mathbb{H}_{1}}^{2})\|\phi_{t}^{\varepsilon}-\phi_{t}\|_{\mathscr{U}}^{2} d t\Big]\Big\}^{\frac{1}{2}} \nonumber\\
\!\!\!\!\!\!\!\!&& \cdot\Big\{\mathbb{E}\Big[\int_{0}^{T } (\|\bar{X}_{t}^{\phi}-\bar{X}_{t(\zeta)}^{\phi}\|_{\mathbb{H}_{1}}^{2}+\|X_{t}^{\phi^\varepsilon}-X_{t(\zeta)}^{\phi^\varepsilon}\|_{\mathbb{H}_{1}}^{2}) d t\Big]\Big\}^{\frac{1}{2}} \nonumber\\
\leq \!\!\!\!\!\!\!\!&& C_{x,y,T,M} \zeta^{\frac{1}{2}}\Big\{\mathbb{E}\Big[\int_{0}^{T}\|\phi_{t}^{\varepsilon}-\phi_{t}\|_{\mathscr{U}}^{2} d t\Big]\Big\}^{\frac{1}{2}} \nonumber\\
\leq \!\!\!\!\!\!\!\!&& C_{x,y,T,M} \zeta^{\frac{1}{2}},
\end{eqnarray}
where we used Lemma \ref{condiff} and Lemmas \ref{skletonproof}-\ref{MonoSkePrioLemm} in Appendix.
Similarly,
\begin{eqnarray}\label{a22}
\widetilde{K}_{23} \leq \!\!\!\!\!\!\!\!&&
\Big\{\mathbb{E} \sup _{t \in[0, T \wedge \tau_{R}^{\varepsilon}]}\Big|\int_{t(\zeta)}^{t}\|\mathfrak{B}_1(\bar{X}_{s(\zeta)}^{\phi})\|_{L_{2}(\mathbb{U}_1, \mathbb{H}_{1})}\|P_1\|\|\phi_{s}^{\varepsilon}-\phi_{s}\|_{\mathscr{U}} d s\Big|^{2}\Big\}^{\frac{1}{2}} \nonumber\\
\!\!\!\!\!\!\!\!&&\cdot\Big\{\mathbb{E}\Big[\sup _{t \in[0, T \wedge\tau_{R}^{\varepsilon}]}\|X_{t}^{\phi^{\varepsilon}}-\bar{X}_{t}^{\phi}\|_{\mathbb{H}_{1}}^{2}\Big]\Big\}^{\frac{1}{2}} \nonumber\\
\leq\!\!\!\!\!\!\!\!&& \zeta^{\frac{1}{2}}\Big\{\mathbb{E} \int_{0}^{T }(1+\|\bar{X}_{s(\zeta)}^{\phi}\|_{\mathbb{H}_{1}}^{2})\|\phi_{s}^{\varepsilon}-\phi_{s}\|_{\mathscr{U}}^{2} d s\Big\}^{\frac{1}{2}}\Big\{\mathbb{E}\Big[\sup _{t \in[0, T]}\|X_{t}^{\phi^{\varepsilon}}-\bar{X}_{t}^{\phi}\|_{\mathbb{H}_{1}}^{2}\Big]\Big\}^{\frac{1}{2}} \nonumber\\
\leq\!\!\!\!\!\!\!\!&& C_{x,y,T,M} \zeta^{\frac{1}{2}}
\end{eqnarray}
and
\begin{equation}\label{a23}
\widetilde{K}_{22}\leq C_{x,y,T,M} \zeta^{\frac{1}{2}}.
\end{equation}

Next, we consider the convergence of term $\widehat{K}_{24}.$ Since  $\phi^{\varepsilon}$ converges to $\phi$, $\mathbb{P}$-a.s., as $S_M$-valued random elements. Then  the integral $\int_{a}^{b} \phi^\varepsilon_s ds \rightharpoonup \int_{a}^{b} \phi_s ds$ in $\mathscr{U}$, as $\varepsilon\to 0$.
By using the compactness of $\mathfrak{B}_1(\bar{X}_{k \zeta}^{\phi}),$ we
infer that for any fixed $k$,
$$\Big\|\mathfrak{B}_1(\bar{X}_{k \zeta}^{\phi})\Big(\int_{k \zeta}^{(k+1) \zeta} P_1\phi_{s}^{\varepsilon} d s-\int_{k \zeta}^{(k+1) \zeta} P_1\phi_{s} d s\Big)\Big\|_{\mathbb{H}_{1}} \to 0,~~\quad \text { as } \varepsilon \to 0.$$
Based on the dominated convergence theorem, it follows that for any $R>0$,
\begin{equation}\label{a24}
\mathbb{E}(\widehat{K}_{24}) \to 0, \text { as }\varepsilon \to 0.
\end{equation}
Recalling (\ref{a2all})-(\ref{a24}), we have
\begin{equation}\label{K4}
 \limsup _{\varepsilon \to 0} \widetilde{K}_{2}=0,
\end{equation}
where, as stated in Step 1, we choose a suitable $\zeta$ such that $\zeta\to 0$, as $\varepsilon\to 0$.

\vspace{1mm}
\textbf{Step 3}: In this step, we prove the convergence of the term $\widetilde{K}_{3}$ in (\ref{aaall3}). It's easy to see that
\begin{eqnarray}\label{a3all}
\!\!\!\!\!\!\!\!&&\Big|\int_{0}^{t}\langle \mathfrak{f}(X_{s(\zeta)}^{\phi^{\varepsilon}}, \widehat{Y}_{s}^{\varepsilon, \delta})-\bar{\mathfrak{f}}(X_{s(\zeta)}^{\phi^{\varepsilon}}), Z_{s(\zeta)}^{\varepsilon}\rangle_{\mathbb{H}_{1}} d s\Big| \nonumber\\
\leq\!\!\!\!\!\!\!\!&&\sum_{k=0}^{[t / \zeta]-1}\Big|\int_{k \zeta}^{(k+1) \zeta}\langle \mathfrak{f}(X_{s(\zeta)}^{\phi^{\varepsilon}}, \widehat{Y}_{s}^{\varepsilon, \delta})-\bar{\mathfrak{f}}(X_{s(\zeta)}^{\phi^{\varepsilon}}), Z_{s(\zeta)}^{\varepsilon}\rangle_{\mathbb{H}_{1}} d s\Big| \nonumber\\
\!\!\!\!\!\!\!\!&&+\Big|\int_{t(\zeta)}^{t}\langle \mathfrak{f}(X_{s(\zeta)}^{\phi^{\varepsilon}}, \widehat{Y}_{s}^{\varepsilon, \delta})-\bar{\mathfrak{f}}(X_{s(\zeta)}^{\phi^{\varepsilon}}), Z_{s(\zeta)}^{\varepsilon}\rangle_{\mathbb{H}_{1}} d s\Big| \nonumber\\
=:\!\!\!\!\!\!\!\!&& \widetilde{K}_{31}(t)+\widetilde{K}_{32}(t) .
\end{eqnarray}
For the term $\widetilde{K}_{31}(t),$ we can get
\begin{eqnarray}\label{K51all}
\!\!\!\!\!\!\!\!&&\mathbb{E}\Big[\sup _{t \in[0, T \wedge \tau_{R}^{\varepsilon}]} \widetilde{K}_{31}(t)\Big] \nonumber\\
\leq\!\!\!\!\!\!\!\!&& \mathbb{E} \sum_{k=0}^{[T \wedge \tau_{R}^{\varepsilon} / \zeta]-1}\Big|\int_{k \zeta}^{(k+1) \zeta}\langle \mathfrak{f}(X_{k \zeta}^{\phi^{\varepsilon}}, \widehat{Y}_{s}^{\varepsilon, \delta})-\bar{\mathfrak{f}}(X_{k \zeta}^{\phi^{\varepsilon}}), Z_{k \zeta}^{\varepsilon}\rangle_{\mathbb{H}_{1}} d s\Big| \nonumber\\
\leq\!\!\!\!\!\!\!\!&& \frac{C_{T}}{\zeta} \sup _{0 \leq k \leq[T \wedge \tau_{R}^{\varepsilon} / \zeta]-1} \mathbb{E}\Big|\int_{k \zeta}^{(k+1) \zeta}\langle \mathfrak{f}(X_{k \zeta}^{\phi^{\varepsilon}}, \widehat{Y}_{s}^{\varepsilon, \delta})-\bar{\mathfrak{f}}(X_{k \zeta}^{\phi^{\varepsilon}}), Z_{k \zeta}^{\varepsilon}\rangle_{\mathbb{H}_{1}} d s\Big| \nonumber\\
\leq\!\!\!\!\!\!\!\!&& \frac{C_{T} \delta}{\zeta} \sup _{0 \leq k \leq[T \wedge \tau_{R}^{\varepsilon} / \zeta]-1}\Big(\mathbb{E}\|X_{k \zeta}^{\phi^{\varepsilon}}-\bar{X}_{k \zeta}^{\phi}\|_{\mathbb{H}_{1}}^{2}\Big)^{\frac{1}{2}} \nonumber\\
\!\!\!\!\!\!\!\!&&\cdot\Big(\mathbb{E}\Big\|\int_{0}^{\frac{\zeta}{\delta}} \mathfrak{f}(X_{k \zeta}^{\phi^{\varepsilon}}, \widehat{Y}_{s \delta+k \zeta}^{\varepsilon, \delta})-\bar{\mathfrak{f}}(X_{k \zeta}^{\phi^{\varepsilon}}) d s\Big\|_{\mathbb{H}_{1}}^{2}\Big)^{\frac{1}{2}} \nonumber\\
\leq\!\!\!\!\!\!\!\!&& \frac{C_{x,y,T,M} \delta}{\zeta} \sup _{0 \leq k \leq[T \wedge \tau_{R}^{\varepsilon} / \zeta]-1}\Big(\int_{0}^{\frac{\zeta}{\delta}} \int_{r}^{\frac{\zeta}{\delta}} \Phi_{k}(s, r) d s d r\Big)^{\frac{1}{2}},
\end{eqnarray}
where for each  $0 \leq r \leq s \leq \frac{\zeta}{\delta} ,$
$$\Phi_{k}(s, r):=\mathbb{E}\Big[\langle \mathfrak{f}(X_{k \zeta}^{\phi^{\varepsilon}}, \widehat{Y}_{s \delta+k \zeta}^{\varepsilon, \delta})-\bar{\mathfrak{f}}(X_{k \zeta}^{\phi^{\varepsilon}}), \mathfrak{f}(X_{k \zeta}^{\phi^{\varepsilon}}, \widehat{Y}_{r \delta+k \zeta}^{\varepsilon, \delta})-\bar{\mathfrak{f}}(X_{k \zeta}^{\phi^{\varepsilon}})\rangle_{\mathbb{H}_{1}}\Big].$$
Now, we aim to estimate  $\Phi_{k}(s, r)$. Let us consider the following equation
\begin{eqnarray}\label{widetildeY}
\left\{
\begin{aligned}
&d Y_{t}=\frac{1}{\delta} \mathfrak{A}_{2}(X, Y_{t}) dt+\frac{1}{\sqrt{\delta}} \mathfrak{B}_2(X, Y_{t}) d W^2_{t}, t \geq s, \\
&Y_{s}=Y,
\end{aligned}
\right.
\end{eqnarray}
where $X$ is an $\mathscr{F}_s$-measurable
$\mathbb{H}_1$-valued random variable and $Y$ is an $\mathbb{H}_2$-valued random variable. Denote by $\{\widetilde{Y}_{t}^{\delta, s, X, Y}\}_{t \geq 0}$ the unique solution of (\ref{widetildeY}).
We can see see that for any $k \in \mathbb{N}$ and $t \in[k \zeta,(k+1) \zeta]$,
$$\widehat{Y}_{t}^{\varepsilon, \delta}=\widetilde{Y}_{t}^{\delta, k \zeta, X_{k \zeta}^{\phi^{\varepsilon}}, \widehat{Y}_{k \zeta}^{\varepsilon, \delta}}.$$
Since $X_{k \zeta}^{\phi^{\varepsilon}}$ and $\widehat{Y}_{k \zeta}^{\varepsilon, \delta}$ are $\mathscr{F}_{k \zeta}$-measurable and $\{\widetilde{Y}_{s \delta+k \zeta}^{\delta, k \zeta, x, y}\}_{s \geq 0}$ is independent of $\mathscr{F}_{k \zeta}$ for each fixed $x \in \mathbb{H}_{1}$ and $y \in \mathbb{H}_{2},$ we can get
\begin{eqnarray*}
\Phi_{k}(s, r)
\!\!\!\!\!\!\!\!&&=\mathbb{E}\Big[\langle \mathfrak{f}(X_{k \zeta}^{\phi^{\varepsilon}}, \widetilde{Y}_{s \delta+k \zeta}^{\delta, k \zeta, X_{k \zeta}^{\phi^{\varepsilon}}, \widehat{Y}_{k \zeta}^{\varepsilon, \delta}})-\bar{\mathfrak{f}}(X_{k \zeta}^{\phi^{\varepsilon}}),
\mathfrak{f}(X_{k \zeta}^{\phi^{\varepsilon}}, \widetilde{Y}_{r \delta+k \zeta}^{\delta, k \zeta, X_{k \zeta}^{\phi^{\varepsilon}}, \widehat{Y}_{k \zeta}^{\varepsilon, \delta}})-\bar{\mathfrak{f}}(X_{k \zeta}^{\phi^{\varepsilon}})\rangle_{\mathbb{H}_{1}}\Big]  \\
\!\!\!\!\!\!\!\!&&=\mathbb{E}\Big\{\mathbb { E } [\langle \mathfrak{f}(X_{k \zeta}^{\phi^{\varepsilon}}, \widetilde{Y}_{s \delta+k \zeta}^{\delta, k \zeta, X_{k \zeta}^{\phi^{\varepsilon}}, \widehat{Y}_{k \zeta}^{\varepsilon, \delta}})-\bar{\mathfrak{f}}(X_{k \zeta}^{\phi^{\varepsilon}}),
\mathfrak{f}(X_{k \zeta}^{\phi^{\varepsilon}}, \widetilde{Y}_{r \delta+k \zeta}^{\delta, k \zeta, X_{k \zeta}^{\phi^{\varepsilon}}, \widehat{Y}_{k \zeta}^{\varepsilon, \delta}})-\bar{\mathfrak{f}}(X_{k \zeta}^{\phi^{\varepsilon}})\rangle_{\mathbb{H}_{1}} \mid \mathscr{F}_{k \zeta}]\Big\}  \\
\!\!\!\!\!\!\!\!&&=\mathbb{E}\Big\{\mathbb{E}[\langle \mathfrak{f}(x, \tilde{Y}_{s \delta+k \zeta}^{\delta, k \zeta, x, y})-\bar{\mathfrak{f}}(x), \mathfrak{f}(x, \tilde{Y}_{r \delta+k \zeta}^{\delta, k \zeta, x, y})-\bar{\mathfrak{f}}(x)\rangle_{\mathbb{H}_{1}}]|_{(x, y)=(X_{k \zeta}^{\phi^{\varepsilon}}, \widehat{Y}_{k \zeta}^{\varepsilon, \delta})}\Big\}.
\end{eqnarray*}
Since $\widetilde{Y}_{s \delta+k \zeta}^{\delta, k \zeta, x, y}$ is the solution of (\ref{widetildeY}), it follows that
\begin{eqnarray*}
\widetilde{Y}_{s \delta+k \zeta}^{\delta, k \zeta, x, y} \!\!\!\!\!\!\!\!&& =y+\frac{1}{\delta} \int_{k \zeta}^{s \delta+k \zeta} \mathfrak{A}_{2}(x, \widetilde{Y}_{r}^{\delta, k \zeta, x, y}) d r+\frac{1}{\sqrt{\delta}} \int_{k \zeta}^{s \delta+k \zeta} \mathfrak{B}_2(x, \widetilde{Y}_{r }^{\delta, k \zeta, x, y}) d W^2_{r} \\
\!\!\!\!\!\!\!\!&& =y+\frac{1}{\delta} \int_{0}^{s \delta} \mathfrak{A}_{2}(x, \widetilde{Y}_{r+k \zeta}^{\delta, k \zeta, x, y}) d r+\frac{1}{\sqrt{\delta}} \int_{0}^{s \delta} \mathfrak{B}_2(x, \widetilde{Y}_{r \delta+k \zeta}^{\delta, k \zeta, x, y}) d W_{r}^{k \zeta} \\
\!\!\!\!\!\!\!\!&& =y+\int_{0}^{s} \mathfrak{A}_{2}(x, \widetilde{Y}_{r \delta+k \zeta}^{\delta, k \zeta, x, y}) d r+\int_{0}^{s} \mathfrak{B}_2(x, \widetilde{Y}_{r \delta+k \zeta}^{\delta, k \zeta, x, y}) d \bar{W}_{r}^{k \zeta},
\end{eqnarray*}
where $$W_{r}^{k \zeta}:=W^2_{r+k \zeta}-W^2_{k \zeta},~\bar{W}_{r}^{k \zeta}:=\frac{1}{\sqrt{\delta}} W_{r \delta}^{k \zeta}.$$
The uniqueness of solutions to (\ref{widetildeY}) and to the parameterized equation  (\ref{frozenequation}) in Appendix yields that the distribution of random sequence $\{\widetilde{Y}_{s \delta+k \zeta}^{\delta, k \zeta, x, y}\}_{0 \leq s \leq \frac{\zeta}{\delta}}$ coincides with that of  $\{Y_{s}^{x, y}\}_{0 \leq s \leq \frac{\zeta}{\delta}}.$ Then by Markov and homogeneous properties of process $Y^{x, y},$ we have
\begin{eqnarray*}
\!\!\!\!\!\!\!\!&& \Phi_{k}(s, r) \\
= \!\!\!\!\!\!\!\!&& \mathbb{E}\Big\{\mathbb{E}\Big[\langle \mathfrak{f}(x, Y_{s}^{x,y})-\bar{\mathfrak{f}}(x), \mathfrak{f}(x, Y_{r}^{x,y})-\bar{\mathfrak{f}}(x)\rangle_{\mathbb{H}_{1}}|_{(x,y)=(X_{k \zeta}^{\phi^{\varepsilon}}, \widehat{Y}_{k \zeta}^{\varepsilon, \delta})}\Big]\Big\} \\
= \!\!\!\!\!\!\!\!&& \mathbb{E}\Big\{\mathbb{E}\Big[\big\langle\mathbb{E}[\mathfrak{f}(x, Y_{s}^{x,y})-\bar{\mathfrak{f}}(x)|\mathscr{F}_{r}], \mathfrak{f}(x, Y_{r}^{x, y})-\bar{\mathfrak{f}}(x)\big\rangle_{\mathbb{H}_{1}}|_{(x,y)=(X_{k \zeta}^{\phi^{\varepsilon}}, \widehat{Y}_{k \zeta}^{\varepsilon, \delta})}\Big]\Big\} \\
= \!\!\!\!\!\!\!\!&& \mathbb{E}\Big\{\mathbb{E}\Big[\big\langle\mathbb{E}[\mathfrak{f}(x, Y_{s-r}^{x, z})-\bar{\mathfrak{f}}(x)]|_{\{z=Y_{r}^{x, y}\}}, \mathfrak{f}(x, Y_{r}^{x, y})-\bar{\mathfrak{f}}(x)\big\rangle_{\mathbb{H}_{1}}|_{(x,y)=(X_{k \zeta}^{\phi^{\varepsilon}}, \widehat{Y}_{k \zeta}^{\varepsilon, \delta})}\Big]\Big\}. \\
\end{eqnarray*}
In terms of the exponential ergodicity result for the parameterized equation  (see Lemma \ref{expon} in Appendix), the Lipschitz continuity of $\mathfrak{f},~\bar{\mathfrak{f}}$ and estimates (\ref{conslowestimate}), (\ref{auxiliaryY}) and (\ref{frozenestimate}) in Appendix, it follows that
\begin{eqnarray*}
\!\!\!\!\!\!\!\!&& \Phi_{k}(s, r) \\
\leq \!\!\!\!\!\!\!\!&& \mathbb{E}\Big\{\mathbb E\big\{[1+\|x\|_{\mathbb{H}_{1}}+\|Y_{r}^{x, y}\|_{\mathbb{H}_{2}}] e^{-\frac{(s-r) \kappa}{2}}[1+\|x\|_{\mathbb{H}_{1}}+\|Y_{r}^{x, y}\|_{\mathbb{H}_{2}}]\big\}|_{(x,y)=(X_{k \zeta}^{\phi^{\varepsilon}}, \widehat{Y}_{k \zeta}^{\varepsilon, \delta})}\Big\} \\
\leq \!\!\!\!\!\!\!\!&& C_{T} \mathbb{E}\Big[1+\|X_{k \zeta}^{\phi^{\varepsilon}}\|_{\mathbb{H}_{1}}^{2}+\|\widehat{Y}_{k \zeta}^{\varepsilon, \delta}\|_{\mathbb{H}_{2}}^{2}\Big] e^{-\frac{(s-r) \kappa}{2}} \\
\leq \!\!\!\!\!\!\!\!&& C_{x,y,T,M} e^{-\frac{(s-r) \kappa}{2}},
\end{eqnarray*}
which implies that
\begin{eqnarray}\label{K51}
\!\!\!\!\!\!\!\!&& \mathbb{E}\Big[\sup _{t \in[0, T \wedge \tau_{R}^{\varepsilon}]}\widetilde{K}_{31}(t)\Big] \nonumber\\
\leq \!\!\!\!\!\!\!\!&& C_{x,y,T,M} \frac{\delta}{\zeta}\Big(\int_{0}^{\frac{\zeta}{\delta}} \int_{r}^{\frac{\zeta}{\delta}} e^{-\frac{(s-r) \kappa}{2}} d s d r\Big)^{\frac{1}{2}} \nonumber\\
\leq \!\!\!\!\!\!\!\!&&C_{x,y,T,M} \frac{\delta}{\zeta}\Big(\frac{\zeta}{\delta \kappa}-\frac{1}{\kappa^{2}}+\frac{1}{\kappa^{2}} e^{-\frac{\kappa\zeta}{2 \delta}}\Big)^{\frac{1}{2}} .
\end{eqnarray}
By the Lipschitz continuity of $\mathfrak{f}$ and $\bar{\mathfrak{f}}$,  Lemmas \ref{controlestimate}, \ref{auxiliaryestimate} and \ref{skletonproof}, we also obtain
\begin{eqnarray}\label{K52}
\!\!\!\!\!\!\!\!&& \mathbb{E}\Big[\sup _{t \in[0, T \wedge \tau_{R}^{\varepsilon}]}\widetilde{K}_{32}(t)\Big] \nonumber\\
\leq \!\!\!\!\!\!\!\!&& C{\Big\{\mathbb{E}\Big(\sup _{t \in[0, T \wedge \tau_{R}^{\varepsilon}]}\|X_{t}^{\phi^\varepsilon}-\bar{X}_{t}^{\phi}\|_{\mathbb{H}_{1}}^{2}\Big)\Big\}^{\frac{1}{2}} } \nonumber\\
\!\!\!\!\!\!\!\!&& \cdot\Big\{\mathbb{E}\Big(\sup _{t \in[0, T \wedge \tau_{R}^{\varepsilon}]}\Big|\int_{t(\zeta)}^{t}(1+\|X_{s(\zeta)}^{\phi^\varepsilon}\|_{\mathbb{H}_{1}}+\|\widehat{Y}_{s}^{\varepsilon, \delta}\|_{\mathbb{H}_{2}}) d s\Big|^{2}\Big)\Big\}^{\frac{1}{2}} \nonumber\\
\leq \!\!\!\!\!\!\!\!&& C\zeta^{\frac{1}{2}}\Big\{\mathbb{E}\Big(\sup _{t \in[0, T \wedge \tau_{R}^{\varepsilon}]}\|X_{t}^{\phi^\varepsilon}-\bar{X}_{t}^{\phi}\|_{\mathbb{H}_{1}}^{2}\Big)\Big\}^{\frac{1}{2}} \nonumber\\
\!\!\!\!\!\!\!\!&& \cdot\Big\{\mathbb{E}\Big[\int_{0}^{T \wedge \tau_{R}^{\varepsilon}}\Big(1+\|X_{t(\zeta)}^{\phi^\varepsilon}\|_{\mathbb{H}_{1}}^{2}\Big) d t\Big]+\sup _{t \in[0, T]} \mathbb{E}\|\widehat{Y}_{t}^{\varepsilon, \delta}\|_{\mathbb{H}_{2}}^{2}\Big\}^{\frac{1}{2}}  \nonumber\\
\leq \!\!\!\!\!\!\!\!&& C_{x,y,T,M}\zeta^{\frac{1}{2}}.
\end{eqnarray}
Recalling (\ref{a3all}), (\ref{K51}) and (\ref{K52}) leads to
\begin{equation}\label{K5}
\widetilde{K}_{3}
\leq C_{x,y,T,M}\Big(\frac{\delta}{\zeta}+\frac{\delta^{\frac{1}{2}}}{\zeta^{\frac{1}{2}}}+\zeta^{\frac{1}{2}}\Big).
\end{equation}
Taking $\zeta=\delta^{\frac{1}{2}}$ and combining (\ref{K4}) and (\ref{K5}) completes (\ref{aaall3}).
\hspace{\fill}$\Box$

\section{Application to examples}
\setcounter{equation}{0}
 \setcounter{definition}{0}
In this section, we apply our main results to establish the LDP for multi-scale stochastic partial differential equations. The slow component can involve  such as stochastic porous medium equations, stochastic Cahn-Hilliard equations, stochastic 2D Liquid crystal equations and stochastic tamed 3D Navier-Stokes equations, and the fast component includes the stochastic Allen-Cahn equations. It is important to note that the aforementioned models for the slow component are not covered by the framework proposed in \cite{liu2015stochastic} and \cite{HLL}, but they can be addressed using the fully local monotone framework.

Since we  mainly focus on the nonlinear operator $\mathfrak{A}_1$ within slow component, we take the fast component by stochastic reaction-diffusion type equations with a Lipschitz drift for simplicity. We define the operator $\mathcal{R}(v)$  by
\begin{equation*}
\mathcal{R}(v):=\Delta v+c_1v-c_2v^3,
\end{equation*}
where $c_1,~c_2$ are some non-negative constants.

Let $\Lambda \subset \mathbb{R}^{d}$ be an open bounded domain with a smooth
boundary and so is $\mathcal{O} \subset \mathbb{R}$. For $p \in[1, \infty),$ we use $L^{p}(\Lambda,\mathbb{R}^{d})$ (for simplicity, write it down as $L^{p}(\Lambda)$) to denote the vector valued $L^p$-space with the norm $\|\cdot\|_{L^p}.$ We also use $H^m_0(\Lambda;\mathbb{R}^{n})$ (resp.~$H^m_0(\mathcal{O})$) to denote the classical Sobolev space with Dirichlet boundary taking values in $\mathbb{R}^{n}$ (resp.~$\mathbb{R}$) for a non-negative integer $m \geq 0$, which is equipped with the following norm
\begin{equation*}
\|u\|_{m}:=\Big(\int_{\Lambda}|(I-\Delta)^{\frac{m}{2}} u|^{2} d \xi\Big)^{\frac{1}{2}}.
\end{equation*}

%\begin{eqnarray*}
%\|u\|_{n, p}=\Big(\sum_{0 \leq|\delta| \leq n} \int_{\Lambda}|D^{\delta} u|^{p} d \xi\Big)^{\frac{1}{p}}.
%\end{eqnarray*}

\begin{example}
(Porous media equation with fast oscillations)
\end{example}

Let $d\geq1$. Let $\Psi: \mathbb{R} \to \mathbb{R}$ be a function having the following properties:
\begin{enumerate}
\item [$({\mathbf{\Psi}}{\mathbf{1}})$]
$\Psi$ is continuous.
\item [$({\mathbf{\Psi}}{\mathbf{2}})$]
For all $s,t \in \mathbb{R}$
\begin{equation*}
(t-s)(\Psi(t)-\Psi(s)) \geqslant 0.
\end{equation*}
\item [$({\mathbf{\Psi}}{\mathbf{3}})$]
There exist $m \in[2, \infty), c_{1} \in(0, \infty), c_{2} \in[0, \infty)$  such that for all $s \in \mathbb{R}$
\begin{equation*}
s \Psi(s) \geqslant c_{1}|s|^{m}-c_{2}.
\end{equation*}
\item [$({\mathbf{\Psi}}{\mathbf{4}})$]
There exist $c_{3}, c_{4} \in(0, \infty)$  such that for all $s \in \mathbb{R}$
\begin{equation*}
|\Psi(s)| \leqslant c_{3}|s|^{m-1}+c_{4},
\end{equation*}
where $m$ is as in $({\mathbf{\Psi}}{\mathbf{3}}).$
\end{enumerate}

Using the following Gelfand triple for the slow component
$$
\mathbb{V}_{1}:=L^{p}(\Lambda;\mathbb{R}) \subseteq \mathbb{H}_{1}:=(H_{0}^{1}(\Lambda;\mathbb{R}))^{*} \subseteq \mathbb{V}_{1}^{*},
$$
and the following Gelfand triple for the fast component
$$
\mathbb{V}_{2}:=H_{0}^{1}(\mathcal{O}) \subseteq \mathbb{H}_{2}:=L^{2}(\mathcal{O}) \subseteq \mathbb{V}_{2}^{*}.
$$

Now, we consider the stochastic porous media equation with fast oscillations,
\begin{eqnarray}\label{51}
\left\{ \begin{aligned}
&d X_{t}^{\varepsilon,\delta}=[\Delta \Psi(X_{t}^{\varepsilon,\delta})+\mathfrak{f}(X_{t}^{\varepsilon,\delta}, Y_{t}^{\varepsilon,\delta})] d t+\sqrt{\varepsilon}\mathfrak{B}_1(X_{t}^{\varepsilon,\delta}) d W_{t}^{1}, \\
&d Y_{t}^{\varepsilon,\delta}=\frac{1}{\delta}[\mathcal{R}(Y_{t}^{\varepsilon,\delta})+g(X_{t}^{\varepsilon,\delta}, Y_{t}^{\varepsilon,\delta})] d t+\frac{1}{\sqrt{\delta}} \mathfrak{B}_2(X_{t}^{\varepsilon,\delta}, Y_{t}^{\varepsilon,\delta}) d W_{t}^{2}, \\
&X_{0}^{\varepsilon,\delta}=x \in \mathbb{H}_{1}, Y_{0}^{\varepsilon,\delta}=y \in \mathbb{H}_{2},
\end{aligned} \right.
\end{eqnarray}
where the measurable maps
\begin{equation*}
\mathfrak{f}: \mathbb{H}_{1} \times \mathbb{H}_{2} \to \mathbb{H}_{1},  \mathfrak{B}_1: \mathbb{V}_{1} \to L_{2}(\mathbb{U}_{1}, \mathbb{H}_{1}),
\end{equation*}
\begin{equation*}
g: \mathbb{H}_{1} \times \mathbb{V}_{2} \to \mathbb{V}_{2}^{*}, \mathfrak{B}_2: \mathbb{H}_{1} \times \mathbb{V}_{2} \to L_{2}(\mathbb{U}_{2}, \mathbb{H}_{2})
\end{equation*}
are Lipschitz continuous, i.e., for all $u_1,u_2\in \mathbb{H}_1,~v_1,v_2\in \mathbb{H}_2,~u,v \in \mathbb{V}_1,~w_1,w_2 \in \mathbb{V}_2,$
\begin{eqnarray}\label{511}
\!\!\!\!\!\!\!\!&&\|\mathfrak{f}(u_{1}, v_{1})-\mathfrak{f}(u_{2}, v_{2})\|_{\mathbb{H}_{1}} \leq C(\|u_{1}-u_{2}\|_{\mathbb{H}_{1}}+\|v_{1}-v_{2}\|_{\mathbb{H}_{2}}),\\
\!\!\!\!\!\!\!\!&&\|\mathfrak{B}_1(u)-\mathfrak{B}_1(v)\|_{L_{2}(\mathbb{U}_{1}, \mathbb{H}_{1})}^{2} \leq C\|u-v\|_{\mathbb{H}_{1}}^{2},\\
\label{512}
\!\!\!\!\!\!\!\!&&\|g(u_1,w_1)-g(u_2,w_2)\|_{\mathbb{H}_{2}} \leq C\|u_1-u_2\|_{\mathbb{H}_{1}}+L_{g}\|w_1-w_2\|_{\mathbb{H}_{2}},\\
\label{513}
\!\!\!\!\!\!\!\!&&\|\mathfrak{B}_2(u_1,w_1)-\mathfrak{B}_2(u_2,w_2)\|_{L_2(\mathbb{U}_2,\mathbb{H}_2)}\leq C\|u_1-u_2\|_{\mathbb{H}_1}+L_{\mathfrak{B}_2}\|w_1-w_2\|_{\mathbb{H}_2}.\label{514}
\end{eqnarray}
Moreover, we have
\begin{equation}\label{515}
\sup_{w\in \mathbb{V}_2}\|\mathfrak{B}_2(u_1,w)\|_{L_2(\mathbb{U}_2,\mathbb{H}_2)}\leq C(1+\|u_1\|_{\mathbb{H}_1}),
\end{equation}
and for the smallest eigenvalue $\lambda_{1}$ of $-\Delta$ in $\mathbb{H}_2$, the Lipschitz constants $L_{g},~L_{\mathfrak{B}_2}$ satisfy
\begin{equation}\label{516}
2 \lambda_{1}-2 L_{g}-L_{\mathfrak{B}_{2}}^{2}>0.
\end{equation}
\begin{theorem}\label{t51}
Assume that $\Psi$ satisfies $({\mathbf{\Psi}}{\mathbf{1}})$-$({\mathbf{\Psi}}{\mathbf{4}})$ and (\ref{511})-(\ref{516}) hold. If
\begin{equation*}
\lim _{\varepsilon \to 0} \frac{\delta}{\varepsilon}=0,
\end{equation*}
then  $\{X^{\varepsilon,\delta}:\varepsilon>0\}$ in (\ref{51}) satisfies the LDP on $C([0,T];\mathbb{H}_1)$ with the following good rate function
\begin{equation*}
I(f)=\inf _{\{\phi \in L^{2}([0, T] ; \mathscr{U_0}): f=\mathcal{G}^{0}(\int_{0}^{\cdot} \phi_{s} d s)\}}\left \{\frac{1}{2} \int_{0}^{T}\|\phi_{s}\|_{\mathscr{U_0}}^{2} d s\right \},
\end{equation*}
where the measurable map $\mathcal{G}^0:C([0,T];\mathscr{U}_0)\to C([0,T];\mathbb{H}_1)$ is given by
\begin{eqnarray*}
\mathcal{G}^0\Big(\int_0^\cdot\phi_s ds\Big):=\left\{ \begin{aligned}&\bar{X}^\phi,~~\phi\in L^2([0,T]; \mathscr{U});\\
&0,~~~~~~\text{otherwise}.
\end{aligned} \right.
\end{eqnarray*}
Here $\bar{X}^\phi$ is the  solution of the
skeleton equation
\begin{equation*}
\frac{d \bar{X}_{t}^{\phi}}{d t}=\Delta \Psi(\bar{X}_{t}^{\phi})+\bar{\mathfrak{f}}(\bar{X}_{t}^{\phi})+\mathfrak{B}_{1}(\bar{X}_{t}^{\phi})P_1\phi,~\bar{X}_{0}^{\phi}=x.
\end{equation*}

\end{theorem}
\begin{proof}
The porous medium operator $\Psi$ can be solved by the classical monotonicity framework, we refer the reader to \cite[Example 4.1.11]{liu2015stochastic}.
\end{proof}

\begin{remark}
A typical example of $\Psi$ is given by
$$\Psi(s):=|s|^{r-1}s,~s\in\mathbb{R}.$$
In the work \cite[Theorem 3.1]{HLL}, the authors have studied the LDP for the stochastic porous media equation with fast oscillations. However, ustilizing our current framework, we are able to  extend the result in  \cite[Theorem 3.1]{HLL} to a more general case, where the fast component can be driven by multiplicative noise.
\end{remark}

\begin{example}
(Cahn-Hilliard equation with fast oscillations)
\end{example}
The stochastic Cahn-Hilliard equation is a well-known model used to characterize phase separation in a
binary alloys and other media. For a comprehensive overview of this model, we refer to the survey \cite{NA}  (see
also \cite{DD,EM} for the stochastic case).

Let $d=1,2,3$. Denote
$$
\mathbb{H}_1:=L^{2}(\Lambda;\mathbb{R}^d),~~\mathbb{V}_1:=\big\{u \in H^{2}(\Lambda;\mathbb{R}^d): \nabla u \cdot \nu=\nabla(\Delta u) \cdot \nu=0 \text { on } \partial \Lambda\big\}.
$$

Choosing the Gelfand triple for the
slow component
$$
\mathbb{V}_{1} \subseteq \mathbb{H}_{1} \subseteq \mathbb{V}_{1}^{*},
$$
and the following Gelfand triple for the fast component
$$
\mathbb{V}_{2}:=H_{0}^{1}(\mathcal{O}) \subseteq \mathbb{H}_{2}:=L^{2}(\mathcal{O}) \subseteq \mathbb{V}_{2}^{*}.
$$
The operator $\mathfrak{A}_{1}(\cdot): \mathbb{V}_{1} \to \mathbb{V}_{1}^{*}$ is defined by
\begin{equation*}
\mathfrak{A}_1(u):=-\Delta^{2} u+\Delta \varphi(u),
\end{equation*}
where $\varphi \in C^{1}(\mathbb{R})$ and there exist $C>0,~ p \in [2, \frac{d+4}{d}]$ such that for all$x, y \in \mathbb{R}$,
\begin{eqnarray}\label{531}
\!\!\!\!\!\!\!\!&&\varphi^{\prime}(x) \geq-C,\\
\!\!\!\!\!\!\!\!&&|\varphi(x)| \leq C(1+|x|^{p}),\\\label{532}
\!\!\!\!\!\!\!\!&&|\varphi(x)-\varphi(y)| \leq C(1+|x|^{p-1}+|y|^{p-1})|x-y|.\label{533}
\end{eqnarray}

Now, we consider the stochastic Cahn-Hilliard equation with fast oscillations,
\begin{eqnarray}\label{53}
\left\{ \begin{aligned}
&d X_{t}^{\varepsilon,\delta}=[-\Delta^{2} X_{t}^{\varepsilon,\delta}+\Delta \varphi(X_{t}^{\varepsilon,\delta})+\mathfrak{f}(X_{t}^{\varepsilon,\delta}, Y_{t}^{\varepsilon,\delta})] d t+\sqrt{\varepsilon}\mathfrak{B}_{1}(X_{t}^{\varepsilon,\delta}) d W_{t}^{1}, \\
&d Y_{t}^{\varepsilon,\delta}=\frac{1}{\delta}[\mathcal{R}(Y_{t}^{\varepsilon,\delta})+g(X_{t}^{\varepsilon,\delta}, Y_{t}^{\varepsilon,\delta})] d t+\frac{1}{\sqrt{\delta}} \mathfrak{B}_{2}(X_{t}^{\varepsilon,\delta}, Y_{t}^{\varepsilon,\delta}) d W_{t}^{2}, \\
&X_{0}^{\varepsilon,\delta}=x \in \mathbb{H}_{1}, Y_{0}^{\varepsilon,\delta}=y \in \mathbb{H}_{2},
\end{aligned} \right.
\end{eqnarray}
where the measurable maps  $\mathfrak{f},$ $\mathfrak{B}_1,$ $g$ and $\mathfrak{B}_2$ satisfy conditions (\ref{511})-(\ref{516}).

\begin{theorem}\label{t53}
Assume that $\varphi$ satisfies (\ref{531})-(\ref{533}) and (\ref{511})-(\ref{516}) hold. If
\begin{equation*}
\lim _{\varepsilon \to 0} \frac{\delta}{\varepsilon}=0,
\end{equation*}
then  $\{X^{\varepsilon,\delta}:\varepsilon>0\}$ in (\ref{53}) satisfies the LDP on $C([0,T];\mathbb{H}_1)$ with the following good rate function
\begin{equation*}
I(f)=\inf _{\{\phi \in L^{2}([0, T] ; \mathscr{U_0}): f=\mathcal{G}^{0}(\int_{0}^{\cdot} \phi_{s} d s)\}}\left \{\frac{1}{2} \int_{0}^{T}\|\phi_{s}\|_{\mathscr{U_0}}^{2} d s\right \},
\end{equation*}
where the measurable map $\mathcal{G}^0:C([0,T];\mathscr{U}_0)\to C([0,T];\mathbb{H}_1)$ is given by
\begin{eqnarray*}
\mathcal{G}^0\Big(\int_0^\cdot\phi_s ds\Big):=\left\{ \begin{aligned}&\bar{X}^\phi,~~\phi\in L^2([0,T]; \mathscr{U});\\
&0,~~~~~~\text{otherwise}.
\end{aligned} \right.
\end{eqnarray*}
Here $\bar{X}^\phi$ is the  solution of the
skeleton equation
\begin{equation*}
\frac{d \bar{X}_{t}^{\phi}}{d t}=-\Delta^{2}\bar{X}_{t}^{\phi}+\Delta \varphi(\bar{X}_{t}^{\phi})+\bar{\mathfrak{f}}(\bar{X}_{t}^{\phi})+\mathfrak{B}_{1}(\bar{X}_{t}^{\phi})P_1\phi,~\bar{X}_{0}^{\phi}=x.
\end{equation*}
\end{theorem}
\begin{proof}
The conditions of $\varphi$ and the Gagliardo-Nirenberg inequality implies that $\mathfrak{A}_1$ satisfies $({\mathbf{A}}{\mathbf{1}})$-$({\mathbf{A}}{\mathbf{4}})$ (see \cite[Example 5.2.27]{liu2015stochastic} for some details). For example, the condition $({\mathbf{A}}{\mathbf{2}})$ reads as
\begin{equation*}
_{\mathbb{V}_{1}^{*}}\left\langle \mathfrak{A}_{1}(u)-\mathfrak{A}_{1}(v), u-v\right\rangle_{\mathbb{V}_{1}} \leq -\frac{1}{2}\|u-v\|_{\mathbb{V}_1}^{2}+C\big(1+\rho(u)+\eta(v)\big)\|u-v\|_{\mathbb{H}_1}^{2},
\end{equation*}
where $$\rho(u):=\|u\|_{\mathbb{V}_1}^{\frac{d(p-1)}{2}}\|u\|_{\mathbb{H}_1}^{\frac{(4-d)(p-1)}{2}},~~\eta(v):=\|v\|_{\mathbb{V}_1}^{\frac{d(p-1)}{2}}\|v\|_{\mathbb{H}_1}^{\frac{(4-d)(p-1)}{2}}.$$ Therefore, Theorem \ref{t53} is a direct consequence of Theorem \ref{t1}.
\end{proof}

\begin{remark}
To the best of our knowledge, this is the first paper to deal with the LDP for stochastic Cahn-Hilliard equations with fast oscillations. As a classical example, the function $\varphi$ can be chosen as  $\varphi(x):=x^3-x$, which represents
the derivative of the double well potential $\Phi(x):=\frac{1}{4}(x^2-1)^2$.
\end{remark}

\begin{example}
(Liquid crystal equation with fast oscillations)
\end{example}

The 2D Liquid crystal equation is a simplified version that captures the essential mathematical properties of the original Ericksen-Leslie equations with the Ginzburg-Landau approximation, we refer the reader to \cite{LL} for more details.

We assume that $\Phi: \mathbb{R}^{3} \to \mathbb{R}^{3}$ satisfies the following condition: there exists a
$k$-th polynomial $\psi:[0, \infty) \to \mathbb{R}$, for some $k \in \mathbb{N}$, such that
\begin{equation*}
\Phi(x):=\psi(|x|^{2}) x=(\sum_{i=0}^{k} a_{i}|x|^{2 i}) x,
\end{equation*}
where $a_i \in \mathbb{R},~i=1,2,\dots,k-1$ and $a_k> 0.$
Denote
$$
\mathbb{H}_1:=H_1 \times[H^{1}(\Lambda;\mathbb{R}^2)], \quad \mathbb{V}_1:=V_1 \times\big\{w_2 \in H^{2}(\Lambda;\mathbb{R}^2): \frac{\partial w_2}{\partial \nu}=0\big\}.
$$
where $V_1:=\big\{w_1 \in H^{1}(\Lambda;\mathbb{R}^2): \nabla \cdot w_1=0,w_1|_{\partial \Lambda}=0\big\}$ and $H_1$ is the closure of $V_1$ w.r.t.~$L^2$-norm. The norms in $\mathbb{H}_1$ and in $\mathbb{V}_1$ are denoted respectively by
$$
\|u\|_{\mathbb{H}_1}^{2}:=\|w_1\|_{H_1}^{2}+\|w_2\|_{1}^{2}, \quad\|u\|_{\mathbb{V}_1}^{2}:=\|w_1\|_{V}^{2}+\|w_2\|_{2}^{2},
$$
for any $u:=(w_1,w_2).$

We choose the Gelfand triple for the slow component
$$
\mathbb{V}_{1} \subseteq \mathbb{H}_{1} \subseteq \mathbb{V}_{1}^{*},
$$
and the following Gelfand triple for the fast component
$$
\mathbb{V}_{2}:=H_{0}^{1}(\mathcal{O}) \subseteq \mathbb{H}_{2}:=L^{2}(\mathcal{O}) \subseteq \mathbb{V}_{2}^{*}.
$$
The operator $\mathfrak{A}_{1}(\cdot): \mathbb{V}_{1} \to \mathbb{V}_{1}^{*}$ is defined by
\begin{eqnarray*}
\mathfrak{A}_1(u):=\left( \begin{aligned}
P_{\mathbb{H}_1}[&\Delta w_1-(w_1 \cdot \nabla) w_1-\nabla w_2 \cdot \Delta w_2] \\
&\Delta w_2-(w_1\cdot \nabla) w_2-\Phi(w_2)
\end{aligned} \right),
\end{eqnarray*}
where $P_{\mathbb{H}_1}$ is the Helmholtz-Leray projection.

Now, we consider the stochastic 2D Liquid crystal equation with fast oscillations,
\begin{eqnarray}\label{54}
\left\{ \begin{aligned}
&d X_{t}^{\varepsilon,\delta}=[\mathfrak{A}_1(X_{t}^{\varepsilon,\delta})+\mathfrak{f}(X_{t}^{\varepsilon,\delta}, Y_{t}^{\varepsilon,\delta})] d t+\sqrt{\varepsilon}\mathfrak{B}_{1}(X_{t}^{\varepsilon,\delta}) d W_{t}^{1}, \\
&d Y_{t}^{\varepsilon,\delta}=\frac{1}{\delta}[\mathcal{R}(Y_{t}^{\varepsilon,\delta})+g(X_{t}^{\varepsilon,\delta}, Y_{t}^{\varepsilon,\delta})] d t+\frac{1}{\sqrt{\delta}} \mathfrak{B}_{2}(X_{t}^{\varepsilon,\delta}, Y_{t}^{\varepsilon,\delta}) d W_{t}^{2}, \\
&X_{0}^{\varepsilon,\delta}=x \in \mathbb{H}_{1}, Y_{0}^{\varepsilon,\delta}=y \in \mathbb{H}_{2},
\end{aligned} \right.
\end{eqnarray}
where the coefficients $\mathfrak{f},$ $\mathfrak{B}_1,$ $g$ and $\mathfrak{B}_2$ satisfy conditions (\ref{511})-(\ref{516}).
\begin{theorem}\label{t54}
Assume that  (\ref{511})-(\ref{516}) hold. If
\begin{equation*}
\lim _{\varepsilon \to 0} \frac{\delta}{\varepsilon}=0,
\end{equation*}
then  $\{X^{\varepsilon,\delta}:\varepsilon>0\}$ in (\ref{54}) satisfies the LDP on $C([0,T]; \mathbb{H}_{1})$ with the following good rate function
\begin{equation*}
I(f)=\inf _{\{\phi \in L^{2}([0, T] ; \mathscr{U_0}): f=\mathcal{G}^{0}(\int_{0}^{\cdot} \phi_{s} d s)\}}\left \{\frac{1}{2} \int_{0}^{T}\|\phi_{s}\|_{\mathscr{U_0}}^{2} d s\right \},
\end{equation*}
where the measurable map $\mathcal{G}^0:C([0,T];\mathscr{U}_0)\to C([0,T];\mathbb{H}_1)$ is given by
\begin{eqnarray*}
\mathcal{G}^0\Big(\int_0^\cdot\phi_s ds\Big):=\left\{ \begin{aligned}&\bar{X}^\phi,~~\phi\in L^2([0,T]; \mathscr{U});\\
&0,~~~~~~\text{otherwise}.
\end{aligned} \right.
\end{eqnarray*}
Here $\bar{X}^\phi$ is the  solution of the
skeleton equation
\begin{equation*}
\frac{d \bar{X}_{t}^{\phi}}{d t}=\mathfrak{A}_1(\bar{X}_{t}^{\phi})+\bar{\mathfrak{f}}(\bar{X}_{t}^{\phi})+\mathfrak{B}_{1}(\bar{X}_{t}^{\phi})P_1\phi,~\bar{X}_{0}^{\phi}=x.
\end{equation*}
\end{theorem}
\begin{proof}
The construction of $\psi$ and the Gagliardo-Nirenberg inequality imply that $\mathfrak{A}_1$ satisfies $({\mathbf{A}}{\mathbf{1}})$-$({\mathbf{A}}{\mathbf{4}})$ (see \cite[Example 4.5]{RSZ} for details). For example, the condition $({\mathbf{A}}{\mathbf{2}})$ reads as
\begin{equation*}
_{\mathbb{V}_{1}^{*}}\left\langle \mathfrak{A}_{1}(u)-\mathfrak{A}_{1}(v), u-v\right\rangle_{\mathbb{V}_{1}} \leq -\frac{1}{2}\|u-v\|_{\mathbb{V}_1}^{2}+C\big(1+\rho(u)+\eta(v)\big)\|u-v\|_{\mathbb{H}_1}^{2},
\end{equation*}
where
$$\rho(u):=\|u\|_{\mathbb{H}_1}^{4k},~\eta(v):=\|v\|_{\mathbb{H}_1}^{4k}+\|v\|_{\mathbb{V}_1}^{2}\|v\|_{\mathbb{H}_1}^{2}.$$
Then Theorem \ref{t54} is a direct consequence of Theorem \ref{t1}.
\end{proof}

\begin{example}
(Tamed 3D Navier-Stokes equation with fast oscillations)
\end{example}

The stochastic tamed 3D Navier-Stokes equation was first proposed by R\"{o}ckner and Zhang in \cite{RZ}, which is helpful to understand certain characteristics of the stochastic 3D Navier-Stokes equation.

Let $\mathbb{T}^3$ be a three dimensional torus. Set
$$
\mathbb{H}^{m}:=\big\{u \in H^{m}_0(\mathbb{T}^3;\mathbb{R}^3): \operatorname{div}(u)=0\big\},
$$
with the norm denoted also by $\|\cdot\|_{m}$ and the inner product denoted  by $\langle\cdot, \cdot\rangle_{m}$. Note that $\mathbb{H}^{0}$ is a closed linear subspace of Hilbert space $L^2(\mathbb{T}^3;\mathbb{R}^3)$.

We choose the Gelfand triple for the slow component
$$
\mathbb{V}_{1}:=\mathbb{H}^{2}(\mathbb{T}^3;\mathbb{R}^3) \subseteq \mathbb{H}_{1}:=\mathbb{H}^{1}(\mathbb{T}^3;\mathbb{R}^3) \subseteq \mathbb{V}_{1}^{*},
$$
and the following Gelfand triple for the fast component
$$
\mathbb{V}_{2}:=H_{0}^{1}(\mathcal{O};\mathbb{R}) \subseteq \mathbb{H}_{2}:=L^{2}(\mathcal{O};\mathbb{R}) \subseteq \mathbb{V}_{2}^{*}.
$$
Let $P_{\mathbb{H}_1}$ be the orthogonal projection from $L^{2}(\mathbb{T}^3;\mathbb{R}^3)$ to $\mathbb{H}_1$. It is well-known that $P_{\mathbb{H}_1}$ commutes with the derivative operators, and that $P_{\mathbb{H}_1}$ can be restricted to a bounded linear operator from $H^m(\mathbb{T}^3;\mathbb{R}^3)$ to $\mathbb{H}^{m}$.
For any $u \in \mathbb{H}^{0}$ and $v \in L^{2}(\mathbb{T}^3;\mathbb{R}^3),$ we have
\begin{equation*}
\langle u, v\rangle_{0}:=\langle u, P_{\mathbb{H}_1} v\rangle_{0}=\langle u, v\rangle_{L^{2}}.
\end{equation*}
Next, we can define the operator $\mathfrak{A}_{1}(\cdot): \mathbb{V}_{1} \to \mathbb{V}_{1}^{*}$ by
\begin{equation*}
\mathfrak{A}_1(u):=P_{\mathbb{H}_1}(\Delta u-(u \cdot \nabla) u-g_N(|u|^2)u),
\end{equation*}
where $g_{N}: \mathbb{R}_{+} \to \mathbb{R}_{+}$ is a smooth function being nonzero only for large arguments, i.e. for a constant $N>0,$
\begin{eqnarray}\label{551}
\left\{ \begin{aligned}
&g_{N}(r)=0, \quad r \leq N ; \\
&g_{N}(r)=(r-N), \quad r \geq N+1 ; \\
&0 \leq g_{N}^{\prime}(r) \leq C, \quad r \geq 0 .
\end{aligned} \right.
\end{eqnarray}

Now, we consider the stochastic tamed 3D Navier-Stokes equation with fast oscillations,
\begin{eqnarray}\label{55}
\left\{ \begin{aligned}
&d X_{t}^{\varepsilon,\delta}=[P_{\mathbb{H}_1}(\Delta X_{t}^{\varepsilon,\delta}-(X_{t}^{\varepsilon,\delta} \cdot \nabla) X_{t}^{\varepsilon,\delta}-g_N(|X_{t}^{\varepsilon,\delta}|^2)X_{t}^{\varepsilon,\delta})+\mathfrak{f}(X_{t}^{\varepsilon,\delta}, Y_{t}^{\varepsilon,\delta})] d t\\
&\quad\quad\quad\,\,\,\,+\sqrt{\varepsilon}\mathfrak{B}_{1}(X_{t}^{\varepsilon,\delta}) d W_{t}^{1}, \\
&dY_{t}^{\varepsilon,\delta}=\frac{1}{\delta}[\mathcal{R}(Y_{t}^{\varepsilon,\delta})+g(X_{t}^{\varepsilon,\delta}, Y_{t}^{\varepsilon,\delta})] d t+\frac{1}{\sqrt{\delta}} \mathfrak{B}_{2}(X_{t}^{\varepsilon,\delta}, Y_{t}^{\varepsilon,\delta}) d W_{t}^{2}, \\
&X_{0}^{\varepsilon,\delta}=x \in \mathbb{H}_{1}, Y_{0}^{\varepsilon,\delta}=y \in \mathbb{H}_{2},
\end{aligned} \right.
\end{eqnarray}
where the coefficients $\mathfrak{f},$ $\mathfrak{B}_1,$ $g$ and $\mathfrak{B}_2$ satisfy conditions (\ref{511})-(\ref{516}).
\begin{theorem}\label{t55}
Assume that $g_N$ satisfies (\ref{551}) and (\ref{511})-(\ref{516}) hold. If
\begin{equation*}
\lim _{\varepsilon \to 0} \frac{\delta}{\varepsilon}=0,
\end{equation*}
then $\{X^{\varepsilon,\delta}:\varepsilon>0\}$ in (\ref{55}) satisfies the LDP on $C([0,T];\mathbb{H}_1)$ with the following good rate function
\begin{equation*}
I(f)=\inf _{\{\phi \in L^{2}([0, T] ; \mathscr{U_0}): f=\mathcal{G}^{0}(\int_{0}^{\cdot} \phi_{s} d s)\}}\left \{\frac{1}{2} \int_{0}^{T}\|\phi_{s}\|_{\mathscr{U_0}}^{2} d s\right \},
\end{equation*}
where the measurable map $\mathcal{G}^0:C([0,T];\mathscr{U}_0)\to C([0,T];\mathbb{H}_1)$ is given by
\begin{eqnarray*}
\mathcal{G}^0\Big(\int_0^\cdot\phi_s ds\Big):=\left\{ \begin{aligned}&\bar{X}^\phi,~~\phi\in L^2([0,T]; \mathscr{U});\\
&0,~~~~~~\text{otherwise}.
\end{aligned} \right.
\end{eqnarray*}
Here $\bar{X}^\phi$ is the  solution of the
skeleton equation
\begin{equation*}
\frac{d \bar{X}_{t}^{\phi}}{d t}=\mathfrak{A}_1(\bar{X}_{t}^{\phi})+\bar{\mathfrak{f}}(\bar{X}_{t}^{\phi})+\mathfrak{B}_{1}(\bar{X}_{t}^{\phi})P_1\phi,~\bar{X}_{0}^{\phi}=x.
\end{equation*}
\end{theorem}

\begin{proof}
We are going to check that $\mathfrak{A}_1$ satisfies $({\mathbf{A}}{\mathbf{1}})$-$({\mathbf{A}}{\mathbf{4}}).$ Let us denote the operator
\begin{equation*}
\widetilde{\mathfrak{A}}(u):=P_{\mathbb{H}_1}(\Delta u),~~~~\widehat{\mathfrak{A}}(u):=P_{\mathbb{H}_1}((u \cdot \nabla) u).
\end{equation*}
For any $u,~v \in \mathbb{V}_1,$ it is easy to see that
\begin{eqnarray}\label{552}
_{\mathbb{V}_{1}^{*}}\left\langle \mathfrak{A}_{1}(u)-\mathfrak{A}_{1}(v), u-v\right\rangle_{\mathbb{V}_{1}}= \!\!\!\!\!\!\!\!&& \langle \widetilde{\mathfrak{A}} (u-v),(I-\Delta) (u-v)\rangle_{0}  \nonumber\\
\!\!\!\!\!\!\!\!&& -\langle \widehat{\mathfrak{A}}(u)-\widehat{\mathfrak{A}}(v),(I-\Delta) (u-v)\rangle_{0} \nonumber\\
\!\!\!\!\!\!\!\!&& -\langle P_{\mathbb{H}_1}(g_{N}(|u|^{2}) u)-P_{\mathbb{H}_1}(g_{N}(|v|^{2}) v),(I-\Delta) (u-v)\rangle_{0}\nonumber\\
\!\!\!\!\!\!\!\!&&=:\uppercase\expandafter{\romannumeral1}+\uppercase\expandafter{\romannumeral2}+\uppercase\expandafter{\romannumeral3}.
\end{eqnarray}
Below, we estimate the terms in (\ref{552}), respectively.
\begin{eqnarray}\label{553}
\uppercase\expandafter{\romannumeral1}= \!\!\!\!\!\!\!\!&& -\langle(I-\Delta) (u-v),(I-\Delta) (u-v)\rangle_{0}+\langle u-v,(I-\Delta) (u-v)\rangle_{0} \nonumber\\
\!\!\!\!\!\!\!\!&&=-\|u-v\|_{2}^{2}+\|u-v\|_{0}^{2}+\|\nabla (u-v)\|_{0}^{2}\nonumber\\
\!\!\!\!\!\!\!\!&&\leq-\|u-v\|_{2}^{2}+C \|u-v\|_{1}^{2}.
\end{eqnarray}
According to \cite[(3.31)]{HLLL} and Young's inequality, we have
\begin{eqnarray}\label{554}
\uppercase\expandafter{\romannumeral2} \leq \!\!\!\!\!\!\!\!&& \|\widehat{\mathfrak{A}}(u)-\widehat{\mathfrak{A}}(v)\|_{0}\|u-v\|_{2}  \nonumber\\
\leq \!\!\!\!\!\!\!\!&& C \|u-v\|_{2}^{\frac{3}{2}}\|u-v\|_{1}^{\frac{1}{2}}\|u\|_{1}+ C \|v\|_{2}^{\frac{1}{2}}\|v\|_{1}^{\frac{1}{2}}\|u-v\|_{1}\|u-v\|_{2} \nonumber\\
\leq \!\!\!\!\!\!\!\!&&\frac{1}{4} \|u-v\|_{2}^{2}+C \|u-v\|_{1}^{2}\|u\|_{1}^{4}+C \|v\|_{2}\|v\|_{1}\|u-v\|_{1}^{2}.
\end{eqnarray}
For $\uppercase\expandafter{\romannumeral3},$ by \cite[(3.32)]{HLLL} and Young's inequality, it follows that
\begin{eqnarray}\label{555}
\uppercase\expandafter{\romannumeral3} \!\!\!\!\!\!\!\!&&\leq \|g_{N}(|u|^{2}) u-g_{N}(|v|^{2}) v\|_{0}\|u-v\|_{2}  \nonumber\\
\!\!\!\!\!\!\!\!&& \leq C \|u-v\|_{1}(\|u\|_{1}^{2}+\|v\|_{1}^{2})\|u-v\|_{2} \nonumber\\
\!\!\!\!\!\!\!\!&& \leq \frac{1}{4} \|u-v\|_{2}^{2}+C\|u-v\|_{1}^{2}(\|u\|_{1}^{4}+\|v\|_{1}^{4}),
\end{eqnarray}
Substituting (\ref{553})-(\ref{555}) into (\ref{552}), we have
\begin{equation*}
_{\mathbb{V}_{1}^{*}}\left\langle \mathfrak{A}_{1}(u)-\mathfrak{A}_{1}(v), u-v\right\rangle_{\mathbb{V}_{1}}
\leq -\frac{1}{2} \|u-v\|_{2}^{2}+ C\big(1+\rho(u)+\eta(v)\big)\|u-v\|_{1}^{2},
\end{equation*}
where $$\rho(u):=\|u\|_{1}^{4},~\eta(v):=\|v\|_{1}^{4}+\|v\|_{1}\|v\|_{2}.$$
Therefore, $({\mathbf{A}}{\mathbf{2}})$ is satisfied. Using Young's inequality, the definition of function $g_N$ and \cite[(3.7)]{HLLL}, it follows that
\begin{eqnarray*}
_{\mathbb{V}_{1}^{*}}\left\langle \mathfrak{A}_{1}(u), u\right\rangle_{\mathbb{V}_{1}}= \!\!\!\!\!\!\!\!&& \langle \widetilde{\mathfrak{A}} (u),(I-\Delta) u\rangle_{0}-\langle \widehat{\mathfrak{A}}(u),(I-\Delta)u\rangle_{0} -\langle P_{\mathbb{H}_1}(g_{N}(|u|^{2}) u),(I-\Delta) u\rangle_{0}\nonumber\\
\!\!\!\!\!\!\!\!&&\leq-\|u\|_{2}^{2}+C \|u\|_{1}^{2}+\|\widetilde{\mathfrak{A}} (u)\|_0\|(I-\Delta) u\|_0-\||u|\cdot |\nabla u|\|_0^2+N\|\nabla u\|_0^2\nonumber\\
\!\!\!\!\!\!\!\!&&\leq-\|u\|_{2}^{2}+C_N \|u\|_{1}^{2}+\frac{1}{2}\|(I-\Delta) u\|^2_0+\frac{1}{2}\|(u\cdot \nabla)u\|_{0}^{2}-\||u|\cdot |\nabla u|\|_0^2\nonumber\\
\!\!\!\!\!\!\!\!&&\leq-\frac{1}{2}\|u\|_{2}^{2}+C_N \|u\|_{1}^{2}-\frac{1}{2}\||u|\cdot |\nabla u|\|_0^2\nonumber\\
\!\!\!\!\!\!\!\!&&\leq-\frac{1}{2}\|u\|_{2}^{2}+C_N \|u\|_{1}^{2}.
\end{eqnarray*}
Thus, the operator $\mathfrak{A}_1(u)$ satisfies condition $({\mathbf{A}}{\mathbf{3}}).$ Due to \cite[(3.31)-(3.32)]{HLLL}, we have
\begin{eqnarray*}
\|\mathfrak{A}_{1}(u)\|_0^2\!\!\!\!\!\!\!\!&&\leq C(\|\Delta u\|_0^2+\|\widehat{\mathfrak{A}}(u)\|_0^2+\|g_N(|u|^2)u\|_0^2)\\
\!\!\!\!\!\!\!\!&&\leq C(1+ \|u\|_2^2+\|u\|_2\|u\|_1^3+\|u\|_1^6).
\end{eqnarray*}
Hence, $({\mathbf{A}}{\mathbf{4}})$ is satisfied. The hemicontinuity condition $({\mathbf{A}}{\mathbf{1}})$
can be easily verified by the dominated convergence theorem.
\end{proof}

\section{Appendix}
\setcounter{equation}{0}
 \setcounter{definition}{0}
In this section, we consider the parameterized equation  associated with the fast component in (\ref{aimequation}) for
any fixed $x \in \mathbb{H}_1$ and the skeleton equation (\ref{skletonequation}). More precisely, we recall  the exponential ergodicity of the parameterized equation,  and then prove the well-posedness of the skeleton equation (\ref{skletonequation}) with some crucial estimates.

For any fixed  $x\in \mathbb{H}_1$, we consider the following parameterized equation
\begin{eqnarray}\label{frozenequation}
\left\{ \begin{aligned}
&d Y_{t}=\mathfrak{A}_{2}(x, Y_{t}) d t+\mathfrak{B}_{2}(x, Y_{t}) d \widetilde{W}_{t}^{2}, \\
&Y_{0}=y \in \mathbb{H}_{2},
\end{aligned}\right.
\end{eqnarray}
where $\widetilde{W}_{t}^{2}$ denotes an $\mathbb{U}_{2}$-valued cylindrical Wiener process on another complete filtered probability space $(\Omega, \mathscr{F},\mathbb{P})$, which is independent of $W^2_t$.

Assume that $({\mathbf{H}}{\mathbf{1}})$-$({\mathbf{H}}{\mathbf{4}})$  hold, from \cite[Theorem 4.2.4]{liu2015stochastic}, we can get for any fixed $x$ there exists a unique solution denoted by $Y_{t}^{x, y}$ to the parameterized equation  (\ref{frozenequation}), which is a homogeneous Markov process.

Set $P_{t}^{x}$ is the Markov transition semigroup of process $Y_{t}^{x, y},$ we have for any bounded measurable map $\varphi:\mathbb{H}_2\to \mathbb{R}$,
$$P^{x}_t\varphi(y):=\mathbb E[\varphi(Y^{x,y}_t)],~~y\in \mathbb{H}_2,\,t>0.$$
By \cite[Theorem 4.3.9]{liu2009harnack}, $P_{t}^{x}$ has the unique invariant probability measure $\mu^{x}$.

The following estimates and the exponential ergodicity result are from \cite[Propositions 4.1-4.2]{HLL4}.
\begin{lemma}
There exists a constant $C>0$ such that for any $x,~x_{1},~x_{2} \in \mathbb{H}_{1},~y \in \mathbb{H}_{2}$, we have
\begin{equation}%\label{frozenlip}
\sup _{t \in[0, \infty)} {\mathbb{E}}\|Y_{t}^{x_{1}, y}-Y_{t}^{x_{2}, y}\|_{\mathbb{H}_{2}}^{2} \leq C\|x_{1}-x_{2}\|_{\mathbb{H}_{1}}^{2},
\end{equation}
\begin{equation}\label{frozenestimate}
\sup _{t \in[0, \infty)} {\mathbb{E}}\|Y_{t}^{x, y}\|_{\mathbb{H}_{2}}^{2} \leq C(1+\|x\|_{\mathbb{H}_{1}}^{2}+\|y\|_{\mathbb{H}_{2}}^{2}).
\end{equation}
\end{lemma}

\begin{lemma}\label{expon}
Assume $({\mathbf{H}}{\mathbf{1}})$-$({\mathbf{H}}{\mathbf{4}})$ hold and the embedding $\mathbb{V}_2 \subset \mathbb{H}_2$ is compact. Then there is a constant $C>0$ such that for any $x \in \mathbb{H}_{1},~y \in \mathbb{H}_{2}$ and Lipschitz function $\mathfrak{f}:\mathbb{H}_1\times\mathbb{H}_2\to \mathbb{H}_1$, we have
$$\Big\|\mathbb{E}\mathfrak{f}(x,Y^{x,y}_t)-\bar{\mathfrak{f}}(x)\Big\|_{\mathbb{H}_1}\leq C(1+\|x\|_{\mathbb{H}_1}+\|y\|_{\mathbb{H}_2})e^{-\frac {\kappa t}2}.$$
\end{lemma}

Now we recall the skeleton equation (\ref{skletonequation}):
\begin{eqnarray}\label{sk2}
\left\{\begin{aligned}
&\displaystyle\frac{d \bar{X}_{t}^{\phi}}{d t}=\mathfrak{A}_1(\bar{X}_{t}^{\phi})+\bar{\mathfrak{f}}(\bar{X}_{t}^{\phi})+\mathfrak{B}_{1}(\bar{X}_{t}^{\phi})P_1\phi,\\
&\bar{X}_{0}^{\phi}=x.
\end{aligned}\right.
\end{eqnarray}

\begin{lemma}\label{skletonproof}
Assume $({\mathbf{A}}{\mathbf{1}})$-$({\mathbf{A}}{\mathbf{5}})$ hold. Then skeleton equation (\ref{skletonequation}) has a unique solution. Moreover,
\begin{equation}
\sup _{\phi \in S_{M}}\Big\{\sup _{t \in[0, T]}\|\bar{X}_{t}^{\phi}\|_{\mathbb{H}_{1}}^{2}+ \int_{0}^{T}\|\bar{X}_{t}^{\phi}\|_{\mathbb{V}_{1}}^{\alpha_{1}} d t\Big\} \leq C_{T, M}(1+\|x\|_{\mathbb{H}_{1}}^2).\label{seletonestimate}
\end{equation}
\end{lemma}
\begin{proof}
The proof can be divided into two steps.

\vspace{1mm}
\textbf{Step 1}: We first consider  $\phi\in L^\infty([0,T];\mathscr{U})$.
By Hypothesis \ref{hypo1} and the Lipschitz continuity of $\bar{\mathfrak{f}}$ (cf.~\cite[Remark 4.1]{HLL4}), it is clear that (\ref{skletonequation}) has a unique solution $\bar{X}^\phi\in C([0,T];\mathbb{H}_1)$, for any $\phi\in L^\infty([0,T];\mathscr{U})$ (cf.~\cite{liu2015stochastic}).

Now, for any $\phi\in L^2([0,T];\mathscr{U})$, we can always find a sequence $\phi^n\in L^\infty([0,T];\mathscr{U})$ such that
\begin{equation}\label{esphi}
\int_0^T\|\phi_t^n-\phi_t\|_\mathscr{U}^2dt\to0,~~n\to\infty.
\end{equation}
Let $\bar{X}^{\phi^n}$ denote the unique solution to (\ref{sk2}) with $\phi^n$ replacing $\phi$. Due to condition $({\mathbf{A}}{\mathbf{2}})$, Young's inequality and the Lipschitz continuity of $\mathfrak{f},$ for any $m,n\in\mathbb N$ we have
\begin{eqnarray*}
\!\!\!\!\!\!\!\!&&\frac d{dt}\|\bar{X}^{\phi^n}_t-\bar{X}^{\phi^m}_t\|_{\mathbb{H}_1}^2\\
=\!\!\!\!\!\!\!\!&&2\,_{\mathbb{V}_1^*}\big\langle \mathfrak{A}_1(\bar{X}^{\phi^n}_t)-\mathfrak{A}_1(\bar{X}^{\phi^m}_t),\bar{X}^{\phi^n}_t-\bar{X}^{\phi^m}_t\big\rangle_{\mathbb{V}_1}
+2\big\langle \bar{\mathfrak{f}}(\bar{X}^{\phi^n}_t)-\bar{\mathfrak{f}}(\bar{X}^{\phi^m}_t),\bar{X}^{\phi^n}_t-\bar{X}^{\phi^m}_t\big\rangle_{\mathbb{H}_1}\\
\!\!\!\!\!\!\!\!&&+2\big\langle \mathfrak{B}_1(\bar{X}^{\phi^n}_t)P_1\phi^n_t-\mathfrak{B}_1(\bar{X}^{\phi^m}_t)P_1\phi^m_t,\bar{X}^{\phi^n}_t-\bar{X}^{\phi^m}_t\big\rangle_{\mathbb{H}_1}\\
\leq\!\!\!\!\!\!\!\!&&2\,_{\mathbb{V}_1^*}\big\langle \mathfrak{A}_1(\bar{X}^{\phi^n}_t)-\mathfrak{A}_1(\bar{X}^{\phi^m}_t),\bar{X}^{\phi^n}_t-\bar{X}^{\phi^m}_t\big\rangle_{\mathbb{V}_1}+\|\bar{\mathfrak{f}}(\bar{X}^{\phi^n}_t)-\bar{\mathfrak{f}}(\bar{X}^{\phi^m}_t)\|_{\mathbb{H}_1}^2\\
\!\!\!\!\!\!\!\!&&+\|\bar{X}^{\phi^n}_t-\bar{X}^{\phi^m}_t\|_{\mathbb{H}_1}^2+2\big\langle \mathfrak{B}_1(\bar{X}^{\phi^n}_t)(\phi^n_t-\phi^m_t),\bar{X}^{\phi^n}_t-\bar{X}^{\phi^m}_t\big\rangle_{\mathbb{H}_1}\\
\!\!\!\!\!\!\!\!&&+2\big\langle \big[\mathfrak{B}_1(\bar{X}^{\phi^n}_t)-\mathfrak{B}_1(\bar{X}^{\phi^m}_t)\big]P_1\phi^m_t,\bar{X}^{\phi^n}_t-\bar{X}^{\phi^m}_t\big\rangle_{\mathbb{H}_1}\\
\leq\!\!\!\!\!\!\!\!&&2\,_{\mathbb{V}_1^*}\big\langle \mathfrak{A}_1(\bar{X}^{\phi^n}_t)-\mathfrak{A}_1(\bar{X}^{\phi^m}_t),\bar{X}^{\phi^n}_t-\bar{X}^{\phi^m}_t\big\rangle_{\mathbb{V}_1}+\|\mathfrak{B}_1(\bar{X}^{\phi^n}_t)-\mathfrak{B}_1(\bar{X}^{\phi^m}_t)\|_{L_2(\mathbb{U}_2,\mathbb{H}_2)}^2\\
\!\!\!\!\!\!\!\!&&+\|\bar{X}^{\phi^n}_t-\bar{X}^{\phi^m}_t\|_{\mathbb{H}_1}^2+\|\mathfrak{B}_1(\bar{X}^{\phi^n}_t)\|_{L_2(\mathbb{U}_2,\mathbb{H}_2)}^2\|\bar{X}^{\phi^n}_t-\bar{X}^{\phi^m}_t\|_{\mathbb{H}_1}^2\\
\!\!\!\!\!\!\!\!&&+\|\phi^n_t-\phi^m_t\|_\mathscr{U}^2+\|\phi^m_t\|_\mathscr{U}^2\|\bar{X}^{\phi^n}_t-\bar{X}^{\phi^m}_t\|_{\mathbb{H}_1}^2\\
\leq\!\!\!\!\!\!\!\!&&(C+\rho(\bar{X}^{\phi^n}_t)+\eta(\bar{X}^{\phi^m}_t)+\|\mathfrak{B}_1(\bar{X}^{\phi^n}_t)\|_{L_2(\mathbb{U}_2,\mathbb{H}_2)}^2+\|\phi^m_t\|_\mathscr{U}^2)\|\bar{X}^{\phi^n}_t-\bar{X}^{\phi^m}_t\|_{\mathbb{H}_1}^2+\|\phi^n_t-\phi^m_t\|_\mathscr{U}^2.
\end{eqnarray*}
Gronwall's lemma yields that
\begin{eqnarray}
\!\!\!\!\!\!\!\!&&\sup\limits_{t\in[0,T]}\|\bar{X}^{\phi^n}_t-\bar{X}^{\phi^m}_t\|_{\mathbb{H}_1}^2
\nonumber\\
\leq\!\!\!\!\!\!\!\!&& \text{exp}\left \{ \int_0^T\Big(C+\rho(\bar{X}^{\phi^n}_t)+\eta(\bar{X}^{\phi^m}_t)+\|\mathfrak{B}_1(\bar{X}^{\phi^n}_t)\|_{L_2(\mathbb{U}_2,\mathbb{H}_2)}^2+\|\phi^m_t\|_\mathscr{U}^2\Big)dt\right\}\nonumber\\
\!\!\!\!\!\!\!\!&& \cdot \int_0^T\|\phi^n_t-\phi^m_t\|_\mathscr{U}^2dt.\label{SkeSqeLipschitz}
\end{eqnarray}
Using the energy equality of $\bar{X}^{\phi^n}_t$, it follows that
\begin{eqnarray*}
\frac{d}{d t}\|\bar{X}_{t}^{\phi^{n}}\|_{\mathbb{H}_{1}}^{2}
\!\!\!\!\!\!\!\!&&\leq2_{\mathbb{V}_{1}^{*}}\langle \mathfrak{A}_1(\bar{X}_{t}^{\phi^{n}}), \bar{X}_{t}^{\phi^{n}}\rangle_{\mathbb{V}_{1}}+\|\mathfrak{B}_{1}(\bar{X}_{t}^{\phi^{n}})\|_{L_2(\mathbb{U}_2,\mathbb{H}_2)}^2+C(1+\|\phi_t^{n}\|_\mathscr{U}^2)\|\bar{X}_{t}^{\phi^{n}}\|_{\mathbb{H}_{1}}^{2}+C\\
\!\!\!\!\!\!\!\!&&\leq-\eta_{1}\|\bar{X}_{t}^{\phi^{n}}\|_{\mathbb{V}_{1}}^{\alpha_{1}}+C(1+\|\phi_{t}^{n}\|_\mathscr{U}^{2})\|\bar{X}_{t}^{\phi^{n}}\|_{\mathbb{H}_{1}}^{2}+C,
\end{eqnarray*}
where the last inequality is due to condition $({\mathbf{A}}{\mathbf{3}})$.
Then Gronwall's lemma implies that
\begin{equation}
\sup\limits_{t \in[0, T]}\|\bar{X}_{t}^{\phi^{n}}\|_{\mathbb{H}_{1}}^{2}+\eta_{1} \displaystyle \int_{0}^{T}\|\bar{X}_{t}^{\phi^{n}}\|_{\mathbb{V}_{1}}^{\alpha_{1}} d t
\leq C_T(1+\|x\|_{\mathbb{H}_{1}}^{2}) \exp \Bigg\{\displaystyle\int_{0}^{T}\|\phi_{t}^{n}\|_{U}^{2} d t\Bigg\}<\infty.\label{se1}
\end{equation}
Substituting (\ref{se1})  into (\ref{SkeSqeLipschitz}) and by conditions (\ref{conloc}) and $({\mathbf{A}}{\mathbf{5}})$ give that
\begin{equation*}
\displaystyle\sup\limits_{t\in[0,T]}\|\bar{X}^{\phi^n}_t-\bar{X}^{\phi^m}_t\|_{\mathbb{H}_1}^2\leq C_{T}\int_0^T\|\phi^n_t-\phi^m_t\|_\mathscr{U}^2dt.
\end{equation*}

In view of (\ref{esphi}), let $m,n\to\infty$, we obtain that $\{\bar{X}^{\phi^n}\}$ is a Cauchy net in $C([0,T];\mathbb{H}_1).$ Due to the completeness of $C([0,T];\mathbb{H}_1)$, $\{\bar{X}^{\phi^n}\}$ is a convergent sequence and  the limit is denoted by $X$.

\vspace{1mm}
\textbf{Step 2}: Next, we aim to prove that $X$ is the unique solution to (\ref{sk2}) with $\phi\in L^2([0,T];\mathscr{U})$ defined in (\ref{esphi}), i.e.,
\begin{equation}\label{e11}
X=\bar{X}^\phi.
\end{equation}

 By (\ref{se1}) and condition $({\mathbf{A}}{\mathbf{4}})$, $({\mathbf{A}}{\mathbf{5}})$, we can get
\begin{equation*}
\sup _{n \geq 1}\Bigg\{\int_{0}^{T}\|\mathfrak{A}_1(\bar{X}^{\phi^n}_t)\|_{\mathbb{V}_1^{*}}^{\frac{\alpha_1}{\alpha_1-1}} d t+\int_{0}^{T}\|\mathfrak{B}_1(\bar{X}^{\phi^n}_t)\|_{L_{2}(\mathbb{U}_1, \mathbb{H}_{1})}^{2} d t\Bigg\}<\infty.
\end{equation*}
Then, there exist $\mathscr{A}(\cdot) \in L^{\frac{\alpha_1}{\alpha_1-1}}([0, T], \mathbb{V}_1^{*})$  such that
\begin{eqnarray*}
\!\!\!\!\!\!\!\!&&\bar{X}^{\phi^{n}}_{\cdot} \to X_{\cdot}~~\text {in }~~ \mathbb{S}_T;\\
\!\!\!\!\!\!\!\!&&\bar{X}^{\phi^{n}}_{\cdot} \rightharpoonup X_{\cdot}~~\text {in }~~ L^{\alpha_{1}}([0, T] ; \mathbb{V}_{1});\\
\!\!\!\!\!\!\!\!&&\mathfrak{A}_1(\bar{X}^{\phi^{n}}_{\cdot}) \rightharpoonup \mathscr{A}(\cdot)~~\text { in }~~  L^{\frac{\alpha_1}{\alpha_1-1}}([0, T], \mathbb{V}_1^{*}),
\end{eqnarray*}
as $n \to \infty.$
Denote
$$\widehat{X}_{t}^{\phi} := x+\int_{0}^{t} \mathscr{A}(s) d s+\int_{0}^{t}\bar{\mathfrak{f}}(\bar{X}_{s}^{\phi})ds+\int_{0}^{t} \mathfrak{B}_1(\bar{X}_s^{\phi})P_1 \phi_s d s,$$
then it is straightforward that $\widehat{X}^{\phi}=X,~dt\text{-a.e.}$. It suffices to prove that
\begin{equation*}
\mathscr{A}(\cdot)=\mathfrak{A}_1(\bar{X}_\cdot^{\phi}),
\end{equation*}
which follows from the same argument as the proof of (\ref{A11}).

In conclusion, $X$ is a solution to (\ref{sk2}).
The uniqueness of solutions to (\ref{sk2}) is a direct
consequence of condition $({\mathbf{A}}{\mathbf{2}})$. Thus (\ref{e11}) holds. Moreover, the estimate (\ref{seletonestimate}) follows from (\ref{se1}) for any $\phi \in S_M$. We complete the proof.
\end{proof}

Finally, we present a time increment estimate of solutions to the skeleton equation. Because of very similar argument as the proof of Lemma \ref{condiff}, we omit the details here.
\begin{lemma}\label{MonoSkePrioLemm}
For any $T>0,\phi\in S_M,M<\infty$, $x\in \mathbb{H}_1$, there is a constant $C_{T,M}>0$ such that for any $\zeta>0$ small enough we have
\begin{equation}%\label{skeletonDifferent}
\sup_{\phi\in S_M}\int_0^T\|\bar{X}^\phi_t-\bar{X}^\phi_{t(\zeta)}\|_{\mathbb{H}_1}^2dt\leq C_{x,T,M}\zeta.
\end{equation}
\end{lemma}

\section*{Statements and Declarations}
\noindent\textbf{Data availability:} Data sharing not applicable to this article as no datasets were generated or analysed during the current study.

\noindent\textbf{Conflict of interest:}  On behalf of all authors, the corresponding author states that there is no conflict of interest.

\end{document}